\documentclass{jmi0} 
\usepackage{color}
\usepackage{subfigure,txfonts,multirow}
\usepackage[pdftex]{graphicx}
\title{Bipartition of graphs based on the\\ normalized cut and spectral methods}
\author{KR}{K.K.K.R.Perera}{kkkrperera@kln.ac.lk}
\author{YM}{Yoshihiro Mizoguchi}{ym@imi.kyushu-u.ac.jp}
\affiliation{KR}{Department of Mathematics, University of Kelaniya, Sri Lanka}
\affiliation{YM}{Institute of Mathematics for Industry, Kyushu University, Japan}%

\theoremstyle{ap-thm}

\begin{document}
\abstract{%
In the first part of this paper,
we survey results that are associated with three types of Laplacian matrices:difference, normalized, and signless.
We derive eigenvalue and eigenvector formulaes for paths and cycles using circulant matrices and present an alternative proof for finding eigenvalues of the adjacency matrix of paths and cycles using Chebyshev polynomials.
Even though each results is separately well known,
we unite them, and 
provide uniform proofs in a simple manner.
The main objective of this study is to solve the problem of finding graphs, on which spectral clustering methods and normalized cuts produce different partitions.
First,
we derive a formula for a minimum normalized cut for graph classes such as paths, cycles, complete graphs, double-trees,
cycle cross paths, and some complex graphs like lollipop graph $LP_{n,m}$, roach type graph $R_{n,k}$, and weighted path $P_{n,k}$.
Next,
we provide characteristic polynomials of the normalized Laplacian matrices ${\mathcal L}(P_{n,k})$ and ${\mathcal L}(R_{n,k})$.
Then,
we present counter example graphs based on $R_{n,k}$,
on which spectral methods and normalized cuts produce different clusters.

}
\keywords{%
 spectral clustering, normalized Laplacian matrices, difference Laplacian matrices, signless Laplacian matrices, normalized cut
}

\maketitle

\section{Introduction}
Clustering techniques are common in multivariate data analysis, data mining, machine learning, and so on.
The goal of the clustering or partitioning problem is to find groups such that entities within the same group are similar and different groups are dissimilar.
In the graph-partitioning problem, much attention is given to find the precise criteria to obtain a good partition.
Clustering methods
that use eigenvalues and eigenvectors of matrices associated with graphs are called spectral clustering methods and are widely used in graph- partitioning problems.
In particular,
eigenvalues and eigenvectors of Laplacian matrices play a vital role in graph-partitioning problems.
In 1973,
Fiedler defined the second smallest eigenvalue $\lambda_2$ of a difference Laplacian matrix as the
algebraic connectivity of a graph \cite{fielder:1973}. 
In 1975,
he showed that we  can decompose a graph $G$ into two connected
components by only using the sign structure of an eigenvector related to
the second smallest eigenvalue \cite{fiedler:1975}.
In 2001, Fiedler's investigation was extended by Davies 
using the discrete nodal domain theorem \cite{Davis:2001}.
Laplacian, normalized Laplacian, and adjacency matrices with negative entries can be used with the nodal domain theorem. 
This theorem is useful to identify the number of connected sign graphs of a given graph on the basis of their eigenvectors and eigenvalues.

In 1984,
Buser \cite{buser:1984} investigated the graph invariant quantity $\displaystyle i(G)=\min_U \frac{|\partial U|}{|U|}$,
which considers the relationship between size of a cut and the size of a separated subset $U$.
He defined the isoperimetric number $i(G)$,
and the optimal bisection was given by the minimum $i(G)$.
Guattery and Miller \cite{step:1995,step:1998} considered two spectral separation algorithms that partition the vertices on the basis of the values of their corresponding entries in the second eigenvector and, in 1995, they provided some counter examples for which each of these algorithms produce poor separators.
They used an eigenvector based on the second smallest eigenvalue of
a difference Laplacian matrix as well as a specified number of eigenvectors corresponding to the smallest eigenvalues.
Finally,
they extended it to the generalized version of spectral methods
that allows for the use of more than a constant number of eigenvectors and showed
that there are some graphs for which the performance of all the above spectral
algorithms was poor. 
We follow their methods especially in the cases of graph automorphism
and even -odd eigenvector theorem for the concrete classes of graphs
such as roach graphs, double-trees, and double-tree cross paths.
We prefer to use a normalized Laplacian matrix rather than a difference Laplacian matrix,
and describe these properties in terms of formal graph notation.

In 1997,
Fan Chung \cite{Fan:1997} discussed the most important theories and properties regarding eigenvalues of normalized Laplacian matrices and their applications to graph separator problems.
She considered the partitioning problem using Cheeger constants and derived fundamental relations between the eigenvalues and Cheeger constants.
In 2000,
Shi and Malik \cite{shi:2000} proposed a measure of disassociation, called normalized cut, for the image segmentations.
This measure computed the cut cost as a fraction of total edge connections. 
The normalized cut is used to minimize the disassociation between groups and maximize the association within groups. 
However,
minimization of normalized cut criteria is an non-deterministic polynomial-time hard (NP- hard) problem.  
Therefore,
approximate discrete solutions are required.
The solution to the minimization problem of the normalized cut is given by the second smallest eigenvector of the generalized eigensystem,
$(D-W)y=\lambda D y$,
where $D$ is the diagonal matrix with vertex degrees and $W$ is a weighted adjacency matrix.
Shi and Malik used a minimum normalized cut value as a splitting point and found a bisection using the second smallest eigenvector.
They realized that the eigenvectors are well separated and that this type of splitting point is very reliable.
The normalized cut introduced by Shi and Malik \cite{shi:2000} is useful in several areas.
This measure is of interest not only for image segmentation but also for network theories and statistics \cite{Andrew:2001,sara:2008,sara1:2008,du:2004,sound:2003}.

In this study,
we review the known results regarding the difference, normalized, and signless Laplacian matrices.
Then,
we give uniform proofs for the eigenvalues and eigenvectors of paths and cycles.
Next,
we analyze the minimum normalized cut from the view point of connectivity of graphs and compare the results with those of the spectral bisection method.
Special emphasis is given to classify the graphs,
that poorly  perform on spectral bisections using normalized Laplacian matrices.
We use the term $Mcut(G)$ to represent the minimum normalized cut and $Lcut(G)$ to represent the normalized cut of the bipartition created by the second smallest eigenvector of the normalized Laplacian based on the sign pattern.
Finding $Mcut(G)$ for a graph is NP-hard.
However,
we derive a formula for $Mcut(G)$ for some basic classes of graphs such as paths, cycles, complete graphs, double-trees, cycle cross paths, and some complex graphs like lollipop type graphs $LP_{n,m}$, roach type graphs $R_{n,k}$ and weighted paths $P_{n,k}$.
Next, we present characteristic polynomials of the normalized
Laplacian matrices ${\mathcal L}(P_{n,k})$ and ${\mathcal L}(R_{n,k})$.
We provide counter example graphs on the basis of a graph $R_{n,k}$ 
on which $Mcut(G)$ and $Lcut(G)$ have different values.

This paper is organized as follows. 
In section 2,
we present basic terminologies and key results related to the difference,
normalized, and signless Laplacian matrices.
In particular,
we summarize the upper and lower bounds of the second smallest eigenvalues.
We also define graphs that are used in other sections using formal notation. 
In section 3,
we review the properties of the $Mcut(G)$ of graphs and derive formulae for the $Mcut(G)$ of some basic classes of graphs and some complex graphs such as $R_{n,k}$, $P_{n,k}$, and $LP_{n,m}$.
In section 4,
we consider the eigenvalues and eigenvectors of paths and cycles for the three types of Laplacian matrices introduced above.
In particular,
we review the eigenvalue formulae for the three types of Laplacian matrices using circulant matrices and then review an alternative proof for the eigenvalues of adjacency matrices of paths and cycles using Chebyshev polynomials.
We also give concrete formulae for the characteristic polynomials of the normalized Laplacian matrices ${\mathcal L}(P_{n,k})$ and ${\mathcal L}(R_{n,k})$.
In section 5,
we provide counter example graphs for which spectral techniques perform poorly compared with the normalized cut.
Specifically, 
we find the conditions for which $Mcut(G)$ and $Lcut(G)$ have different values on the $R_{n,k}$ graph.

\section{Preliminaries}
An undirected graph is an ordered pair $G=(V(G),E(G))$,
where $V(G)$ is a finite set,
elements of which are
called vertices,
and we represent $V(G)$ as $V(G)=\{v_1,v_2,\ldots,v_n\}$.
$E(G)$ is a set of two-element subsets of $V(G)$,
called edges.
Conventionally,
we denote an edge $\{v_i,v_j\}$ by $(v_i,v_j)$ in this paper.
Two vertices $v_i$ and $v_j$ of $G$ are called adjacent,
if $(v_i,v_j) \in E(G)$. 
For simplicity,
sometimes we use $V$ instead of $V(G)$ and $E$ instead of $E(G)$.
The number of vertices in $G$ is the order of $G$ and the number of edges is the size of $G$.
For a given subset $S\subseteq  V $,
$|S|$ represent the size of the set $S$.
For a subset $A \subseteq V$,
we represent the set of vertices not belongs to $A$ as $V\setminus A= \{ v_i \ | \ v_i \notin A \}$. 
A graph of order 1 is called a trivial graph.
A graph which has two or more vertices is called a nontrivial graph.
A graph of size 0 is called an empty graph.
Assume that all graphs in this paper are finite,
undirected and have edge weight 1.

\begin{definition}[Adjacency matrix]
Let $G=(V,E)$ be a graph and $|V|=n$.
The adjacency matrix $A(G)=(a_{ij})$ of an undirected graph $G$ is a $n \times n$ matrix whose entries are given by
\[
 a_{ij}=
\left \{ \begin{array}{ll} 1 & \mbox{ if $(v_i,v_j) \in E$,}\\
0 & \mbox{otherwise.}
\end{array} \right. \] 
\end{definition}
\begin{definition}[Degree]
The degree $d_i$ of a vertex $v_i$ of a graph $G$ is defined as $\displaystyle d_i= \sum_{j=1}^n a_{ij}$. 
Minimum and maximum degree of a graph $G$ are denoted by $\delta(G)$ and $\bigtriangleup(G) $,
respectively.
\end{definition}

\begin{definition}[Degree Matrix]
The diagonal matrix of a graph $G$ is denoted by $D(G)=diag(d_1,d_2,\ldots,d_n)$,
where $d_i$ is the degree of a vertex $v_i$.

\end{definition}
Note: For simplicity,
 sometimes we use $D$ instead of $D(G)$.
\begin{definition}[Volume]
The volume of a graph $G=(V,E)$ denoted by $\displaystyle vol(G)=\sum_{i =1}^{|V|} d_i$,
 is the sum of the degrees of vertices in $V$.
The volume of a subset $A\subset V$ is denoted by $\displaystyle vol(A)=\sum_{i \in A} d_i$.
\end{definition}

\begin{definition}[Edge Connectivity]
The edge connectivity of a graph $G$ is denoted by $\kappa'(G)$, is the minimum number of edges needed to remove in order to disconnect the graph.
A graph is called $k$-edge connected if every disconnecting set has at
 least $k$ edges. A $1$-edge connected graph is called a connected graph.
\end{definition}
\begin{definition}[Cartesian product]
The Cartesian product of graphs $G$ and $H$ is denoted by $G \Box H =( V(G \Box H),E(G \Box H))$,
where $\displaystyle V(G \Box H)= V(G) \times V(H)$ and $\displaystyle
 E(G \Box H)= \{(u_1,v_1),(u_2,v_2) \ $
$| \ u_1=u_2 \ and \ (v_1,v_2)  \in E(H) $
$\ or \ v_1=v_2 \ and \ (u_1,u_2) \in E(G) \}$.
\end{definition}

We note that 
$\displaystyle G_1 \Box G_2 \cong G_2 \Box G_1 $,
$\displaystyle \delta (G_1 \Box G_2) =\delta(G_1) + \delta(G_2)$,
and
$\displaystyle \kappa'(G \Box H) =\min \{ \kappa'(G)|V(H)|,
\kappa'(H)|V(G)|, \delta(G)+\delta(H) \}$. 


\begin{definition}[Path]
Let $G=(V,E)$ be a graph. 
A path in a graph is a sequence of vertices such that from each of its vertices there is an edge to the next vertex in the sequence.
This is denoted by $P=(u=v_0,v_1,\ldots,v_k=v)$,
where $(v_i,v_{i+1}) \in E$ for $0 \le i \le k-1$.
The length of the path is the number of edges encountered in $P$.
\end{definition}
\begin{definition}[Shortest Path]
Let $G=(V,E,w)$ be a weighted graph.
Let $\it{P}$ be a set of paths from vertex $i$ to $j$.
Denote $\ell (p)$, the length of the path $p \in \it{P}$.
Then $p$ is a shortest path if $\ell (p)=\min_{p' \in \it{P}} \ell (p')$.
\end{definition}
\begin{definition}[Distance]
The distance between two vertices $i,j \in V$ of the graph $G$ is denoted by $dist(i,j)$ is the length of a shortest path between vertex $i$ and $j$.
\end{definition}
\begin{definition}[Diameter]
The diameter of a graph $G=(V,E)$ is given by $\displaystyle diam(G)=\max \{dist(i,j) \ | \ i,j \in V\}$.
\end{definition}

\begin{definition}[Permutation matrix]
Let $G=(V,E)$ be a graph.
The permutation $\phi$ defined on $V$ can be represented by a permutation matrix $\displaystyle P=(p_{ij})$,
where 
 \[ p_{ij}= \left \{ \begin{array}{cc}
  1 & \mbox{if $v_i=\phi(v_j)$,}\\
  0 & \mbox{otherwise. }
\end{array} \right. \]
\end{definition}

\begin{definition}[Automorphism]
Let $G=(V,E)$ be a graph.
Then a bijection $\phi :V\rightarrow V$ is an automorphism of $G$ if $(v_i, v_j) \in E$ then $(\phi(v_i),\phi(v_j)) \in E$.
In other words automorphisms of $G$ are the permutations of vertex set $V$ that maps edges onto edges.
\end{definition}

\begin{proposition}[Biggs \cite{bigg:1993}]
Let $A(G)$ be the adjacency matrix of a graph $G=(V,E)$,
and $P$ be the permutation matrix of permutation $\phi$ defined on $V$.
Then $\phi$ is an automorphism of $G$ if and only if $PA=AP$.
\end{proposition}\hfill\qed

\begin{definition}[Weighted graph]
A weighted graph is denoted by $G=(V,E,w)$,
where $w: E\rightarrow \Re$.
\end{definition}
\begin{definition}[Weighted adjacency matrix]
The weighted adjacency matrix $W=(w_{ij})$ is defined as
\[
 w_{ij}=
\left \{ \begin{array}{ll} w(i,j) & \mbox{if $(i,j)\in E$,} \\
0 & \mbox{otherwise.}
\end{array} \right. \] 
\end{definition}
The degree $d_i$ of a vertex $v_i$ of a weighted graph is defined by $\displaystyle d_i =\sum_{j=1}^{n}w_{ij}$.
Unweighted graphs are special cases,
where all edge weights are 0 or 1.
\begin{definition}[Graph cut]
A subset of edges which disconnects the graph is called a graph cut.
Let $G=(V,E,w)$ be a weighted graph and $W=(w_{ij})$ the weighted adjacency matrix.
Then for $A,B \subset  V$ and $A \cap B = \emptyset $, 
the graph cut is denoted by $\displaystyle cut(A,B)=\sum_{i \in A, j \in B} w_{ij}$.
\end{definition}
\begin{definition}[Isoperimetric number]
The isoperimetric number $i(G)$ of a graph $G$ of order $n \geq 2$ is defined as
$$
 i(G) = \min  \{ \frac{cut(S,V\setminus S)}{|S|},S\subset V, 0 < |S| \leq  \frac{n}{2} \}.
$$ 
\end{definition}
\begin{definition}[Cheeger Constant-edge expansion]
Let $G=(V,E)$ be a graph. 
For a nonempty subset $S\subset V$,
define\\
$\displaystyle h_G(S)=\frac{cut(S,V\setminus S)}{\min(vol(S),vol(V\setminus S))}$.
The Cheeger constant(edge expansion) $h_G$ is defined as 
$\displaystyle h_G=\min_S h_G(S)$.
\end{definition}

\begin{definition}[Cheeger constant-vertex expansion]
Let $G=(V,E)$ be a graph. For a nonempty subset $S\subset V$,
define\\
$\displaystyle g_G(S)=\frac{vol(\delta S)}{\min(vol( S),vol(V\setminus S))}$,
where $\displaystyle \delta S=\{v \notin S :(u,v)\in E, u \in  S\}$.
Then the Cheeger constant(vertex expansion) $g_G$ is defined as
$\displaystyle g_G=\min_S g_G(S)$. 
\end{definition}

\begin{definition}[Weighted difference Laplacian]
The {\bf weighted difference Laplacian} $L(G)=(l_{ij})$ is defined as 
  \[
 l_{ij}=
\left \{ \begin{array}{lll} d_i-w_{ii} & \mbox{if  $v_i=v_j$,} \\
-w_{ij} & \mbox{if $v_i$ and $v_j$ are adjacent and $v_i \neq v_j$,} \\
0 & \mbox{otherwise.}
\end{array} \right. \]
This can be written as $L(G)=D(G)-W(G)$.
\end{definition}

\begin{definition}[Weighted normalized Laplacian]
The {\bf weighted  normalized Laplacian} $\mathcal{L}(G)=(\ell_{ij})$ is defined as
  \[
 \ell_{ij}=
\left \{ \begin{array}{lll} 1-\frac{w_{jj}}{d_j} & \mbox{if  $v_i=v_j$,} \\
-\frac{w_{ij}}{\sqrt{d_i d_j}} & \mbox{ if $v_i$ and $v_j$ are adjacent and $v_i \neq v_j$,} \\
0 & \mbox{otherwise.}
\end{array} \right. \]
\end{definition}

\begin{lemma}
Let $G$ be a graph,
$n$ the size of graph $G$,
$A=(w_{ij})$ a weighted adjacency matrix of $G$,
$\lambda$ an eigenvalue of ${\mathcal L}(G)$
and $x$ an eigenvector corresponding to $\lambda$
with $x^Tx=1$. Then,
$$
\lambda = 
\frac{1}{2}\sum_{i=1}^{n}\sum_{j=1}^n
\left(\frac{x_i}{\sqrt{d_i}}-\frac{x_j}{\sqrt{d_j}}\right)^2 w_{ij}.
$$
\end{lemma}
\begin{proof}
Let $D$ be the degree matrix of $G$.
The normalized Laplacian matrix ${\mathcal L}(G)$ is defined by
$D^{-\frac{1}{2}}(D-A)D^{-\frac{1}{2}}$.
Let $y$ be a vector with size $n$ and $x=D^{\frac{1}{2}} y$.
Then
$x^T{\mathcal L}(G) x$
$\displaystyle = \left(D^{\frac{1}{2}}y\right)^T{\mathcal L}(G)
\left(D^{\frac{1}{2}}y\right)$
$\displaystyle =y^TD^{\frac{1}{2}}{\mathcal L}(G)D^{\frac{1}{2}}y$
$=y^T(D-A)y$
$\displaystyle =\sum_{i=1}^n y_i^2d_i-\sum_{i=1}^n\sum_{j=1}^n y_iy_jw_{ij}$
$\displaystyle =\frac{1}{2}\sum_{i=1}^{n}\sum_{j=1}^n
\left(y_i-y_j\right)^2 w_{ij}$
$\displaystyle =\frac{1}{2}\sum_{i=1}^{n}\sum_{j=1}^n
\left(\frac{x_i}{\sqrt{d_i}}-\frac{x_j}{\sqrt{d_j}}\right)^2 w_{ij}$.
Since $x$ is an 
an eigenvector of ${\mathcal L}(G)$ corresponding to $\lambda$
and $x^T x=1$,
we have
$\displaystyle \lambda =\frac{x^T(\lambda x)}{x^Tx}$
$\displaystyle =\frac{x^T({\mathcal L}(G)x)}{x^Tx}$
$\displaystyle =\frac{1}{2}
\sum_{i=1}^n \sum_{j=1}^n
\left(\frac{x_i}{\sqrt{d_i}}-\frac{x_j}{\sqrt{d_j}}\right)^2w_{ij}$.
\end{proof}


There are several properties about
bounds of the second eigenvalue $\lambda_2$.

\begin{proposition}[Mohar\cite{mohar:1989}]
Let $G=(V,E)$ be a graph and
$\lambda_2$ be the second smallest eigenvalue of $L(G)$. Then,
$$
\frac{\lambda_2}{2} \leq
i(G) 
\leq \sqrt{(2\triangle(G) -\lambda_2)\lambda_2}.
$$
\end{proposition}\hfill\qed

\begin{proposition}[Chung\cite{Fan:1997}]
Let $G$ be a connected graph and
$h_G$ the Cheeger constant of $G$. Then,
\begin{enumerate}
 \item $\displaystyle \frac{2}{vol(G)} < h_G$,
 \item $\displaystyle 1- \sqrt{1-h^2_G} < \lambda_2 $, and
 \item $\displaystyle  \frac{h^2_G}{2} < \lambda_2 \leq  2h_G$.
\end{enumerate}
\end{proposition}\hfill\qed


 
\begin{definition}[Signless Laplacian]
The {\bf weighted signless Laplacian} $SL(G)=(sl_{ij})$ is defined as 
  \[
 sl_{ij}=
\left \{ \begin{array}{lll} d_i+w_{ii} & \mbox{if  $v_i=v_j$,} \\
w_{ij} & \mbox{if $v_i$ and $v_j$ are adjacent and $v_i \neq v_j$,} \\
0 & \mbox{otherwise.}
\end{array} \right. \]
This can be written as $SL(G)=D+W$.
\end{definition}

\begin{definition}[Path graph]
A path graph $P_n=(V_n,E_n)$ consists of a vertex set
$\displaystyle V_n=\{v_l \ | \ l \in \mathbb{Z^+}, l \leq n \}$ and an edge set $\displaystyle E_n= \{(v_l,v_{l+1}) \ | \  1 \leq l <n \}$.
\end{definition}

\begin{example}
The Table~\ref{tab:matrices} shows an adjacency matrix and the three Laplacian matrices discussed in the above for path graph $P_4$.
\begin{table}
{\small
\begin{tabular}{|l|l|}
\hline
Matrix $M$ & $M(P_4)$  \\
\hline
\begin{tabular}{l}
Adjacency \\
$A(P_4)$
\end{tabular}
 & $ \left(
\begin{array}{cccc}
 0 & 1 & 0 & 0 \\
 1 & 0 & 1 & 0 \\
 0 & 1 & 0 & 1 \\
 0 & 0 & 1 & 0
\end{array}
\right) $ \\
\hline
\begin{tabular}{l}
Difference Laplacian \\
$L(P_4)$
\end{tabular}
& $\left(
\begin{array}{cccc}
 1 & -1 & 0 & 0 \\
 -1 & 2 & -1 & 0 \\
 0 & -1 & 2 & -1 \\
 0 & 0 & -1 & 1
\end{array}
\right) $ \\
\hline
\begin{tabular}{l}
Normalized Laplacian \\
$\mathcal{L}(P_4)$
\end{tabular}
& $\left(
\begin{array}{cccc}
 1 & -\frac{1}{\sqrt{2}} & 0 & 0 \\
 -\frac{1}{\sqrt{2}} & 1 & -\frac{1}{2} & 0 \\
 0 & -\frac{1}{2} & 1 & -\frac{1}{\sqrt{2}} \\
 0 & 0 & -\frac{1}{\sqrt{2}} & 1
\end{array}
\right) $\\
\hline
\begin{tabular}{l}
Signless Laplacian \\
$SL(P_4)$
\end{tabular}
& $\left(
\begin{array}{cccc}
 1 & 1 & 0 & 0 \\
 1 & 2 & 1 & 0 \\
 0 & 1 & 2 & 1 \\
 0 & 0 & 1 & 1
\end{array}
\right)$\\
\hline
\end{tabular}
\caption{Matrices associated with graphs.}
\label{tab:matrices}
}
\end{table}
\end{example}

\begin{lemma}
\label{lema1}
Let $G=(V,E,w)$ be a weighted graph.
Then the eigenvalues of $\mathcal{L}(G)$ and $D^{-1}L(G)$ are equal.
\end{lemma}
\begin{proof}
$\displaystyle D^{-1}L= D^{-1}(D-W)=I-D^{-1}W= D^{-1/2}DD^{-1/2}- D^{-1/2}D^{-1/2}W=D^{-1/2}(D-W)D^{-1/2}$.
Therefore $D^{-1}L(G) = \mathcal{L}(G)$ and has the same spectrum.\end{proof}
\begin{definition}[Regular graph]
A graph $G=(V,E)$ is called $r$-regular, if $d_i=r \ ( i=1,\ldots,|V|)$. 
\end{definition}

\begin{lemma}
Let $\mu_i,(i=1, \ldots, n) $ be eigenvalues of difference Laplacian matrix $L(G)=D(G)-A(G)$.
Then for any regular graph of degree $r$,
normalized Laplacian eigenvalues are $\displaystyle \lambda_i= \frac{\mu_i}{r},(i=1, \ldots, n)$. 
\end{lemma}
\begin{proof}
$\displaystyle L=(D-A)=rI-A$. 
Then $\displaystyle \mathcal{L}(G)= D^{-1/2}LD^{-1/2}=\frac{I}{r^{1/2}}(rI-A)\frac{I}{r^{1/2}}=I-\frac{A}{r}$.
Then $\displaystyle r \mathcal{L}(G)=L(G)$.
If $\mu_i$ is an eigenvalue of $L$ then it is an eigenvalue of $r\mathcal{L}(G)$.
This shows that $\displaystyle \lambda(\mathcal{L}(G))= \frac{\mu_i}{r}(i=1,\ldots, n)$.\end{proof}

\begin{proposition}
Let $\mathcal{L}(G)$ be the normalized Laplacian matrix of a graph $G$ and $P$ be the permutation matrix corresponding to the automorphism $\phi$ defined on $V$.
If $U$ is an eigenvector of $\mathcal{L}(G)$ with an eigenvalue $\lambda$,
then $PU$ is also an eigenvector of $\mathcal{L}(G)$ with the same eigenvalue.
\label{prop1}
\end{proposition}
\begin{proof}
From the definition of automorphism $P^T\mathcal{L}(G)P=\mathcal{L}(G)$.
Then $\mathcal{L}(G)U=\lambda U$ implies that $(P^T\mathcal{L}(G)P)U=\lambda U$.
Since $PP^T=I$,
 we get $ \mathcal{L}(G)PU=\lambda (PU)$.
If $U$ is an eigenvector of $\mathcal{L}(G)$ with an eigenvalue $\lambda$ then $PU$ is also an eigenvector with the same eigenvalue.\end{proof}
{\bf{Remarks.}}This result holds for any matrix associated with a graph under the automorphism defined on a vertex set.
\begin{definition}[Odd-even vectors]
Let $G=(V,E)$ be a graph and $\phi:V\rightarrow V$ be an automorphism of order 2.
A vector $x$ is called an even vector if $x_i =x_{\phi(i)}$ for all $ 1\leq i
 \leq n$ and a vector $y$ is called an odd vector  if $y_i =-y_{\phi(i)}$ for all $ 1\leq i
 \leq n$, 
where $n=|V|$.
\end{definition}

\begin{proposition}
\label{prop2}
Let $G$ be a graph,
$\phi$ be an automorphism of $G$ with order 2 and $P$ a permutation matrix of $\phi$. 
If an eigenvalue of $\mathcal{L}(G)$ is simple then the corresponding eigenvector is odd or even with respect to $\phi $. 
\end{proposition}
\begin{proof}
Let $\lambda$ be an eigenvalue,
$U$ an eigenvector of $\mathcal{L}(G)$.
If $\lambda$ is simple then $ PU$ and $U$ are linearly dependent.
Then there exists a constant $c$ such that $PU=cU$.
Since $P^2 =I$ for an automorphism of order 2,
$\displaystyle IU =c PU=c^2U$ and  $c= \pm 1$.
Then $PU=U$ or $PU=-U$. 
Hence an eigenvector $U$ is odd or even with respect to $\phi $.
\end{proof}
\begin{definition}
Let $G=(V,E)$ be a graph,
$V=\{v_i \ |\ 1\le i \le n\}$ $(n=|V|)$
and $U=(u_1,u_2,\ldots,u_n) \in \Re^n$ a vector.
We define three subsets of $V$ as follows:
\begin{eqnarray*}
V^+(U) &=& \{ v_i \in V \ | \ u_i >0 \},\\
V^-(U) &=& \{ v_i \in V \ | \ u_i <0 \}, \ and \\
V^0(U) &=& \{ v_i \in V \ | \ u_i =0 \}.
\end{eqnarray*} 
\end{definition}
\begin{lemma}
\label{lemma4}
Let $\mathcal{L}(G)$ be the normalized Laplacian of graph $G$ and $U=(u_i),(i=1,\ldots,n)$ the second eigenvector.
If $U  \neq \mathbf{0}$ then $V^+(U) \neq \emptyset$ and $V^-(U)\neq \emptyset $.
\end{lemma}
\begin{proof}
The vector $D^{1/2}\vec{1}$ is an eigenvector corresponding to the zero eigenvalue.
Since the second eigenvector $U$ is orthogonal to $D^{\frac{1}{2}}\vec{1}$,
$(D^{\frac{1}{2}}\vec{1})^TU=0$ and $\sum_i \sqrt{d_i}u_i =0$.
Since $d_i >0, U \neq \mathbf{0}$,
there exist at least two values such that $u_i>0 $ and $u_j <0$ for $i \neq j$.
Hence $V^+(U) \neq \emptyset$ and $V^-(U)\neq \emptyset $.\end{proof}

\begin{lemma}
\label{lemma5}
Let $G$ be a graph with an automorphism $\phi$ of order 2.
Let $U= (u_1,u_2,\ldots,u_n)$ be an eigenvector and  $\displaystyle \phi (U)= (u_{\phi(1)},u_{\phi(2)},\ldots,u_{\phi(n)})$.
If $U  \neq \mathbf{0}$ and $\phi(U)= -U$ then $V^+(U) \neq \emptyset $ and $V^-(U)\neq \emptyset$.
\end{lemma}
\begin{proof}
Assume $V^+(U)=\emptyset$.
If $u_i \leq 0,(i=1,\ldots,n)$,
$\phi(U)= -U$ implies that $u_{\phi(i)} >0$.
This contradicts that $V^+(U)=\emptyset $.
Similarly,
if we assume that $V^-(U)=\emptyset $ and $u_i \geq 0$ for $(i=1,\ldots,n)$, then
$\phi(U)= -U$ implies that $u_{\phi(i)} <0$.
Then this contradicts that $V^-(U)=\emptyset $.
If $u_i =0,(i=1,\ldots,n)$,
then $U = \mathbf{0}$ and contradicts that $U \neq \mathbf{0}$.
\end{proof}
\begin{proposition}[Guattery et al.\cite {step:1995}]
\label{prop5}
Let $P_n$ be a weighted path graph and $\mathcal{L}(P_n)$ be its normalized Laplacian matrix.
For any eigenvector $X=(x_1,x_2,\ldots,x_n)$,
\begin{enumerate}
\item $x_1=0$ implies $X=0$,
\item $x_n=0$ implies $X=0$ and,
\item $x_i=x_{i+1}=0$ implies that $X=0$.
\end{enumerate}
\end{proposition}\hfill\qed

\begin{lemma}[Guattery et al.\cite{step:1995}]
\label{lemasimple}
For a path graph $P_n$,
$\mathcal{L}(P_n)$ has $n$ simple eigenvalues.
\end{lemma}
\begin{proof}
Let $U=(u_1,u_2,\ldots,u_n)$ and $\bar{U}=(\bar{u}_1,\bar{u}_2,\ldots,\bar{u}_n)$ be two eigenvectors of $\mathcal{L}(P_n)$ with eigenvalue $\lambda$.
From the proposition~\ref{prop5},
we have $u_n \neq 0$ and $\bar{u}_n \neq 0$.
Let $\displaystyle \alpha =\frac{\bar{u_n}}{u_n}$, 
where $ \alpha \neq 0$.
Consider $\displaystyle \mathcal{L}(P_{n})(\alpha U -\bar{U})= \lambda ( \alpha U -\bar{U})$.
The $n$-th element of $\displaystyle (\alpha U -\bar{U})$ is $(\bar{u}_n u_n-\bar{u}_nu_n)=0$.
Then $\displaystyle \alpha U =\bar{U}$.
Thus $U$ and $\bar{U}$ are linearly dependent and hence $\lambda$ is simple.\end{proof}

\begin{proposition}
\label{prop3}
Let $P_n$ be the path graph and $\phi$ the automorphism of order 2
defined on $V(P_n)$.
Then any second eigenvector $U_2$ of $\mathcal{L}(P_n)$ is an odd vector.
\end{proposition}\hfill\qed

\begin{example}
Let \[ M= \left( \begin{array}{cccccc}
1 & -1 & 0 & 0 &0 & 0\\
-1 & 2 & -1 & 0& 0&  0\\
0 & -1 & 2 & -1 & 0& 0\\
0& 0&-1& 2&-1&   0\\
0 &0&0&-1 &2&-1 \\
0 & 0& 0& 0&  -1 &  1
\end{array} \right) \] and \[ P= \left( \begin{array}{cccccc}
0 & 0 & 0 & 0 &0 & 1\\
0 & 0 & 0 & 0& 1 & 0\\
0 & 0 & 0 & 1 & 0 & 0\\
0& 0&1& 0& 0& 0\\
0 &1 &0&0& 0& 0\\
1 & 0& 0& 0&  0&  0
\end{array} \right). \] 
If $U_M$ is a second eigenvector of $M$ then by the Proposition~\ref{prop1},
$PU_M$ is also a second eigenvector.
By the Proposition~\ref{prop2},
$PU_M= U_M$ or $PU_M=-U_M$.
By the Proposition~\ref{prop3},
$U_M$ is an odd vector and $PU_M=-U_M$.\end{example}

\begin{definition}[Weighted Path]
For $n$ ($n \ge 1$) and $k$ ($k \ge 1$),
the adjacency matrix $(P_{ij})$
of a weighted path $P_{n,k}=(V,E)$ is 
the $(n+k) \times (n+k)$ matrix such that 
\[ P_{ij} = \left \{ \begin{array}{cc}
0 & \mbox{($ i=j$ and $i \leq n$) or ($i \neq j+1$ and $j \neq i+1$),}\\
1 & \mbox{($i=j$ and $n+1 \leq i$) or ($i = j+1$ or $j = i+1$).}
\end{array} \right. \]
That is
$V=\{x_1,x_2,\ldots, x_n, x_{n+1}, \ldots, x_{n+k}\}$
and
$E=\{(x_i,x_j)\ |\ P_{ij}=1, 1 \le i,j \le n+k \}$.

\end{definition}

Let $\Sigma$ be an alphabet and
$\Sigma^*$ a set of strings over $\Sigma$
including the empty string $\epsilon $.
We denote the length of $w \in \Sigma^*$ by $|w|$.
Let $\Sigma^{< n}=\{w \in \Sigma^* | |w| < n\}$
and $\displaystyle \Sigma_1^{< n}=\{w \in \Sigma^*  | 1  \le |w| < n \}$.
In this paper,
we assume $\displaystyle \Sigma=\{0,1\}$. 

\begin{definition}[Complete binary tree]
A complete binary tree $T_n=(V,E)$ of depth $n$ is defined as follows.
\begin{eqnarray*}
V &=&\Sigma^{< n},\\
E &=& \{ (w,wu) \ | \ w \in \Sigma^{< (n-1)}, u \in \Sigma \}.
\end{eqnarray*}
\end{definition}
\begin{definition}[Double tree]
A double tree $DT_n=(V,E)$,
where $n$ is the depth of the tree,
 consists of two complete binary trees connected by their root.
We define double tree as follows.
\begin{eqnarray*}
V&=&\{x(w) \ | \ w\in \Sigma^{< n}\} \cup \{y(w) \ | \ w\in \Sigma^{< n}\},\\
E_1&=&\{(x(w),x(wu))\ | \ w \in \Sigma^{< (n-1)}, u \in \Sigma\}, \\
E_2&=&\{(y(w),y(wu))\ | \ w \in \Sigma^{< (n-1)}, u \in \Sigma\}, \\
E&=& E_1 \cup E_2 \cup \{(x(\epsilon ),y(\epsilon ))\}.
\end{eqnarray*}
\end{definition} 
\begin{definition}[Cycle]
A cycle $C_n=(V_n,E_n)$ consists of a vertex set $\displaystyle V_n=\{ v_l \ | \ l \in \mathbb{Z^+}, l \leq n \}$ and an edge set $\displaystyle E_n= \{(v_l,v_{l+1})\ | \ 1 \leq l < n \} \cup  \{(v_1,v_n)\}$.
\end{definition}

\begin{definition}[Complete graph]
A complete graph $K_n=(V_n,E_n)$ consists of a vertex set $\displaystyle V_n =\{ v_i \ | \ 1 \leq i \leq n \}$ and an edge set $\displaystyle E_n = \{ (v_i,v_j) \ | \ i \neq j \ and \ 1 \leq i \leq n,1 \leq j \leq n \} $.
\end{definition}

\begin{definition}[Graph $R_{n,k}$]
The graph $R_{n,k}(n \ge 1, k \ge 2)$ is a bounded degree planer graph with a vertex set $V=V_1 \cup V_2 $ and an edge set $E=E_1 \cup E_2 \cup E_3$.
\begin{eqnarray*}
V_1 &=& \{ x_i \ | \ 1 \leq i \leq n+k  \},\\
V_2 &=& \{ y_i \ | \ 1 \leq i \leq n+k  \}, \\
E_1 &=& \{ ( x_i,x_{i+1}) \ | \ 1 \leq i \leq n+k-1  \}, \\
E_2 &=& \{ ( y_i,y_{i+1}) \ | \ n+k+1  \leq i \leq 2(n+k)-1  \}, \\
E_3 &=& \{ ( x_i,y_{i}) \ | \ n+1 \leq i \leq n+k  \}. 
\end{eqnarray*}
\end{definition}

\begin{definition}[Cycle cross paths $C_m \Box P_n$]
Let $C_m$ be a cycle with $\displaystyle V= \{c_i \ | \ 1 \leq i \leq m \}$ and $\displaystyle E=\{(c_i,c_{i+1})\ | \ 1 \leq i < m \} \cup  \{(c_1,c_m) \}$.
Let $P_n$ be a path with $\displaystyle V= \{p_i \ | \ 1 \leq i \leq n\}$ and $\displaystyle E=\{(p_i,p_{i+1})\ | \ 1 \leq i < n \}$.
Graph $C_m \Box P_n$ has $n$ copies of cycles $C_m$,
each corresponding to the one vertex of the path graph.
A vertex set $V $ and an edge set $E=E_1 \cup E_2 \cup E_3$ of $C_m \Box P_n$ is defined as follows.
\begin{eqnarray*}
V &=& \{ (c_i,p_j) \ | \ 1 \leq i \leq m, 1 \leq j \leq n  \},\\
E_1 &=& \bigcup_{i=1}^m\{ (( c_i,p_j),(c_i,p_{j+1})) \ | \ 1 \leq j \leq n-1  \}, \\
E_2 &=& \bigcup_{j=1}^n\{ (( c_i,p_j),(c_{i+1},p_j)) \ | \ 1 \leq i \leq m-1  \}, \\
E_3 &=& \{ (( c_1,p_i),(c_m,p_i)) \ | \ 1 \leq i \leq n  \}, \\
E&=& E_1 \cup  E_2 \cup  E_3. 
\end{eqnarray*}
\end{definition}
\begin{example}
Double tree $DT_3$ shown in the Figure~\ref{fig:figure4-a} has a vertex set $V=\{ x(\epsilon),x(0),x(1),y(\epsilon),y(0),y(1),x(00),x(01),$
$x(10),x(11),y(00), y(01),y(10),y(11)\}$ and an edge set $E=\{(x(\epsilon),y(\epsilon)),(x(\epsilon),x(0)),(x(\epsilon),x(1)),(y(\epsilon),y(0)),(y(\epsilon),y(1)),$
$(x(0),x(00)),(x(0),x(01)),(x(1),x(10)),(x(1),x(11)),y(0),y(00)),$
$((y(0),y(01)),(y(1),y(10)),(y(1),y(11))$.
Graph $R_{5,5}$ shown in the Figure~\ref{fig:figure4-b} has a vertex set $V =\{ x_1,x_2,x_3,x_4,x_5,x_6,x_7,$
$x_8,x_9,x_{10},y_1,y_2,y_3,y_4,y_5,y_6,y_7,y_8,y_9,y_{10} \}$ and an edge set $E=\{(x_6,y_6),(x_7,y_7),$
$(x_8,y_8),(x_9,y_9),(x_{10},y_{10}) \}$. 
\end{example}
\begin{figure}[htb]
\begin{center}
\subfigure[Double Tree $DT_3$]{\label{fig:figure4-a}\includegraphics[scale=0.5]{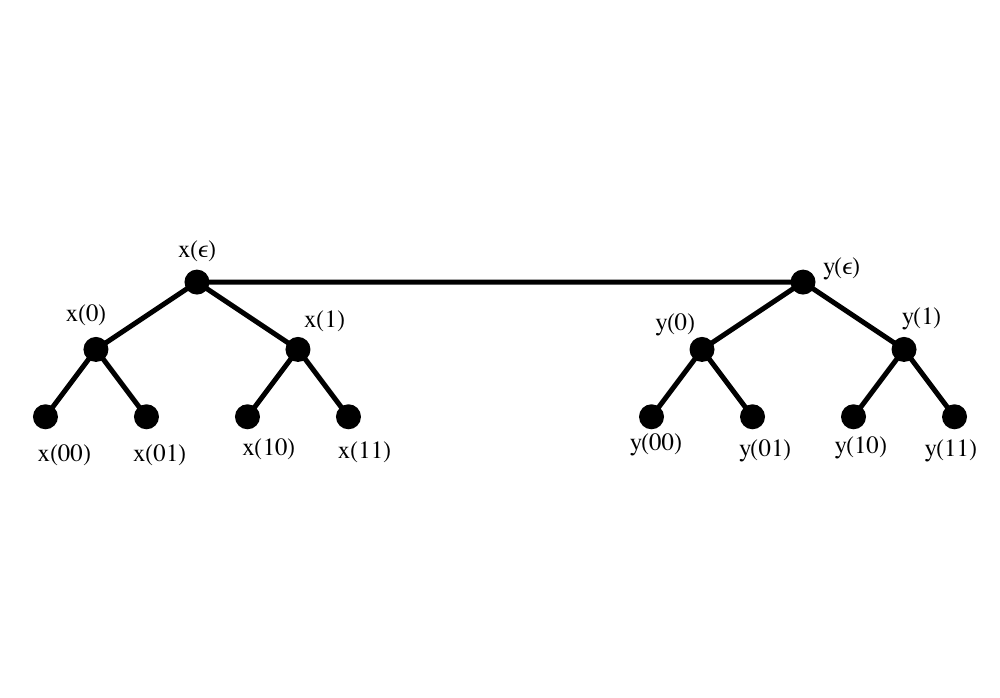}}
\subfigure[$R_{5,5}$]{\label{fig:figure4-b}\includegraphics[scale=0.5] {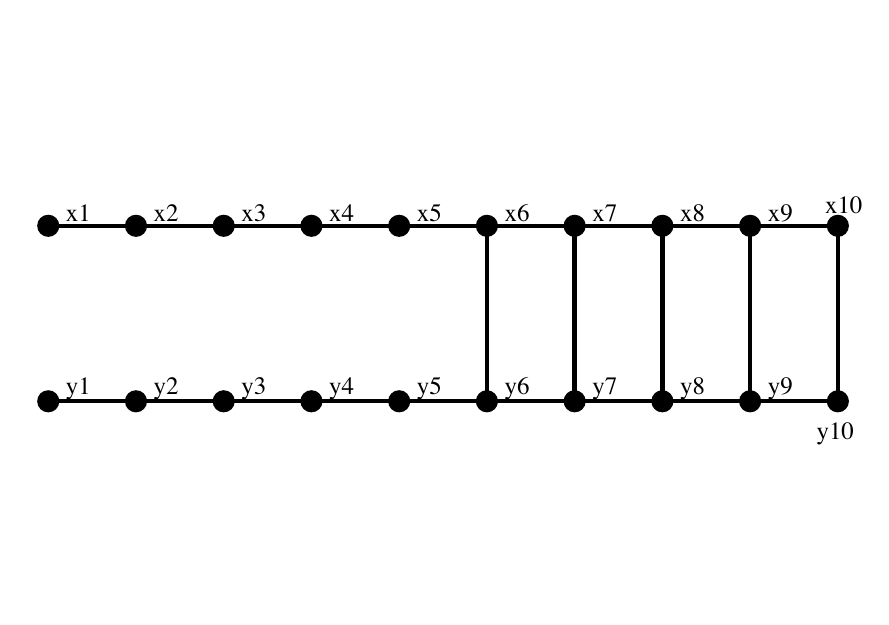}}
\caption{Double tree $DT_3$ and graph $R_{n,k}(n=5,k=5)$.}
\label{fig:figure4}
\end{center}
\end{figure}

\begin{definition}[Lollipop graph $LP_{n,m}$]
\label{def:lp}
The lollipop graph $LP_{n,m},(n \ge 3, m \ge 1)$ is obtained by
connecting  a vertex of $K_n$ to the end vertex of $P_m$
as shown in the Figure~\ref{fig:figureLP6}.
We start vertex numbering from the end vertex of the path.
Define $LP_{n,m}=(V,E)$ as follows.
\begin{eqnarray*}
V&=&\{x_1,x_2,\ldots,x_m,y_1,\ldots,y_n\},\\
E&=& \{ (x_i,x_{i+1}) \ | \ 1 \leq i \leq m-1 \} \cup \{ (y_i,y_{j}) \ | \ i \neq j,\\
&~&1 \leq i \leq n, 1 \leq j \leq n \} \cup \{(x_m,y_1)\}.
\end{eqnarray*}
\end{definition}

\begin{figure}[htb]
\begin{center}
\includegraphics[scale=0.5]{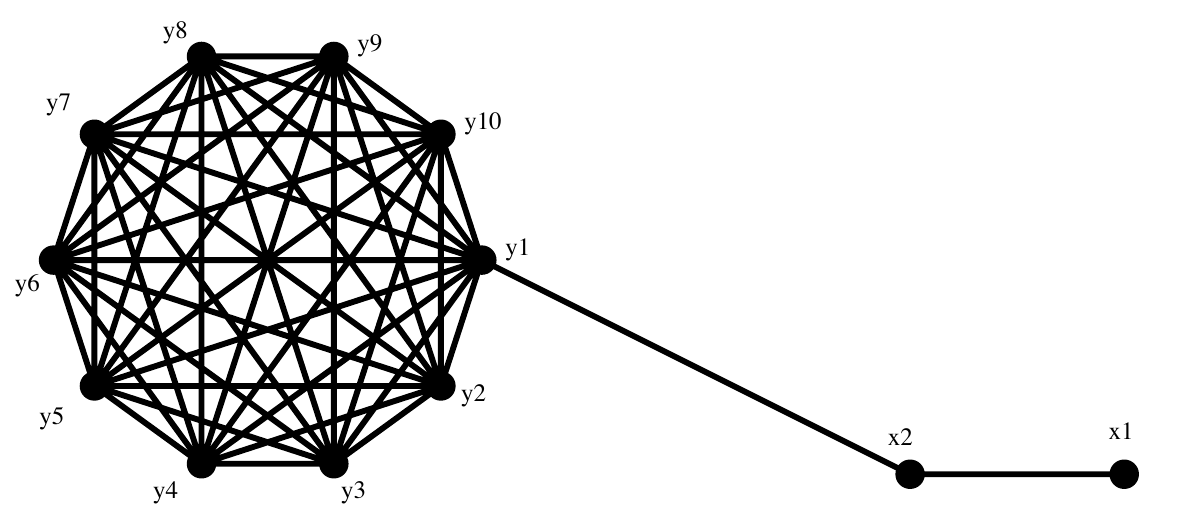}
\caption{Graph $LP_{n,m}(n=10,m=2)$}
\label{fig:figureLP6}
\end{center}
\end{figure}


\section{Minimum normalized cut of graphs}
We use the term $Mcut(G)$ to represent the minimum normalized cut.
In this section,
we review the basic properties of $Mcut(G)$ and its relation to 
the connectivity and second smallest eigenvalue of normalized Laplacian.
We derive $Mcut(G)$ of basic classes of graphs such as paths, cycles,
double trees, cycle cross paths,
complete graphs
and other graphs such as
 $R_{n,k}$, $P_{n,k}$ and $LP_{n,m}$.

\subsection{Properties of minimum normalized cut $Mcut(G)$}

\begin{definition}[Normalized cut]
Let $G=(V,E)$ be a connected graph.
Let $A, B \subset  V$, $A\neq \emptyset$, $B\neq \emptyset$ and $A \cap B =\emptyset $. 
Then the normalized cut $Ncut(A,B)$ of $G$ is defined by 
$$
Ncut(A,B)= cut(A,B)\left( \frac{1}{vol(A)}+ \frac{1}{vol(B)} \right).
$$
\end{definition}  

\begin{definition}[$Mcut(G)$]
Let $G=(V,E)$ be a connected graph.
The $Mcut(G)$ is defined by 
$$
Mcut(G)= \min \{Mcut_j(G) \mid j=1,2,\dots \}.
$$
 Where,
{\small
$$
 Mcut_j(G)=\min \{ Ncut(A,V\setminus A) \mid
  cut(A,V\setminus A)=j , A \subset V\}.
$$
}
\end{definition}

\begin{example}
Graph $G=(V,E)$ shown in the Figure~\ref{fig:ncutcom} has vertex set $V=\{1,2,3,4,5,6,7\}$ and edge set 
$ E=\{(1,2),(2,3),(3,1),$
$(3,4),(1,4),(1,5),(3,6),(6,5),(7,5),(7,6)\}$.
Volume of the graph is 20.
We compute normalized cut for the following cases.\\
$Case(1) \ A=\{1,2,3,4\},B= \{5,6,7\},vol(A)=12,vol(B)=8,$
$ cut(A,B)=2$ and $Ncut(A,B)=0.417 $.\\
$Case(2) \ A=\{1,2,3\},B= \{4,5,6,7\},vol(A)=10, vol(B)=10,$
$ cut(A,B)=4$ and $Ncut(A,B)=0.8 $.\\
$Case(3) \ A=\{1,3,4,5,6,7\},B= \{2\},vol(A)=2, vol(B)=18,$
$ cut(A,B)=2$ and $Ncut(A,B)=1.1111 $.
Comparing above 3 cases,
we obtain $Mcut(G)$ for the case(1).
\begin{figure}[htb]
\begin{center}
\subfigure[$G=(V,E)$]{\label{fig:ncutcom-a}\includegraphics[scale=0.4] {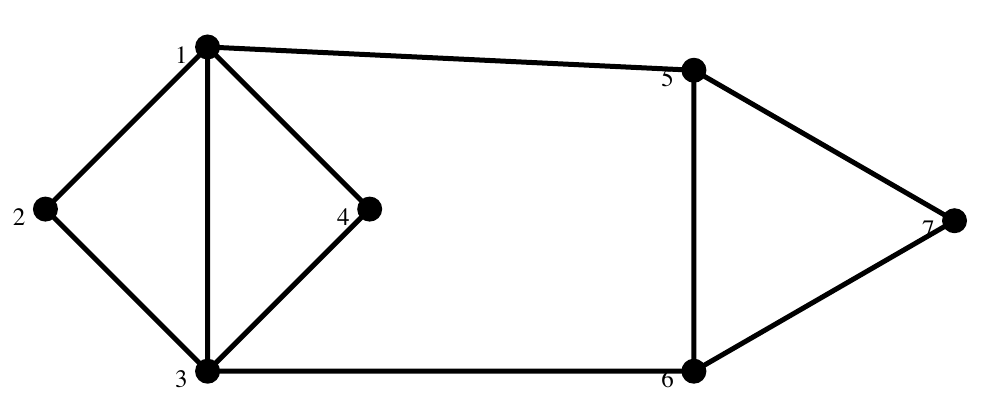}}\subfigure[$Case(1)$]{\label{fig:ncutcom-b}\includegraphics[scale=0.4] {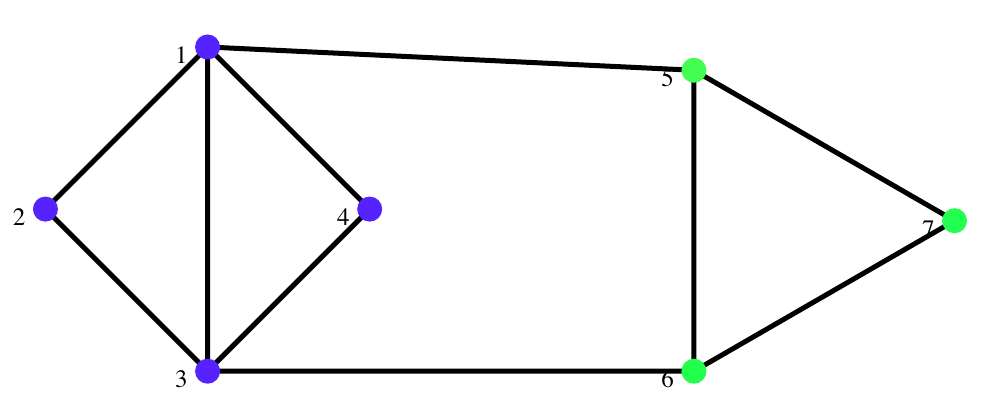}}
\subfigure[$Case(2)$ ]{\label{fig:ncutcom-c}\includegraphics[scale=0.4] {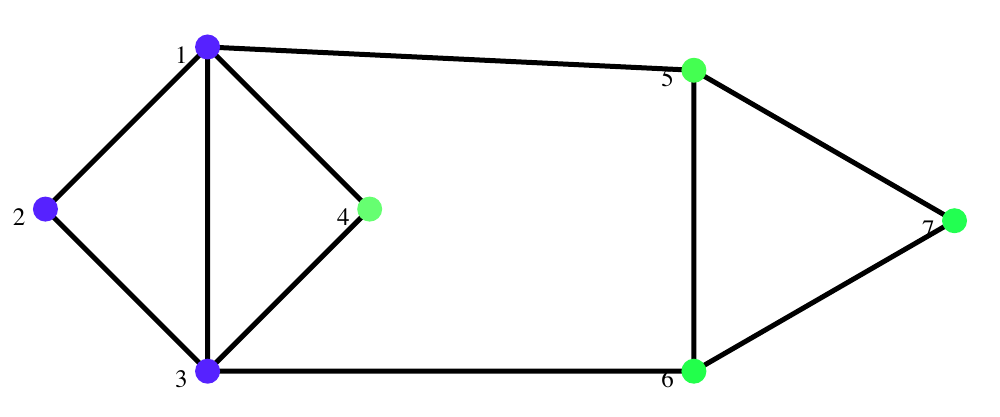}}\subfigure[$Case(3)$ ]{\label{fig:ncutcom-d}\includegraphics[scale=0.4]{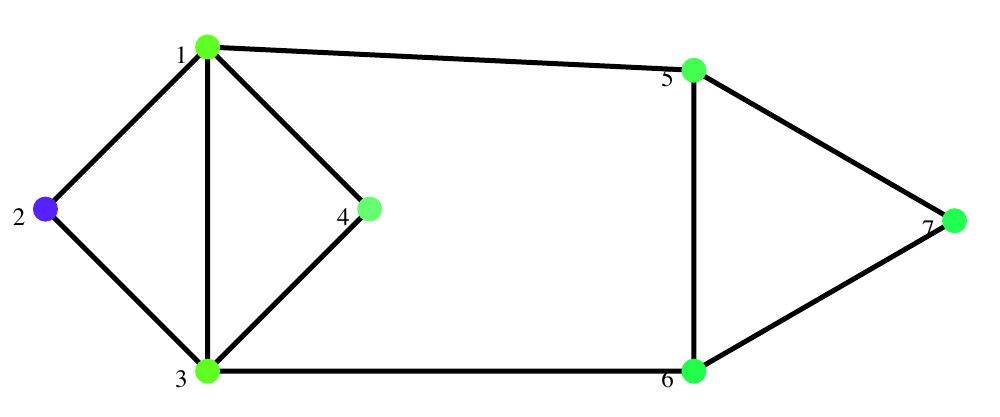}}
\caption{Normalized cut example.}
\label{fig:ncutcom}
\end{center}
\end{figure}
\end{example}

\begin{lemma}
\label{minvol}
Let $G=(V,E)$ be a connected graph.
Then
$\displaystyle \left( \frac{1}{vol(A)} + \frac{1}{vol(V \setminus A)} \right )$ is minimum when $\displaystyle vol(A)=vol(V \setminus A)=\frac{vol(G)}{2}$.
\end{lemma}\hfill\qed

\begin{proposition}
Let $G=(V,E)$ be a connected graph, 
$A \subseteq  V$ and $\triangle(G)$ the maximum degree of $G$.
Then 
\begin{enumerate}
\item $\displaystyle cut(A,V \setminus A) \geq \kappa'(G) $,
\item $\displaystyle Mcut(G) \geq \frac{4\kappa'(G)}{\triangle(G) |V|}$ and
\item If $\displaystyle cut(A,V\setminus A)= \kappa'(G)$ and $\displaystyle 2vol(A)=vol(G)$ then $\displaystyle Mcut(G) = \frac{4\kappa'(G)}{vol(G)}$.
\end{enumerate}
\label{prop4}
\end{proposition}
\begin{proof}
\noindent
\begin{enumerate}
\item Since $\kappa'(G)$ is the edge connectivity, 
$cut(A,V\setminus A) \geq \kappa'(G) $ for any $A \subseteq V$.
\item  From Lemma~\ref{minvol},
$\displaystyle \left( \frac{1}{vol(A)} + \frac{1}{vol(V\setminus A)} \right)$ is minimum when $vol(A)=vol(V\setminus A)$.
That is $\displaystyle \left( \frac{1}{vol(A)} + \frac{1}{vol(V\setminus A)} \right) \geq \frac{2}{vol(A)} =\frac{4}{vol(G)}$.
Since $\displaystyle vol(G)= \sum_{i=1}^{|V|} d_{i}  \leq |V|\triangle(G) $,
$\displaystyle Ncut(A,V\setminus A) = cut(A,V\setminus A)\left( \frac{1}{vol(A)} + \frac{1}{vol(V\setminus A)} \right) \geq  \frac{4\kappa'(G)}{ \triangle(G)|V|}$.
\item  If $cut(A,V\setminus A) = \kappa'(G) $ and $\displaystyle 2vol(A)=vol(G)$ then it is clear that,
 $\displaystyle Mcut(G)= \frac{4\kappa'(G)}{vol(G)}$.
\end{enumerate}
\end{proof}
 
 \begin{proposition}[Luxburg \cite{luxburg:2007}] 
 \label{mcutlema}
Let $G=(V,E)$ be a connected graph and $A \subset  V$.
Let $\lambda_1 \leq \lambda_2 \leq \cdots\leq \lambda_n$ be eigenvalues of $\mathcal{L}(G)$.
Then $\displaystyle  Mcut(G) \geq \lambda_2(\mathcal{L}(G))$.
\end{proposition}

\begin{proof}
Let $V= \{1,2,\ldots,n\}$.
Let $A\subset V$,
$g=(g_1,\ldots,g_n) \in \Re^n $ an eigenvector and $g=D^{1/2}f$.
Define $f_i$ as \[ f_i= \left  \{ \begin{array}{ll}
 a & \mbox{if $i \in A$,}\\
-b & \mbox{if $i \notin A$.}\end{array} \right. \]
Then
\begin{eqnarray*}
\frac{\sum_{i=1}^n \sum_{j=1}^n (f_i-f_j)^2 w_{ij}}{2 \sum_{i=1}^n f_i^2 d_i} &=& \frac{2cut(A,(V \setminus A))(a+b)^2}{2(a^2vol(A)+b^2vol(V \setminus A))}.
\end{eqnarray*}
Let this as $X$.

Now let $a=vol(V \setminus A)$ and $b=vol(A)$.
Then we have
\begin{eqnarray*}
X &=& \frac{cut(A,(V \setminus A))(vol(G))^2}{vol(V \setminus A)^2vol(A)+vol(A)^2vol(V \setminus A)}\\
&=& \frac{cut(A,(V \setminus A))(vol(G))^2}{vol(V \setminus A)vol(A)(vol(A)+vol(V \setminus A)}\\
&=& \frac{cut(A,(V \setminus A))(vol(G))}{vol(V \setminus A)vol(A)}\\
&=& cut(A,(V \setminus A))\left(\frac{1}{vol(V \setminus A)}+\frac{1}{vol(A)}\right)\\
&=& Ncut(A,V \setminus A). 
\end{eqnarray*}
With the choice of $f,a,b$ we have, 
$\displaystyle (D\vec{1})^T f= \sum_{i=1}^n d_i f_i = \sum_{i \in A} d_i a - \sum_{i \notin A} d_i b =0$.
So $f \perp D \vec{1}$.
Since $\displaystyle \lambda_2= \inf _{f \perp D \vec{1}} \frac{\sum_{i=1}^n \sum_{j=1}^n (f_i-f_j)^2 w_{ij}}{2 \sum_{i=1}^n f_i^2 d_i}$,
we have
$\displaystyle \lambda_2  \leq \frac{\sum_{i=1}^n \sum_{j=1}^n (f_i-f_j)^2 w_{ij}}{2 \sum_{i=1}^n f_i^2 d_i}=  \min_A Ncut(A,(V \setminus A))=Mcut(G)$.\end{proof}



\begin{lemma}\label{lemma:mcut4}
Let $G=(V,E)$ be a connected graph, $A$ a nonempty subset of $V$.
Then
{\small
\begin{enumerate}
\item[(i)]
$ \displaystyle Ncut(A,V\setminus A)=\frac{4 cut(A,V\setminus A) \cdot vol(V)}{(vol(V))^2 -
	  (vol(A)-vol(V\setminus A))^2}
$, and
\item[(ii)]
$\displaystyle Mcut_j(G)= \frac{4 j vol(V)}{ vol(V)^2 -  X_j }$,
where \\
$X_j=\min \{(vol(A)-vol(V\setminus A))^2 \ | \ cut(A,V\setminus A)=j, \ A \subset V\}$.
\end{enumerate}
}
\end{lemma}
\begin{proof}
\noindent
(i) 
Let $s=vol(V)$, $j=cut(A,V\setminus A)$,
$s_A=vol(A)$ and $s_{\bar{A}}=vol(V\setminus A)$.
Since $s=s_A+s_{\bar{A}}$, we have
$s_A-s_{\bar{A}}$
$=2s_A-s$ and
$s^2 - (s_A-s_{\bar{A}})^2$
$=4 s_A s_{\bar{A}}$.
\begin{eqnarray*}
Ncut(A,V\setminus A)
&=& j \cdot (\frac{1}{s_A}+\frac{1}{s_{\bar{A}}}) \\
&=& \frac{j(s_A+s_{\bar{A}})}{s_A s_{\bar{A}}} 
= \frac{j s}{s_{A}s_{\bar{A}}} \\
&=& \frac{4 j s}{s^2-(s_{A}-s_{\bar{A}})^2}
\end{eqnarray*}
\noindent
(ii)
It is followed by the definition of $Mcut_j(G)$ and (i).
\end{proof}

\begin{lemma}
\label{genlem4}
Let $G=(V,E)$ be a graph. 
If there exists a nonempty subset $A \subset V$ such that
$$
\left \vert vol(A)-vol(V\setminus A) \right
 \vert \leq \frac{vol(V)}{\sqrt{cut(A,V\setminus A)+1}},
$$
then
$$
\displaystyle Mcut(G)=\min\{ Mcut_j(G) \mid
 j=1,2,\ldots,cut(A,V\setminus A)\}.
$$
\end{lemma}

\begin{proof}
Let $j=cut(A,V\setminus A)$, $a=|vol(A)-vol(V\setminus A)|$,
$s=vol(V)$, $s_A=vol(A)$ and $s_{\bar{A}}=vol(V\setminus A)$.
Since
$a^2 \le \frac{s^2}{j+1}$ and
$Ncut(A,V\setminus A)=\frac{4 j s}{s^2 - a^2}$
by the Lemma~\ref{lemma:mcut4},
we have
$s^2-(j+1) a^2 \ge 0$ and
\begin{eqnarray*}
&& \frac{4(j+1)}{s} - Ncut(A,V\setminus A) \\
&=& \frac{4(j+1)}{s} - \frac{4 j s}{s^2-a^2} \\
&=& \frac{4(j+1)(s^2-a^2)-4 j s^2}{ s (s^2-a^2)} \\
&=& \frac{4(s^2-(j+1)a^2)}{s(s^2-a^2)} \ge 0.
\end{eqnarray*}
Let $B$ be a subset of $V$,
$s_B=vol(B)$, $s_{\bar{B}}=vol(V\setminus B)$
and $j_B=cut(B,V\setminus B)$.
If $j_B \ge j+1$ then 
we have the following using Lemma~\ref{lemma:mcut4}.
\begin{eqnarray*}
Ncut(B,V\setminus B) &=&
cut(B,V\setminus B)\left(\frac{1}{vol(B)}+\frac{1}{vol(V\setminus B)}\right) \\
&=& \frac{4 j_B s}{s^2-(s_B-s_{\bar{B}})^2} \\
&\ge& \frac{4 (j+1) s}{s^2-(s_B-s_{\bar{B}})^2} \\
&\ge& \frac{4 (j+1) s}{s^2} = \frac{4 (j+1)}{s} \\
&\ge& Ncut(A,V\setminus A) \ge Mcut_j(G).
\end{eqnarray*}
So we have $Mcut_{j'}(G) \ge Mcut_j(G)$
for any $j' > j$.
\end{proof}

\begin{lemma}
\label{gencut1}
Let $G=(V,E)$ be a graph with $vol(G)\geq 9$. 
If there exists a subset $A \subset V$
such that $cut(A,V\setminus A)=1$ and $|vol(A)-vol(G)/2 | \leq 3$,
then
$$
Mcut(G)=Mcut_1(G).
$$
\end{lemma}

\begin{proof}
Let $s=vol(G)$,
$s_A=vol(A)$ and $s_{\bar{A}}=vol(V\setminus A)$.
Since $\left\vert s_A - \frac{s}{2} \right\vert \le 3$
and $s=s_A+s_{\bar{A}}$,
we have
$\left\vert s_A - s_{\bar{A}} \right\vert \le 6$.
Since $\sqrt{1+1}\left\vert s_A - s_{\bar{A}} \right\vert
$ $ \le 6\sqrt{2}$ $< 9$ $ \le s$,
we have $Mcut(G)=Mcut_1(G)$ by the Lemma~\ref{genlem4}.
\end{proof}

\begin{lemma}
\label{genmcut2}
Let $G=(V,E)$ be a graph and $vol(G) \geq 11$.
If there exists a set $A \subset V$
such that $cut(A,V\setminus A)=2$ and $|vol(A)-vol(G)/2| \leq 3$,
then
$$
Mcut(G)=\min(Mcut_1(G), Mcut_2(G)).
$$
\end{lemma}
\begin{proof}
Since
$\left\vert vol(A) - vol(G)/2 \right\vert \leq 3$
and $vol(A)+vol(V\setminus A)=vol(G)$,
we have
$\left\vert vol(A) - vol(V\setminus A) \right\vert \leq 6$
and
$\sqrt{3}\left\vert vol(A) - vol(V\setminus A)\right\vert$
$\leq 6 \sqrt{3}$ $< 11$.
So we have
$Mcut(G)=\min(Mcut_1(G), Mcut_2(G))$
by the Lemma~\ref{genlem4}.
\end{proof}

\begin{lemma}
Let $G=(V,E)$ be a graph with $vol(G)\geq 11$.
Suppose there exists a subset $A\subset V$
such that $cut(A,V\setminus A)=2$ and $|vol(A)-vol(G)/2| \leq 3$.
If there exists no subset $B\subset V$
such that $cut(B,V\setminus B)=1$
and $\displaystyle \left\vert vol(B)-vol(G)/2 \right\vert\leq \frac{\sqrt{36+(vol(G))^2}}{2 \sqrt{2}}$,
then
$$
Mcut(G)=Mcut_2(G).
$$
\end{lemma}
\begin{proof}
Let $s=vol(G)$,
$s_A=vol(A)$ and $s_{\bar{A}}=vol(V\setminus A)$.
Since $\left\vert s_{A} - s/2 \right\vert \leq 3$,
we have $\left\vert s_{A} -s_{\bar{A}} \right\vert \leq 6$
and
\begin{eqnarray*}
Mcut_2(G) & \leq & Ncut(A,V\setminus A) \\
&=& \frac{8s}{s^2 - (s_A-s_{\bar{A}})^2} \\
&\leq& \frac{8s}{s^2 -36}.
\end{eqnarray*}
Let $B \subset V$ with $cut(B,V\setminus B)=1$,
$s_B=vol(B)$ and $s_{\bar{B}}=vol(V\setminus B)$.
If $B$ exists, then 
$\left\vert s_B - s/2 \right\vert > \frac{\sqrt{s^2+36}}{2\sqrt{2}}$,
by the assumption.
So we have
$\left\vert s_B - s_{\bar{B}} \right\vert > \frac{\sqrt{s^2+36}}{\sqrt{2}}$
and
\begin{eqnarray*}
Ncut(B,V\setminus B) &=& \frac{4s}{s^2 - (s_B-s_{\bar{B}})^2} \\
&\geq& \frac{4s}{s^2-\frac{s^2+36}{2}} \\
&=& \frac{8s}{s^2 -36} \geq Mcut_2(G).
\end{eqnarray*}
That is $Mcut(G)=Mcut_2(G)$ by the Lemma~\ref{genmcut2}.
\end{proof}
 
Next we derive formulae for minimum normalized cut $Mcut(G)$ of some elementary graphs.

\subsection{$Mcut(G)$ of basic classes of graphs}

\begin{theorem}
\label{propmcut}
Let $G=(V,E)$ be a graph. 
\begin{enumerate}
\item If $G$ is a regular graph of degree $d$ and $G \neq K_n, n>3$ and
$|V|=n$,
then   \[ Mcut(G) \geq  \left \{ \begin{array}{ll}
                 \frac{4}{n} & \mbox{if $n$ is even,}\\
               \frac{4n}{(n^2-1)} & \mbox{if $n$ is odd.}
               \end{array} \right. \]
               
\item
For the cycle $C_{n}$ ($n \ge 3$),
\[ Mcut(C_n)= \left \{ \begin{array}{ll}
 \frac{4}{n} & \mbox{if $n$ is even,}\\
\frac{4n}{(n^2-1)} & \mbox{if $n$ is odd.} \end{array} \right. \]
This can be written as
$\displaystyle Mcut(C_n) = 
 \frac{n}{\lfloor \frac{n}{2} \rfloor \lceil \frac{n}{2} \rceil}.$

\item For the complete graph $K_n$,
\begin{eqnarray*}
Mcut(K_n)& = & \frac{n}{n-1} \\
&=& \lambda_2.
\end{eqnarray*} 

\item For the path graph $P_n$ ($n \ge 2$),
\[ Mcut(P_n) = \left \{ \begin{array}{ll}
 \frac{2}{n-1} & \mbox{if $n$ is even,}\\
\frac{2(n-1)}{n(n-2)} & \mbox{if $n$ is odd.} \end{array} \right. \]\\
This can be written as
$$
Mcut(P_n)= \frac{2n-2}{4\lfloor \frac{n}{2} \rfloor \lceil \frac{n}{2}
      \rceil -2n+1}.
$$

\item For the cycle cross paths $G = C_m \Box  P_n $ , 
\[ Mcut(C_m \Box P_n) = \left \{ \begin{array}{cc}
\frac{2(2n-1)}{16 \lfloor{\frac{n}{2}}\rfloor \lceil{\frac{n}{2}}\rceil-4n+1} & \mbox{$2n > m$,}\\
\frac{nm}{(2n-1)\lfloor{\frac{m}{2}}\rfloor \lceil{\frac{m}{2}}\rceil} & \mbox{$2n \leq m$.} \end{array} \right. \]

\item For the double tree $DT_n$ with depth $n$,
$\displaystyle Mcut(DT_n)=\frac{2}{2^{n+1} -3}$. 
\end{enumerate}
\end{theorem}

\begin{proof}
\begin{enumerate}


\item For a regular graph of degree $d$,
 $\displaystyle \kappa'(G)=\triangle(G) =\delta(G) =d$.
For $A \subset  V$, 
$\displaystyle Ncut(A,V\setminus A)
 \geq \kappa'(G)\left( \frac{1}{d|A|}+\frac{1}{d|V \setminus A|}\right)
 = \frac{|V|}{|A||V\setminus A|}$.
If $cut(A,V\setminus A)=\kappa'(G)$ then we have
 $\displaystyle Ncut(A,V\setminus A)=\frac{|V|}{|A||V\setminus A|}$.
$Ncut(A,V\setminus A)$ is minimum,
 when $|A|=|V\setminus A|$ by the Lemma~\ref{minvol}.
If $V$ is even then $\displaystyle Mcut(G) \geq \frac{4}{|V|}
 =\frac{4}{n}$ by Lemma~\ref{minvol}.

If $|V|$ is odd then,
we can write $|V|$ as $\displaystyle
      |V|=\frac{|V|-1}{2}+\frac{|V|+1}{2}$,
where  $\displaystyle -1 \leq |A|-|V \setminus A| \leq 1$.
Then $\displaystyle Ncut(A,V\setminus A) \geq \kappa'(G)\left(
 \frac{2}{d(|V|-1)}+\frac{2}{d(|V|+1)}\right)=
\frac{4|V|}{(|V|+1)(|V|-1)}$.
Hence $\displaystyle Mcut(G) \geq \frac{4|V|}{(|V|+1)(|V|-1)}=\frac{4n}{n^2-1}$. \hfill\qed\\


\item
Let $A_k=\{x_i \ |\ i \le k\}$  ($k=1,\ldots,n-1$).
We note $vol(C_n)=2n$, $vol(A_k)=2k$, $vol(V\setminus A_k)=2n-2k$,
$vol(A_k)-vol(V\setminus A_k)=4k-2n$ and
$$
Ncut(A_k,V\setminus A_k)=\frac{4n}{n^2-(2k-n)^2}.
$$
If $n$ is even then $Ncut(A_{\frac{n}{2}},V\setminus A_{\frac{n}{2}})$
$=\frac{4}{n}$ is the minimum of $Ncut(A_k,V\setminus A_k)$.
If $n$ is odd then $Ncut(A_{\frac{n+1}{2}},V\setminus
     A_{\frac{n+1}{2}})$
$=Ncut(A_{\frac{n-1}{2}},V\setminus A_{\frac{n-1}{2}})$
$=\frac{4n}{n^2-1}$ is the minimum of $Ncut(A_k,V\setminus A_k)$.
Since $vol(A_{\frac{n}{2}})-vol(V\setminus A_{\frac{n}{2}})=0$,
$vol(A_{\frac{n-1}{2}})-vol(V\setminus A_{\frac{n-1}{2}})=-2$
and 
$\displaystyle \frac{vol(V)}{\sqrt{cut(A_k,V\setminus A_k)+1}}$ 
$= \frac{2n}{\sqrt{3}}$ $\ge \frac{6}{\sqrt{3}}$,
we have
$Mcut(C_n)=Mcut_2(C_n)$ by Lemma~\ref{genlem4}.

We note that for any nonempty subset $A \subset V$ with
     $cut(A,V\setminus A)=2$,
there exists a $k$
such that $Ncut(A,V\setminus A)=Ncut(A_k,V\setminus A_k)$
and $\kappa'(C_n)=2$.

For even $n$,
$\displaystyle n=\lfloor \frac{n}{2} \rfloor=\lceil \frac{n}{2} \rceil$ and for odd $n$,
$\displaystyle \frac{(n-1)}{2}=\lfloor \frac{n}{2} \rfloor$ and $\displaystyle \frac{(n+1)}{2}=\lceil\frac{n}{2} \rceil$.
Combining odd and even cases together we can write $Mcut(C_n)$ as $\displaystyle Mcut(C_n) = \frac{4n}{4\lfloor \frac{n}{2} \rfloor \lceil \frac{n}{2} \rceil}$. \hfill\qed


\item
For a complete graph $K_n$, 
$|V|=n$,
$\kappa'(K_n)=n-1$ and $vol(K_n)=n(n-1)$.
For any subset $A \subset V $,
 we have  $vol(A)=|A|(n-1)$ and 
$\displaystyle cut(A,(V\setminus A))=|A|(n-|A|)$.
Then $\displaystyle Mcut(K_n)= |A|(n-|A|) \left(
\frac{1}{|A|(n-1)}+\frac{1}{(n-|A|)(n-1)} \right)
 =\frac{n}{n-1}$.\hfill\qed


\item
Let $A_k=\{x_i \ |\ i \le k\}$  ($k=1,\ldots,n-1$).
We note that $vol(P_n)=2n-2$, $vol(A_k)=2k-1$, $vol(V\setminus A_k)=2n-2k-1$,
$vol(A_k)-vol(V\setminus A_k)=4k-2n$ and
$$
Ncut(A_k,V\setminus A_k)=\frac{2(n-1)}{(n-1)^2-(2k-n)^2}.
$$

If $n$ is even then $Ncut(A_{\frac{n}{2}},V\setminus A_{\frac{n}{2}})$
$=\frac{2}{n-1}$ is the minimum of $Ncut(A_k,V\setminus A_k)$.
If $n$ is odd then $Ncut(A_{\frac{n+1}{2}},V\setminus
     A_{\frac{n+1}{2}})$
$=Ncut(A_{\frac{n-1}{2}},V\setminus A_{\frac{n-1}{2}})$
$=\frac{2(n-1)}{(n-1)^2-1}$
$=\frac{2(n-1)}{n(n-2)}$ is the minimum of $Ncut(A_k,V\setminus A_k)$.
Since $vol(A_{\frac{n}{2}})-vol(V\setminus A_{\frac{n}{2}})=0$,
$vol(A_{\frac{n-1}{2}})-vol(V\setminus A_{\frac{n-1}{2}})=-2$
and 
$\displaystyle \frac{vol(V)}{\sqrt{cut(A_k,V\setminus A_k)+1}}$ 
$= \frac{2n-2}{\sqrt{2}}$ $\ge \frac{2}{\sqrt{2}}$,
we have
$Mcut(C_n)=Mcut_1(P_n)$ by Lemma~\ref{genlem4}.
\hfill\qed

\item The cycle cross path $G=C_m \Box P_n ( n \geq 2, m\geq 3)$ is a graph which has $n$ copies of cycles $C_m$,
 each corresponding to the one vertex of $P_n$. 
$\displaystyle \kappa'(C_m \Box P_n) =\min \{ \kappa'(C_m)|V(P_n)|, \kappa'(P_n)|V(C_m)|, \delta(C_m)+\delta(P_n) \}= \delta(C_m)+\delta(P_n) =3$.\\
\noindent{\bf{Case (i)}} Let
$\displaystyle A_1=\{(c_i,p_j)|
 1 \leq j \leq \lfloor\frac{n}{2} \rfloor, 1 \leq i \leq m \}$
 and
 $\displaystyle V\setminus A_1=\{ (c_i,p_j) |
 \lfloor\frac{n}{2} \rfloor+1 \leq j \leq n, 1 \leq i \leq m \}$.
 We note that
 $\displaystyle vol(A_1)=\lfloor\frac{n}{2} \rfloor(vol(C_m)+2m) -m$,
 $\displaystyle vol(V \setminus A_1)=\lceil\frac{n}{2}
      \rceil(vol(C_m)+2m) -m$
 and $cut(A_1,V\setminus A_1)=m$.
 Then 
 $\displaystyle Ncut(A_1,V \setminus A_1)
 =\frac{m(4mn-2m)}{(\lfloor \frac{n}{2} \rfloor 4m-m)
 ( \lceil \frac{n}{2} \rceil 4m-m)}
 =\frac{2(2n-1)}{16 \lfloor\frac{n}{2}
 \rfloor \lceil \frac{n}{2} \rceil -4n+1}$.
 When $n$ is even,
 $\displaystyle Ncut(A_1,V \setminus A_1)=\frac{2}{2n-1}$.
 When $n$ is odd,
 $\displaystyle Ncut(A_1,V \setminus A_1)=\frac{2(2n-1)}{(2n-3)(2n+1)}$.\\
 {\bf{Case (ii)}} Let $\displaystyle A_2=\{(c_i,p_j) \ |
 \ 1 \leq i \leq \lfloor\frac{m}{2} \rfloor, 1 \leq j \leq n \}$ and
 $ \displaystyle (V\setminus A_2)=\{(c_i,p_j) \ |
 \ \lfloor\frac{m}{2} \rfloor+1 \leq i \leq m, 1 \leq j \leq n \}$.
 We note that $\displaystyle vol(A_2)= n vol(C_{\lfloor\frac{m}{2}\rfloor}) + 2
 \lfloor\frac{m}{2} \rfloor(n-1) =2(2n-1)\lfloor\frac{m}{2} \rfloor$
 and $\displaystyle vol(V\setminus A_2)=2(2n-1)\lceil\frac{m}{2} \rceil$.
In this case,
the graph cut horizontally through the cycles and we have
 $cut (A_2,V\setminus A_2)=2n$.
Hence $\displaystyle Ncut(A_2,V \setminus A_2)=
\frac{nm}{(2n-1)\lfloor\frac{m}{2} \rfloor \lceil\frac{m}{2} \rceil}$.
 When $m$ is odd,
 $\displaystyle \frac{4nm}{(2n-1)(m^2-1)}$ and when $m$ is even,
$\displaystyle \frac{4n}{(2n-1)m}$.\\

\noindent{\bf{Case (iii)}} Let
$\displaystyle B_k=\{(c_i,p_1)| 1 \leq i \leq k \}$
($1 \le k < m$).
We note that
$cut(B_k,V\setminus B_k)=k+2$ and $vol(B_k)=3k$.
Since
$Ncut(A_1,V\setminus A_1)-Ncut(B_k,V\setminus B_k)=$
{\tiny$
\frac{2 (-1+2 n) \left(9 k^2+k m \left(3-16 n+4 n^2\right)+2 m \left(-3-4 n+4 n^2\right)\right)}{3 k (3 k+m (2-4 n)) \left(-3-4 n+4 n^2\right)}
$},
we can verify
$Ncut(A_1,V\setminus A_1)$
$\le Ncut(B_k,V\setminus B_k)$ for any $k$ and $m \le 2n$.

\noindent{\bf{Case (iv)}} Let
$\displaystyle C_k=\{(c_1,p_j)| 1 \leq j \leq k \}$
($1 \le k < n$).
We note that
$cut(C_k,V\setminus B_k)=2k+1$ and $vol(C_k)=4k-1$.
Since
$Ncut(C_k,V\setminus C_k)=$
{\tiny
$-\frac{2 (1+2 k) m (-1+2 n)}{(-1+4 k) (-1+4 k+m (2-4 n))}$
},
we can verify
$Ncut(A_2,V\setminus A_2)$
$\le Ncut(C_k,V\setminus C_k)$ for any $k$ and $2n \le m$.

Now compare the case (i) with case (ii).

For the case of $2n \ge m+1$, we have
$\frac{vol(G)}{\sqrt{cut(A_1, V\setminus A_1)}}$
$=\frac{2m (2n-1)}{\sqrt{m+1}}$
$\ge \frac{2m (2n-1)}{\sqrt{2n}}$
$= 2m \sqrt{\frac{4n^2-4n+1}{2n}}$
$\ge 2m \sqrt{2n -2}$
$\ge 4m $
$\ge \vert vol(A_1) - vol(V\setminus A_1) \vert$
and $Mcut_{2n}(G) > Mcut_m(G)$.
So $Mcut(G)=Ncut(A_1,V\setminus A_1)$.

If $2n \le m$, then we have
$\frac{vol(G)}{\sqrt{cut(A_2, V\setminus A_2)}}=$
$\frac{2m (2n-1)}{\sqrt{2n+1}}$
$\ge\frac{4n (2n-1)}{\sqrt{2n+1}}$
$=2(2n-1)\sqrt{\frac{4n^2}{1+2n}}$
$=2(2n-1)\sqrt{2n-\frac{2n}{2n+1}}$
$\ge 2(2n-1)\sqrt{2n-1}$
$\ge 2(2n-1)$
$\ge \vert vol(A_2) - vol(V\setminus A_2) \vert$
and $Mcut_{m}(G) > Mcut_{2n}(G)$.
So $Mcut(G)=Ncut(A_2,V\setminus A_2)$.

\hfill\qed

\item 
The size of a tree is $|T_n|=1+2+\cdots+2^n=2^n-1$ and the size of a double tree is $|DT_n|=2|T_n|=2^{n+1}-2$.
The volume of a tree is $vol(T_n)=2 vol(T_{n-1})+4$,
which can be written as 
$vol(T_n)+4 = 2(vol(T_{n-1})+4)=2^2 (vol(T_{n-2})+4)=\cdots=2^{n-1} (vol(T_1)+4)=2^{n+1}$.
Therefore the volume of a tree is $vol(T_n)=2^{n+1}-4$ and the volume of a double tree is $\displaystyle vol(DT_n)= 2vol(T_n) +2= 2^{n+2}-6$.

Let $A_1=\{x(w) \ | \ w \in \Sigma^{< n} \}$ and $ V \setminus A_1 =\{ y(w) \ | \ w \in \Sigma^{< n} \}$.
Then we have
$vol(A_1) = vol(T_n)+1$
$= 2^{n+1} -3$,
$vol(V\setminus A_1) = 2^{n+1} -3$,
$cut(A_1,V\setminus A_1)= 1$.

Therefore 
$Ncut(A_1,V\setminus A_1) =$
$\frac{2}{(vol(T_n)+1)}$
$= \frac{2}{2^{n+1} -3}$
$=\frac{4}{vol(DT_n)}$.

Here $\kappa'(DT_n)=1$ and $2vol(A_1)=vol(DT_n)$.
Then from the Proposition~\ref{minvol}, 
$\displaystyle Mcut(DT_n) = \frac{2}{2^{n+1} -3}$.

\end{enumerate}
\end{proof}


\subsection{$Mcut$ of roach type graphs $R_{n,k}$}
\label{subsec:roach}

Next, we consider the graph $R_{n,k}$ and derive a formula for $Mcut(R_{n,k})$ based on $n,k$.

\begin{theorem}
\label{propmcutg}
For $R_{n,k} (n\geq 1, k>1)$,
 $Mcut(R_{n,k})$ is given by
\[ \left\{ \begin{array}{cc}
\frac{2}{3} & \mbox{$(n=1,k=2)$,}\\
\\
\frac{4}{-2+3 k+2 n} & \mbox{$(*_1 \wedge  (k \geq 4) \wedge$}\\
& \mbox{$   (n < K_1))$,}\\
\\
\frac{4 (-2+3 k+2 n)}{(-5+3 k+2 n) (1+3 k+2 n)}& \mbox{$(*_4 \wedge ( k \geq 4) \wedge$}\\
&\mbox{$  ( n < K_4))$,}\\
\\
\frac{4 (-2+3 k+2 n)}{(-4+3 k+2 n)(3 k+2 n)} &\mbox{$(*_3 \wedge  (k \geq 4) \wedge$}\\
&\mbox{$  ( n < K_3))$,}\\
\\
\frac{4 (-2+3 k+2 n)}{(-3+3 k+2 n)(-1+3 k+2 n)}& \mbox{$(*_2\wedge  (k \geq 4) \wedge $}\\
&\mbox{$  ( n < K_2)),\vee  (n=1,k=3)$}\\
&\mbox{$ \vee (n=2,k=3)$,}\\
\\
\frac{6k+4n-4}{(2n-1)(6k+2n-3)} &  \mbox{$((k \geq 4)$}\\
&\mbox{$ \wedge ((*_1 \wedge  (K_1 \leq n))\vee$}\\
&\mbox{$(*_4 \wedge ( K_4 \leq n))\vee$}\\
&\mbox{$ (*_3 \wedge  (K_3 \leq n))\vee $}\\
&\mbox{$(*_2 \wedge  (K_2\leq n)) )) \vee$}\\
&\mbox{$  (k=2 \wedge (n\geq 2))\vee $}\\
&\mbox{$ (k=3 \wedge (n \geq 3))$},
\end{array} \right. \]
where 
\begin{eqnarray*}
*_1 &=& ((3 \mid n) \wedge (2 \mid k)),\\
*_2 &=& (3 \nmid n) \ and \ (2 \nmid k),\\
*_3 &=& (3 \nmid n) \ and \ ( 2 \mid k),\\
*_4 &=&((3 \mid n) \wedge (2 \nmid k)), \\
K_1 &=& 1-\frac{1}{\sqrt{2}}-\frac{3 k}{2}+\frac{3 k}{\sqrt{2}},\\
K_2 & =& 1-\frac{3k}{2}+\frac{\sqrt{1-12 k+18 k^2}}{2},\\
K_3 & =& 1-\frac{3 k}{2}+\frac{\sqrt{-1-6 k+9 k^2}}{\sqrt{2}},\\
K_4 &=&  1-\frac{3 k}{2}+\frac{\sqrt{-7-12 k+18 k^2}}{2}. 
\end{eqnarray*}
\end{theorem}

\begin{proof}
Let $\displaystyle V(R_{n,k})= \{ x_i \ | \ 1 \leq i \leq n+k \}\cup \{ y_i \ | \ 1  \leq i \leq n+k \}$.
Volume of $R_{n,k}$ is $vol(R_{n,k})= 2(2n-1+3k-1)=6k+4n-4$.\\
We consider the following cases in order to find the $Mcut(R_{n,k})$.\\
\noindent {\bf{Case(i)}} Let $A_1\subseteq V(R_{n,k})$,
 where $\displaystyle A_1 =\{ x_i \ | \ 1 \leq i \leq n+k \}$ and $\displaystyle V\setminus A_1=\{ y_i \ | \ 1  \leq i \leq n+k \}$.
Then the volume $vol(A_1)$ is $\displaystyle \frac{vol(R_{n,k})}{2}=2n+3k-2$ and $cut(A_1,V\setminus A_1)=k$.
So we have 
\begin{eqnarray*}
Ncut(A_1,V \setminus A_1) &=& k \left(\frac{1}{2n+3k-2}+\frac{1}{3k+2n-2} \right)\\
&=& \frac{2k}{3k+2n-2}.
 \end{eqnarray*}
Let this value as $c_1$.\\
\noindent {\bf{Case(ii)}}
Let $A_2\subseteq V(R_{n,k})$ such that $\displaystyle A_2 =\{ x_i \ | \ 1 \leq i \leq n \}$ and $\displaystyle V\setminus A_2=\{ x_i \ | \ n+1 \leq i \leq n+k \}\cup \{ y_i \ | \ 1  \leq i \leq n+k \}$.
Then the volume $vol(A_2)=2n-1$, 
$vol(V\setminus A_2)=vol(R_{n,k})-vol(A_2)=2n+6k-3$ and $\displaystyle cut(A_2,(V\setminus A_2))=1$.
So we have
$$
  Ncut(A_2,V \setminus A_2)=\frac{(6k+4n-4)}{(2n-1)(6k+2n-3)}.
$$
 Let this value as $c_2$.\\
\noindent {\bf{Case(iii)}} Suppose there exists $|A_3| <n$ such that $cut(A_3,(V\setminus A_3))=1$. 
Let $vol(A_3)=2n-1-2x$,
where $x=|A_2|-|A_3|$ and $|A_2|=n$.
Then $vol(V\setminus A_3)= 6k+2n-3 +2x$.
$\displaystyle Ncut(A_3,V \setminus A_3)=\frac{1}{2n-1-2x} +\frac{1}{6k+2n-3+2x} = \frac{6k+4n-4}{(2n-1)(6k+2n-3)+4x(1-(3k+x))}$.
Since $4x(1-(3k+x)) <0$,
$Ncut(A_3,V\setminus A_3)> Ncut(A_2,V \setminus A_2) (Case(ii)<Case(iii))$.
Since $c_2$ is smaller than $Case(iii)$,
we can ignore this case.\\
\noindent {\bf{Case(iv)}} Let $\displaystyle  A_4(\alpha)=\{ x_i \ | \ 1 \leq i \leq n+\alpha  \} \cup \{ y_i \ | \ 1 \leq i \leq n+\alpha  \}$,
where $1 \leq \alpha <k$ and $\displaystyle  V \setminus A_4(\alpha)= \{ x_i \ | \ n+\alpha +1 \leq i \leq n+k \} \cup \{ y_i \ | \ n+\alpha +1 \leq i \leq n+k \}$.
Then $\displaystyle  vol(A_4(\alpha))=2(2n-1+3\alpha )=4n+6\alpha -2$,
$vol(V \setminus A_4(\alpha))=6k-2-6\alpha $ and
$cut(A_4(\alpha),V \setminus A_4(\alpha))=2$.
Then we have, 
$$
 Ncut(A_4(\alpha),V\setminus A_4(\alpha))
=\frac{(3k+2n-2)}{(2n-1+3\alpha)(3k-3\alpha-1)}.
$$
Let this value as $c_4(\alpha)$.\\
Minimum of $c_4(\alpha)$ can be obtained by differentiating with respect to $\alpha$.

$\displaystyle \frac{dc_4(\alpha)}{d\alpha} =0$ gives minimum value of $c_4(\alpha)$ at $\displaystyle \alpha_0=\frac{3k-2n}{6}$.
But $\alpha_0$ is not an integer for all $n,k$.
If  $\frac{3k-2n}{6} <1$ that is $\displaystyle  1 \leq k < \frac{6+2n}{3}$ then the minimum value is $c_4(1)$.
 Then we have
$$
c_4(1)= \frac{2-3k-2n}{8-6k+8n-6kn}.
$$

If $\displaystyle 1 \leq \frac{3k-2n}{6} < k$ that is $k \geq \frac{6+2n}{3}$ then the minimum value is $\displaystyle c_4(\frac{3k-2n}{6})$ whenever $\displaystyle \frac{3k-2n}{6} \in \mathbf{Z}$. 
$$
c_4(\frac{3k-2n}{6})= \frac{4}{-2+3k+2n}.
$$
If $k \geq \frac{6+2n}{3}$ and $2 \nmid k$ and $3 \mid n$ then the minimum value is $\displaystyle c_4(\frac{3k-2n}{6}+\frac{1}{2})=c_4(\frac{3k-2n}{6}-\frac{1}{2})$.
$$
c_4(\frac{3k-2n}{6}+\frac{1}{2})=\frac{4 (-2+3 k+2 n)}{(-5+3 k+2 n) (1+3 k+2 n)}.
$$
If $k \geq \frac{6+2n}{3}$ and $3 \nmid n$ and $2 \mid k$ then the minimum value is $\displaystyle c_4(\frac{3k-2n}{6}+\frac{1}{3})=c_4(\frac{3k-2n}{6}-\frac{1}{3})$.
$$
c_4(\frac{3k-2n}{6}-\frac{1}{3})=\frac{4 (-2+3 k+2 n)}{(-4+3 k+2 n) (3 k+2 n)}.
$$
If $k \geq \frac{6+2n}{3}$ and $3 \nmid n$ and $2 \nmid k$ then the minimum value is $\displaystyle c_4(\frac{3k-2n}{6}+\frac{1}{6})=c_4(\frac{3k-2n}{6}-\frac{1}{6})$.
$$
c_4(\frac{3k-2n}{6}-\frac{1}{6})=\frac{4 (-2+3 k+2 n)}{(-3+3 k+2 n)(-1+3 k+2 n)}.
$$ 
\noindent {\bf{Case(v)}} Let $A_5 =\{ x_i \ | \ 1 \leq i \leq n+1 \} $ and $V\setminus A_5=\{x_i \ | \ n+2 \leq i \leq n+k \}\cup \{ y_i \ | \ 1 \leq i \leq n+k \}$. 
Then $vol(A_5)= 2n+2$ and $vol(V\setminus A_5) =2n+6k-6$.
Then we have $\displaystyle Ncut(A_5,V\setminus A_5) = 2 \left( \frac{1}{2n+2} + \frac{1}{2n+6k-6} \right)=\frac{2n+3k-2}{(n+1)(n+3k-3)}$.\\
\noindent Now we can compare all cases considered above.\\
If $k=2$ and $n=1$ then it is easy show that $c_1$ is the minimum.
If $k=2$ and $n \geq 2$ then it is easy to show that $c_2$ is the minimum.
If $k=3$ and $n=1$ then $c_4(\frac{3k-2n}{6}-\frac{1}{6})$ is the minimum.
If $k=3$ and $n=2$ then $c_4(\frac{3k-2n}{6}+\frac{1}{6})$ is the minimum.
If $k=3$ and $n \geq 3$ then we can easily show that $c_2$ is the minimum.
If $k \ge 4$ and $n=1$ then $c_4$ is the minimum.
Next we assume that $k \geq 4$ and $n \ge 2$.
It is easy to check that $c_2$ is smaller than $c_1$, $c_3$ and $c_5$.
So we compare $c_2$ with $c_4$ for $k \geq 4$.
Then we have the following results. 
 If $(*_1 $ and $(n < K_1))$ then $c_4(\frac{3k-2n}{6})$ is smaller than $c_2$. 
 If $(*_2 $ and $( n < K_2))$ then $c_4(\frac{3k-2n}{6}-\frac{1}{6})$ is smaller than $c_2$.
 If $(*_3 $ and $( n < K_3))$ then $c_4(\frac{3k-2n}{6}-\frac{1}{3})$ is smaller than $c_2$.
 If $(*_4$ and $( n < K_4))$ then $c_4(\frac{3k-2n}{6}+\frac{1}{2})$ is smaller than $c_2$.
We can summarize the results as follows.
\[\left\{ \begin{array}{cc}
c_1 & \mbox{$n=1,k=2$,}\\
c_4(\frac{3k-2n}{6}) & \mbox{$(*_1 \wedge  (k \geq 4) \wedge   (n < K_1))$,}\\
c_4(\frac{3k-2n}{6}+1/2) & \mbox{$(*_4 \wedge ( k \geq 4) \wedge  ( n < K_4))$,}\\
c_4(\frac{3k-2n}{6}-1/3) & \mbox{$(*_3 \wedge  (k \geq 4) \wedge  ( n < K_3))$,}\\
c_4(\frac{3k-2n}{6}-1/6) & \mbox{$(*_2\wedge  (k \geq 4) \wedge   ( n < K_2))\vee$}\\
&\mbox{$(n=1,k=3)\vee (n=2,k=3)$,}\\
c_2& \mbox{$((k \geq 4) \wedge ((*_1 \wedge  (K_1 \leq n))\vee$}\\
&\mbox{$(*_2 \wedge  (K_2\leq n)) ))$}\\
&\mbox{$\vee (*_3 \wedge  (K_3 \leq n))\vee(*_4 \wedge ( K_4 \leq n))$}\\
& \mbox{$\vee  (k=2 \wedge (n\geq 2))\vee  (k=3 \wedge (n \geq 3))$.}
\end{array} \right. \]
Finally, 
we want to show that for any arbitrary subset $A$, 
$cut(A,V\setminus A)=1$ or $cut(A,V\setminus A)=2$ gives the minimum normalized cut.
We notice that every subset $A$ with $cut(A,V\setminus A)=1$ is $A_2$ or $A_3$ and every subset $A$ with $cut(A,V\setminus A)=2$ are $A_1,A_5,A_4$.
We consider all cases with $cut(A,V\setminus A)=1$ and the minimum occurs at $A_2$.
There may be several partitions with $cut(A,V\setminus A) \geq 2$.
Let $k \geq 4$.
Then we note that $vol(R_{n,k})\geq 24$ and there exists a subset $A_4$ in Case(iv),
which minimize the $\displaystyle \left(\frac{1}{vol(A)}+\frac{1}{vol(R_{n,k})-vol(A)}\right)$ with $cut(A,V\setminus A)=2$.
We note that $\displaystyle |vol(A_4)-\frac{vol(R_{n,k})}{2}| \leq 3$.
From Lemma~\ref{genmcut2},
 $\displaystyle 3\left(\frac{1}{vol(R_{n,k})/2}+\frac{1}{vol(R_{n,k})/2}\right)> 2\left(\frac{1}{vol(R_{n,k})/2+3}+\frac{1}{vol(R_{n,k})/2-3}\right)$ for $vol(R_{n,k}) \geq 11$.
Then we can show that there is no subset $A$ with $cut(A,V\setminus A) \geq 3$ and $Mcut(A,V\setminus A) \leq Mcut(A_4,V\setminus A_4)$.
This conclude that minimum Ncut always have cut value 2 for all cases which has cut size more than 1.
 \end{proof}
Figure~\ref{fig:rnk} shows the above regions for $n,k$.
For a given $R_{n,k}$, 
we can find $Mcut(R_{n,k})$.
\begin{figure}[htb]
\begin{center}
\includegraphics[scale=0.75]{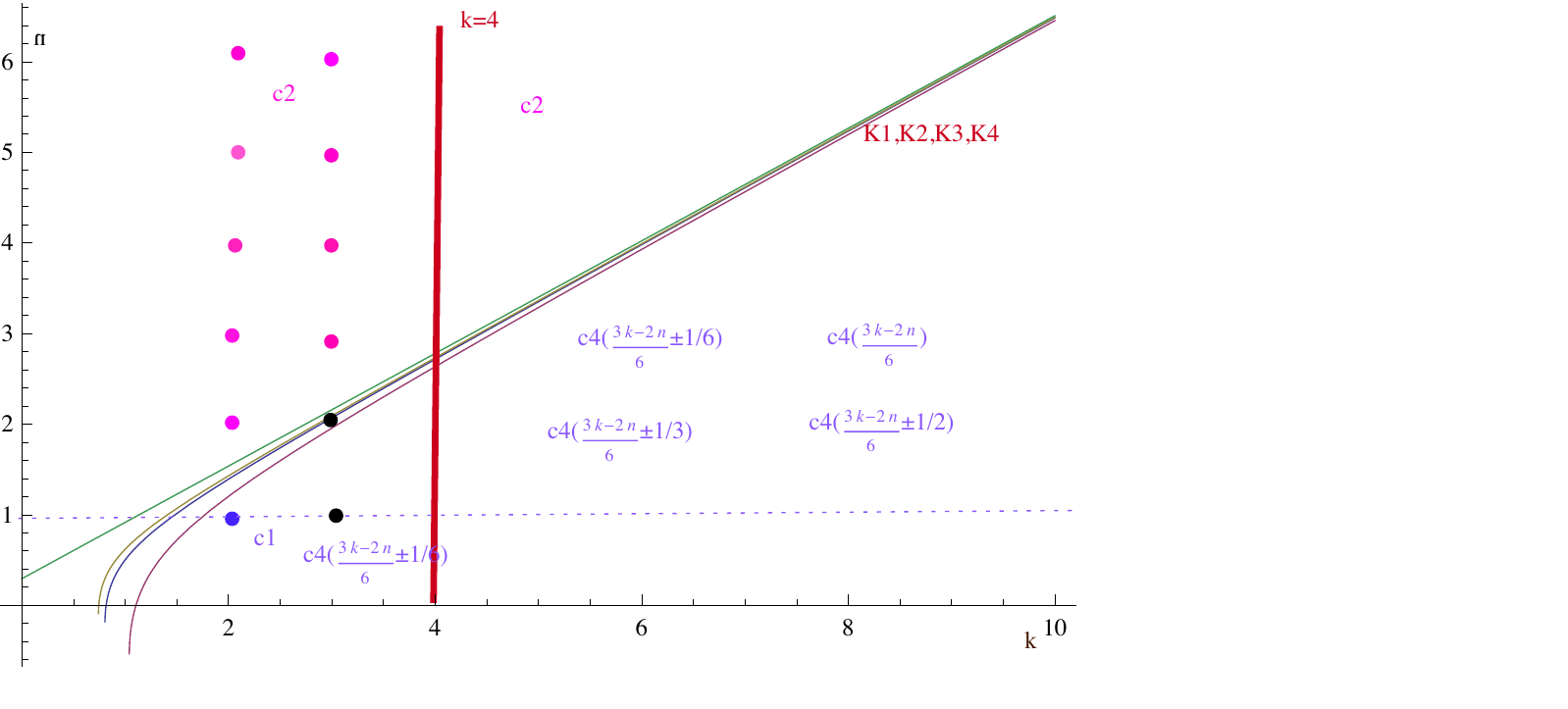}
\caption{$Mcut(R_{n,k})$.}
\label{fig:rnk}
\end{center}
\end{figure}


\subsection{$Mcut$ of weighted paths $P_{n,k}$}

In this section, 
we consider a weighted path graph $P_{n,k}$ and find a formula for $Mcut(P_{n,k})$ based on $n,k$. 
We consider subsets of $V(P_{n,k})$
defined by 
$\displaystyle A(\alpha) =\{ x_i \ | \ 1 \leq i \leq \alpha \}$
for $1 \le \alpha \le n+k-1$.
We note that every subset $A \subset V(P_{n,k})$ with
$cut(A,V\setminus A)=1$ is $A=A(\alpha)$ for some $\alpha$.

\begin{lemma}
Let $G=P_{n,k}$.
There exists a subset $A \subset V(P_{n,k})$ such that
$cut(A,V\setminus A)=1$ and $Mcut(G)=Ncut(A,V\setminus A)$.
\end{lemma}
\begin{proof}
Since $vol(P_{n,k})=2n+3k-2$, if $k \ge \frac{1}{3}(11-2n)$
then $vol(P_{n,k})\ge 9$.
By the Lemma~\ref{gencut1}, we have $Mcut(G)=Mcut_1(G)$.

If $k < \frac{1}{3}(11-2n)$, we have only five cases
$(n,k)=(1,1)$, $(2,1)$, $(3,1)$, $(1,2)$ and $(2,2)$.
For each cases
$Mcut(P_{1,1})=Ncut(A(1),V\setminus A(1))$,
$Mcut(P_{2,1})=Ncut(A(2),V\setminus A(2))$,
$Mcut(P_{3,1})=Ncut(A(2),V\setminus A(2))$,
$Mcut(P_{1,2})=Ncut(A(2),V\setminus A(2))$, and
$Mcut(P_{2,2})=Ncut(A(2),V\setminus A(2))$.
\end{proof}

Let $P_{n,k}$ ($k \ge \frac{1}{3}(11-2n)$) be a weighted path graph.
We first note that
\begin{eqnarray*}
vol(P_{n,k})&=&2n+3k-2,\\
vol(A(\alpha))&=& \begin{cases}
 2 \alpha -1 & \mbox{($\alpha \le n$)} \\
 3 \alpha - n -1 & \mbox{($n+1 \le \alpha$)}
\end{cases}, and \\
{\small Ncut(A(\alpha),V\setminus A(\alpha))} &=& c(\alpha),
\end{eqnarray*}
where a function $c(t)$ ($1 \le t \le n+k$) is defined by
$$
c(t)= \begin{cases}
\frac{2n+3k-2}{(2 t-1)(2n+3k-2 t-1)} &
\mbox{($1 \le t \le n+\frac{1}{2}$)} \\
\frac{2n+3k-2}{(-n+3 t-1)(3n+3k-3 t-1)} &
\mbox{($n+\frac{1}{2} < t \le n+k-1$).}
\end{cases}
$$
We note that
$c(\alpha-x)=c(\alpha+x)$
for an integer $\alpha$ ($1 \le \alpha \le n$, $n+1< \alpha \le n+k-1$)
and a real number $x$ ($0 \le x \le \frac{1}{2}$).

We also note $vol(A_i)<vol(A_{i+1})$ ($1 \le i \le n+k-2$),
$vol(A(n))=2n-1$, $vol(A(n+1))=2n+2$ and $vol(A(n+k-1))=2n+3k-4$.
Since
{\small
\begin{eqnarray*}
&&Ncut(A(\alpha),V\setminus A(\alpha)) \\
&=&
\frac{4 vol(P_{n,k})}
{(vol(P_{n,k}))^2-(vol(A(\alpha))-vol(V\setminus A(\alpha)))^2} \\
&=&
\frac{4 vol(P_{n,k})}
{(vol(P_{n,k}))^2-4(vol(A(\alpha))-\frac{1}{2}vol(P_{n,k}))^2},
\end{eqnarray*}
}
if $Mcut(P_{n,k})=Ncut(A(\alpha_0),V\setminus A(\alpha_0))$
then 
{\small
\begin{eqnarray*}
&& \vert  vol(A(\alpha_0))-\frac{1}{2}vol(P_{n,k}) \vert \\
&=& \min \{
\vert vol(A(\alpha))-\frac{1}{2}vol(P_{n,k}) \vert \ | \ 
1 \le \alpha \le n+k-1 \}.
\end{eqnarray*}
}
We consider four cases:
Case (i) $\frac{1}{2}vol(P_{n,k}) \le vol(A(n))$,
Case (ii) $ vol(A(n)) < \frac{1}{2}vol(P_{n,k})
 \le \frac{1}{2}(vol(A(n))+vol(A(n+1)))$,
Case (iii) $\frac{1}{2}(vol(A(n))+vol(A(n+1))) < 
\frac{1}{2}vol(P_{n,k}) \le vol(A(n+1))$, and
Case (iv) $ vol(A(n+1)) < \frac{1}{2}vol(P_{n,k})<vol(A(n))$.
\\

\noindent
{\bf Case (i)}
Assume $\frac{1}{2}vol(P_{n,k}) \le vol(A(n))$.
That is $k \le \frac{2}{3} n$.
We find $\alpha$ minimizing 
$\vert  vol(A(\alpha))-\frac{1}{2}vol(P_{n,k}) \vert$
$=\vert  2\alpha -1 -(n+\frac{3}{2}k -1) \vert$.
For such $\alpha$ we have
$$
(2\alpha-1)-1 <
n+\frac{3}{2}k-1 \le (2\alpha-1)+1.
$$
That is
$$
\alpha -\frac{1}{2} < \frac{2n+3k}{4} \le \alpha + \frac{1}{2}
$$
which means $\alpha$ is the nearest integer of $\frac{2n+3k}{4}$.

We consider three cases
($K \in \mathbf{Z}$),
($K \not\in \mathbf{Z}$ and $2 \mid k$), and
($2 \nmid k$), 
where $K=\frac{2n+3k}{4}$.

If $K \in \mathbf{Z}$ then $\alpha=K$.
If $2 \nmid k$ then $\alpha=K+\frac{1}{4}$ or
$\alpha=K-\frac{1}{4}$.
If $K \not\in \mathbf{Z}$ and $2 \mid k$ then
$\alpha=K+\frac{1}{2}$ or
$\alpha=K-\frac{1}{2}$.

Since $c(\alpha-x)=c(\alpha+x)$
for an integer $\alpha$ ($1 \le \alpha \le n $)
and a real number $x$ ($0 \le x \le \frac{1}{2}$),
$Mcut(P_{n,k})$
will be 
\begin{eqnarray*}
c(K)&=&\frac{4}{-2+3 k+2 n}, \\
c(K+\frac{1}{2})&=&\frac{4 (-2+3 k+2 n)}{(-4+3 k+2 n) (3 k+2 n)}, or\\
c(K+\frac{1}{4})&=&\frac{4 (-2+3 k+2 n)}{(-3+3 k+2 n) (-1+3 k+2 n)}.
\end{eqnarray*}
following the conditions of $n$ and $k$.

\noindent
{\bf Case (ii)}
Assume $ vol(A(n)) < \frac{1}{2}vol(P_{n,k})
 \le \frac{1}{2}(vol(A(n))+vol(A(n+1)))$.
That is $\frac{2}{3}n < k \le \frac{2}{3} n + 1$.
In this case
$$
Mcut(P_{n,k})=c(n)
=\frac{3k+2n-2}{(3k-1)(2n-1)}.
$$

\noindent
{\bf Case (iii)}
Assume $\frac{1}{2}(vol(A(n))+vol(A(n+1))) < 
\frac{1}{2}vol(P_{n,k}) \le vol(A(n+1))$.
That is $\frac{2}{3}n +1 < k \le \frac{2}{3} n+2$.
In this case 
$$
Mcut(P_{n,k})=c(n+1)
=\frac{2-3k-2n}{8-6k+8n-6kn}.
$$

\noindent
{\bf Case (iv)}
Assume $ vol(A(n+1)) < \frac{1}{2}vol(P_{n,k})<vol(A(n))$.
That is $\frac{2}{3} n+2 < k$.
We find $\alpha$ minimizing 
$\vert  vol(A(\alpha))-\frac{1}{2}vol(P_{n,k}) \vert$
$=\vert  3\alpha -n -1 -(n+\frac{3}{2}k -1) \vert$.
For such $\alpha$ we have
$$
(3\alpha-n-1)-\frac{3}{2} <
n+\frac{3}{2}k-1 \le (3\alpha-n-1)+\frac{3}{2}.
$$
That is
$$
\alpha -\frac{1}{2} < \frac{4n+3k}{6} \le \alpha + \frac{1}{2}
$$
which means $\alpha$ is the nearest integer of $\frac{4n+3k}{6}$.

We consider four cases
($K' \in \mathbf{Z}$),
($3 \nmid n$ and $2 \mid k$),
($3 \mid n$ and $2 \nmid k$), and
($3 \nmid n$ and $2 \nmid k$),
where $K'=\frac{4n+3k}{6}$.

If $K' \in \mathbf{Z}$ then $\alpha=K'$.
If $3 \nmid n$ and $2 \mid k$ then
$\alpha=K'+\frac{1}{3}$ or
$\alpha=K'-\frac{1}{3}$.
If $3 \mid n$ and $2 \nmid k$ then
$\alpha=K'+\frac{1}{2}$ or
$\alpha=K'-\frac{1}{2}$.
If $3 \nmid n$ and $2 \nmid k$ then
$\alpha=K'+\frac{1}{6}$ or
$\alpha=K'-\frac{1}{6}$.
Since $c(K'-x)$ $=c(K'+x)$ ($x=\frac{1}{2},\frac{1}{3},\frac{1}{6}$),
we have $Mcut(P_{n,k})$ as one of
\begin{eqnarray*}
c(K')&=&\frac{4}{-2+3 k+2 n},\\
c(K'+\frac{1}{2})&=&\frac{4 (-2+3 k+2 n)}{(-5+3 k+2 n) (1+3 k+2 n)},\\
c(K'+\frac{1}{3})&=&\frac{4 (-2+3 k+2 n)}{(-4+3 k+2 n) (3 k+2 n)}, or\\
c(K'+\frac{1}{6})&=&\frac{4 (-2+3 k+2 n)}{(-3+3 k+2 n) (-1+3 k+2 n)}
\end{eqnarray*}
following the conditions of $n$ and $k$.

We note
$c(K)=c(K')$,
$c(K+\frac{1}{2})=c(K'+\frac{1}{3})$ and
$c(K+\frac{1}{4})=c(K'+\frac{1}{6})$ before
summarizing them as a proposition.

\begin{theorem}
\label{pathnk}
For $P_{n,k}$,$(n,k \geq 1$, $k \ge \frac{1}{3}(11-2n))$,
 $Mcut(P_{n,k})$ is given by

\[ \left \{ \begin{array}{cc}
\frac{4}{-2+3 k+2 n} &   \mbox{$(((3 \mid n) \wedge (2 \mid k) \wedge (R_3 < k))\vee$}\\
&\mbox{$(o_1\wedge (k \le R_1))$}\\
\\
\frac{4 (-2+3 k+2 n)}{(-5+3 k+2 n) (1+3 k+2 n)} &
\mbox{$((3 \mid n) \wedge  (2 \nmid k)\wedge (R_3 < k))$, }\\
\\
\frac{4 (-2+3 k+2 n)}{(-4+3 k+2 n) (3 k+2 n)}&
\mbox{$((3 \nmid n)\wedge  (2 \mid k)\wedge (R_3 < k))\vee$}\\
&\mbox{$(o_2\wedge (2 \mid k)\wedge (k \le R_1))$,}\\
\\
\frac{4 (-2+3 k+2 n)}{(-3+3 k+2 n) (-1+3 k+2 n)} &
\mbox{$((3 \nmid n)\wedge  (2 \nmid k)\wedge (R_3 < k))\vee$}\\
&\mbox{$((2 \nmid k) \wedge (k \le R_1))$}\\
\\
\frac{3k+2n-2}{(3k-1)(2n-1)} &
\mbox{$(R_1 < k \le R_2)$}\\
\\
\frac{2-3k-2n}{8-6k+8n-6kn}&
\mbox{$(R_2 < k \le R_3)$},\\
\end{array}\right. \]
where
\begin{eqnarray*}
o_1 &=& (\frac{3k+2n}{4} \in \mathbf{Z}),\\
o_2 &=& (\frac{3k+2n}{4} \not\in \mathbf{Z}), \\
R_1&=& \frac{2n}{3},\\
R_2 &=& \frac{2n}{3}+1,\\
R_3 &=&\frac{2n}{3}+2.
\end{eqnarray*}
\end{theorem}
\hfill\qed

Figure~\ref{fig:patha} shows minimum $Mcut(P_{n,k})$ for each $n,k$.

\begin{figure}[htb]
\begin{center}
\includegraphics[scale=0.7]{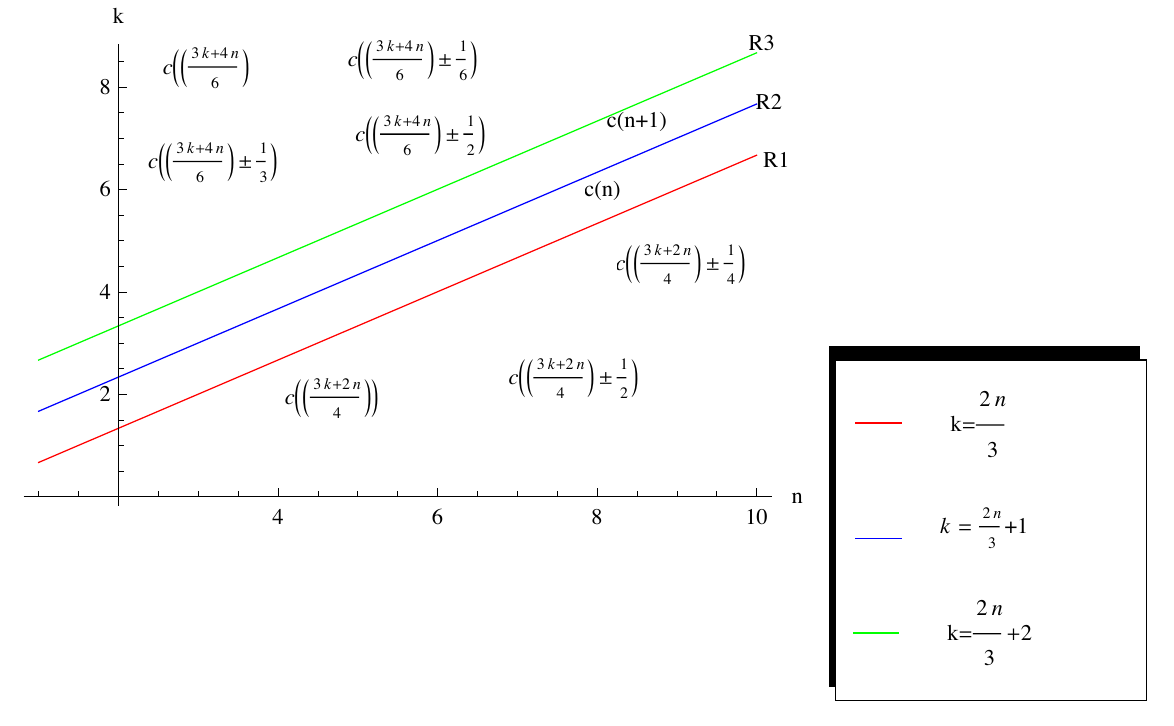}
\caption{$Mcut(P_{n,k})$.}
\label{fig:patha}
\end{center}
\end{figure}

\begin{corollary}
For $P_{2k,k}$,
 \[ Mcut(P_{2k,k})= \left \{ \begin{array}{cc}
\frac{4}{-2+7k} & \mbox{$(4 \mid k) $,}\\ 
\\
\frac{4 (-2+7 k)}{(-4+7k) (7k)} & \mbox{$(4 \nmid k)\wedge (2 \mid k)$,}\\ 
\\
\frac{4 (-2+7 k)}{(-3+7k) (-1+7k)} & \mbox{$(2 \nmid k)$.}\\ 
\end{array} \right. \]
\end{corollary}

\begin{proof}
By substituting $n=2k$ to the formula given for $Mcut(P_{n,k})$,
we can directly obtain the result.
According to the Theorem~\ref{pathnk},
 for $n=2k$,
$k>R_3$ that is $\displaystyle k > \frac{2n}{3}+2$ implies that $k \leq -6$.
Since $k \geq 1$,
this does not holds.
For $R_1 <k \le R_2$ that is 
$\displaystyle  \frac{2n}{3} < k \le \frac{2n}{3}+1$ implies that $ -3 \le k < 0$.
Since $k \geq 1$, 
this does not holds.
For $R_2 <k \le R_3$ that is 
$\displaystyle  \frac{2n}{3}+1 < k \le \frac{2n}{3}+2$ implies that $ -6 \le k < -3$.
Since $k \geq 1$, 
this does not holds.
Therefore the only case,
 which holds for $n=2k$ is,
 $k \le R_1$ that is 
$\displaystyle k  \le \frac{4k}{3}$.
This implies that $k \ge 0$.
Substituting $n=2k$ in the Theorem~\ref{pathnk},
we have,
\[ Mcut(P_{2k,k})= \left\{ \begin{array}{cc}
\frac{4}{-2+7k} & \mbox{ $(4 \mid k)$,}\\ 
\frac{4 (-2+7 k)}{(-4+7k) (7k)} & \mbox{$(4 \nmid k)\wedge (2 \mid k)$,}\\ 
\frac{4 (-2+7 k)}{(-3+7k) (-1+7k)} & \mbox{$(2 \nmid k)$.}
\end{array} \right. \] 
\end{proof}


\subsection{$Mcut$ of graph $LP_{n,m}$}
Here,
we consider lollipop graph $LP_{n,m}$
and derive a formula for $Mcut(LP_{n,m})$.
A lollipop graph $LP_{n,m}$ defined in Definition~\ref{def:lp}
is constructed by joining an end vertex of
a path graph $P_m$ to a vertex of a complete graph $K_n$.

We consider three kinds of subsets of $V(LP_{n,m})$
defined by 
$A_1(\alpha)=\{x_i\ |\ 1 \le i \le \alpha\}$
for $1\le \alpha \le m$,
$A_2(\beta)=\{x_i\ |\ 1 \le i \le m\}$
$\cup \{y_i\ |\ 1 \le i \le \beta\}$
for $1 \le \beta < n$,
and,
$B(\alpha,\beta)=\{x_i\ |\ 1 \le i \le \alpha\}$
$\cup \{x_m\}$
$\cup \{y_i \ |\ 1 \le i \le \beta\}$
for $1 \le \alpha < m-1$, $1 \le \beta < n$.

\begin{lemma}
Let $A$ be a subset of $V(LP_{n,m})$.
\begin{enumerate}
\item If $y_i \in A$ and $y_{i+1} \not\in A$ for some $i$ ($2 \le i \le n-1$)
then $Ncut(A',V\setminus A')$ $=Ncut(A,V\setminus A)$,
where
$A'=(A\setminus\{y_i\})\cup\{y_{i+1}\}$.
\item If $x_i \in A$,
$x_{i+1},\ldots,x_{j} \not\in A$,
and $x_{j+1} \in A$ for some $i,j$ ($1 \le i < j \le m-1$)
then $Ncut(A',V\setminus A')$
$\le Ncut(A,V\setminus A)$,
where
$A'=(A\setminus\{x_{j+1}\})\cup\{x_{i+1}\}$.
\item There exists a subset
$A_1(\alpha)$, $A_2(\beta)$ or $B(\alpha,\beta)$ 
such that 
$Mcut(LP_{n,m})=Ncut(A_1(\alpha),V\setminus A_1(\alpha))$,
$Mcut(LP_{n,m})=Ncut(A_2(\beta),V\setminus A_2(\beta))$, or
$Mcut(LP_{n,m})=Ncut(B(\alpha,\beta),V\setminus B(\alpha,\beta))$.
\end{enumerate}
\end{lemma}
\begin{proof}
\begin{enumerate}
\item It is easy to check
$vol(A)=vol(A')$ and $cut(A,V\setminus A)=cut(A',V\setminus A')$.
\item 
It is easy to check $vol(A)=vol(A')$ and
$cut(A',V\setminus A') \le cut(A,V\setminus A)$.
\item Let $A$ be a subset of $V(LP_{n,m})$ such that
$Mcut(LP_{n,m})=Ncut(A,V\setminus A)$.
Using the above results 1. and 2., we have a subset $A'$
which is one of $A_1(\alpha)$, $A_2(\alpha)$ or $B(\alpha,\beta)$
such that $Ncut(A',V\setminus A')=Mcut(LP_{n,m})$.
\end{enumerate}
\end{proof}

Let $G=LP_{n,m}$ ($n \ge 3$, $m \ge 1$) a lollipop graph.
We first note that
{\small
\begin{eqnarray*}
vol(LP_{n,m})&=&2m+n(n-1), \\
vol(A_1(\alpha))&=&2\alpha -1, \\
cut(A_1(\alpha),V\setminus A_1(\alpha))&=& 1, \\
vol(A_2(\beta))&=&2m+\beta(n-1), \\
cut(A_2(\beta),V\setminus A_2(\beta))&=& \beta(n-\beta), \\
vol(B(\alpha,\beta))&=& 2\alpha + 2+\beta(n-1), \\
cut(B(\alpha,\beta,V\setminus B(\alpha,\beta)))
&=& \beta(n-\beta) + 2, and \\
Ncut(A_1(\alpha),V\setminus A_1(\alpha)) 
&=& c(\alpha),
\end{eqnarray*}}
where a function $c(t)$ ($1 \le t \le m$) is defined by
$$
c(t)=\frac{2 m-n+n^2}{\left(1+2 m-n+n^2-2 t\right) (-1+2 t)}.
$$
It is also showed that
{\small \begin{eqnarray*}
&&Ncut(A_2(\beta),V\setminus A_2(\beta))\\
&=&
\frac{\beta \left(2 m-n+n^2\right)}{(-1+n) (-\beta+2 m+n \beta)}, and \\
&&Ncut (B(\alpha,\beta),V\setminus B(\alpha,\beta))\\
&=&
\frac{\left(2 m-n+n^2\right) \left(2+n \beta -\beta ^2\right)}{(2+2 \alpha -\beta +n \beta ) \left(2 m-n+n^2-2 \alpha +\beta -n \beta - 2 \right)}.
\end{eqnarray*}
}

\begin{lemma}
\label{lemma:lp}
Let $G=LP_{n,m}$ $(n \ge 3$, $m \ge 2)$.
\begin{enumerate}
\item $c(\alpha-1)) < c(\alpha)$ iff
$m > \frac{1}{2}(n^2-n+4)$ \\
$(2 \le \alpha \le m)$.
\item $c(m) \le \frac{1}{2}vol(LP_{n,m})$ iff
$m \le \frac{1}{2}(n^2-n+4)$.
\item $c(m) \le Ncut(A_2(\beta),V\setminus A_2(\beta))$ $(1 \le \beta < n)$.
\item If $m \le \frac{1}{2}(n^2-n+2)$ then 
$$
c(m) \le Ncut(B(\alpha,\beta),V\setminus B(\alpha,\beta)),
$$
$(1 \le \alpha \le m-2$, $1 \le \beta < n)$.
\end{enumerate}
\end{lemma}
\begin{proof}
Each items are given by straightforward computations.
\end{proof}

Since $cut(A_1(\alpha),V\setminus A_1(\alpha))=1$,
if $vol(A_1(m)) \ge \frac{1}{2}vol(LP_{n,m})$ then
there exists some $\alpha$ such that
$Mcut(LP_{n,m})=Ncut(A_1(\alpha),V\setminus A_1(\alpha))$.
To find the $\alpha$,
we solve
$$
vol(A_1(\alpha))-1 < \frac{1}{2}vol(LP_{n,m}) \le vol(A_1(\alpha))+1.
$$
That is
$$
\alpha - \frac{1}{2} < \frac{n^2-n+2m+2}{4} \le \alpha + \frac{1}{2}
$$
which means $\alpha$ is the nearest integer of $\frac{n^2-n+2m+2}{4}$.
We consider two cases
$(K \in \mathbf{Z})$ and $(K \not\in \mathbf{Z})$,
where
$K=\frac{n^2-n+2m+2}{4}$.
If $K \in \mathbf{Z}$ then $\alpha=K$.
If $K \not\in \mathbf{Z}$ then
$K+\frac{1}{2}$ is an integer and $\alpha=K+\frac{1}{2}$
or $\alpha=K-\frac{1}{2}$.
Since $c(K+\frac{1}{2})=c(K-\frac{1}{2})$,
$Mcut(LP_{n,m})$ will be
{\small
\begin{eqnarray*}
c(K)&=&\frac{4}{n^2-n+2m}, \ or \\
c(K+\frac{1}{2})&=&
\frac{4(n^2-n+2m)}{(n(n-1)+2(m-1))(n(n-1)+2(m+1))}
\end{eqnarray*}
}

By Lemma~\ref{lemma:lp}, if 
$m \le \frac{1}{2}(n^2-n+4)$
then
$Mcut(LP_{n,m})=Ncut(A_1(m),V\setminus A_1(m))$.
That is
$$
Ncut(A_1(m),V\setminus A_1(m))=\frac{n^2-n+2m}{(2m-1)(n^2-n+1)}.
$$
If $m=1$ then
it is easy to verify
$Mcut(LP_{n,1})=Ncut(A_2(1),V\setminus A_2(1))$
$=\frac{n^2-n+2}{(n+1)(n-1)}$.
\begin{theorem}
\label{prop316}
For the graph $LP_{n,m},(n \geq 3$ and $m \geq 1)$,
$Mcut(LP_{n,m})$ is given by,
\[\left \{ \begin{array}{cc}
 \frac{n^2-n+2m}{(2m-1)(n^2-n+1)} & \mbox{$(2\leq  m \leq \frac{n^2-n+4}{2})$,} \\
\frac{4}{(n^2-n+2m)}& \mbox{$(o_1 \wedge m >  \frac{n^2-n+4}{2})$,}\\ 
\frac{4(n^2-n+2m)}{(n(n-1)+2(m-1))(n(n-1)+2(m+1))}
& \mbox{$(o_2 \wedge m >  \frac{n^2-n+4}{2})$,}\\
\frac{n^2-n+2}{(n+1)(n-1)} & \mbox{$(m=1)$,}
\end{array} \right. \]
where
\begin{eqnarray*}
o_1 &=& (\frac{n^2-n+2m+2}{4} \in \mathbf{Z}), \\
o_2 &=& (\frac{n^2-n+2m+2}{4} \not\in \mathbf{Z}).
\end{eqnarray*}
\end{theorem}
\hfill\qed


\section{Eigenvalues and eigenvectors of paths and cycles}
In this section,
we derive formulae for the eigenvalues and eigenvectors of cycles and paths
using circulant matrices and give an alternate proof for the eigenvalues
of adjacency matrix of cycles and paths using Chebyshev polynomials.

\subsection{Circulant matrices and eigenvalues of cycles and paths}

Let $\omega_n=\mathbf{e}^{-\frac{2\pi}{n}\mathbf{i}}$
$=\cos\frac{2\pi}{n}+\mathbf{i}\sin\frac{2\pi}{n}$
be a primitive $n$-th root of unity.

\begin{definition}
A circulant matrix $C=(c_{ij})$ is 
a matrix having a form $c_{ij}=c_{(j-i)\mod n}$.
$$
C=\left(
\begin{array}{cccccc}
c_0     & c_1     & c_2    & \cdots & \cdots & c_{n-1} \\
c_{n-1} & c_0     & c_1    & c_2    &        & c_{n-2} \\
\vdots  & c_{n-1} & c_0    & c_1    &        & \vdots \\
\vdots  &         & \ddots & \ddots & \ddots &\vdots \\
\vdots  &         &        & \ddots & \ddots & c_1 \\
c_1     & c_2     & \cdots & \cdots & c_{n-1}& c_0
\end{array}
\right).
$$
\end{definition}

\begin{proposition}
Let $C=(c_{ij})$ be a circulant matrix and $c_{ij}=c_{(j-i)\mod n}$.
For $k=0,\ldots,n-1$, 
we have
$$
C\mathbf{u}_k = \lambda_k \mathbf{u}_k,
$$
where $\displaystyle \lambda_k=\sum_{j=0}^{n-1} c_j (\omega_n^k)^j$,
$\mathbf{u}_k=(u_{ki})$ and $u_{ki}=(\omega_n^k)^i= \cos\frac{2k\pi i}{n}+\mathbf{i}\sin\frac{2k\pi i}{n}$.
\end{proposition}
\begin{proof}
\begin{eqnarray*}
(C \mathbf{u}_k)_i &= & \sum_{j=0}^{n-1} c_{ij} u_{kj} \\
&=& \sum_{j=0}^{n-1} c_{(j-i)\mod n} (\omega_n^k)^j \\
&=& (\omega_n^k)^i \sum_{j=0}^{n-1} c_{(j-i)\mod n} (\omega_n^k)^{j-i} \\
&=& (\omega_n^k)^i \sum_{j=0}^{n-1} c_j (\omega_n^k)^j \\
&=& \lambda_k u_{ki} \\
&=& (\lambda_k \mathbf{u}_k)_i.
\end{eqnarray*}\end{proof}

\begin{proposition}
\begin{enumerate}
\item The eigenvalues of the adjacency matrix of $C_n$ is given by $\displaystyle \lambda_k = 2 \cos (\frac{2k\pi}{n})$,
\item The eigenvalues of the difference Laplacian matrix of $C_n$ is given by $\displaystyle \lambda_k = 2-2 \cos (\frac{2k\pi}{n})$,
\item The eigenvalues of the normalized Laplacian matrix of $C_n$ is given by $\displaystyle \lambda_k = 1- \cos (\frac{2k\pi}{n})$,  and
\item The eigenvalues of the signless Laplacian matrix of $C_n$ is given by $\displaystyle \lambda_k = 2+2 \cos (\frac{2k\pi}{n})$,\\
where $k=(0,\ldots,n-1)$.
\end{enumerate}
\end{proposition}

\begin{proof}
\begin{enumerate}
\item Let $A$ be an adjacency matrix of a cycle graph with $n$ vertices.
That is $A=(c_{ij})=c_{(j-i)\mod n}$ and $c_0=0$, $c_1=c_{n-1}=1$
and $c_i=0$ for $i=2,\ldots,n-2$.
\begin{eqnarray*}
\lambda_k &=& (\omega_n^k)^1 + (\omega_n^k)^{n-1} \\
&=& (\omega_n^k)^1 + (\omega_n^k)^{-1} \\
&=& 2 \cos (\frac{2k\pi }{n}).
\end{eqnarray*}\hfill\qed
\item Let $L(C_n)$ be the Laplacian matrix of a cycle graph with $n$ vertices.
That is $L(C_n)=(c_{ij})=c_{(j-i)\mod n}$ and $c_0=2$, $c_1=c_{n-1}=-1$
and $c_i=0$ for $i=2,\ldots,n-2$.
\begin{eqnarray*}
\lambda_k &=& 2 - (\omega_n^k)^1 - (\omega_n^k)^{n-1} \\
&=& 2 - ((\omega_n^k)^1 + (\omega_n^k)^{-1}) \\
&=& 2 - 2 \cos (\frac{2k\pi}{n}).
\end{eqnarray*}\hfill\qed
\item Let ${\cal L}(C_n)$ be the normalized 
Laplacian matrix of a cycle graph with $n$ vertices.
That is ${\cal L}(C_n)=(c_{ij})=c_{(j-i)\mod n}$ and $c_0=1$,
$c_1=c_{n-1}=-\frac{1}{2}$ and $c_i=0$ for $i=2,\ldots,n-2$.
\begin{eqnarray*}
\lambda_k &=& 1 - \frac{1}{2} (\omega_n^k)^1 - \frac{1}{2} (\omega_n^k)^{n-1} \\
&=& 1 - \frac{1}{2}( (\omega_n^k)^1 + (\omega_n^k)^{-1}) \\
&=& 1 - \cos (\frac{2k\pi}{n}).
\end{eqnarray*}\hfill\qed

\item 
Let $SL(C_n)$ be the signless Laplacian matrix of a cycle graph with $n$ vertices.
That is $SL(C_n)=(c_{ij})=c_{(j-i)\mod n}$ and $c_0=2$, $c_1=c_{n-1}=1$
and $c_i=0$ for $i=2,\ldots,n-2$.
\begin{eqnarray*}
\lambda_k &=& 2 + (\omega_n^k)^1 + (\omega_n^k)^{n-1} \\
&=& 2 + ((\omega_n^k)^1 + (\omega_n^k)^{-1}) \\
&=& 2 + 2 \cos (\frac{2k\pi}{n}).\hfill\qed
\end{eqnarray*}
\end{enumerate}
\end{proof}

\begin{proposition}
\label{propcycle}
Let $\lambda_k(0\leq k \leq n-1)$ be the $k^{th}$ eigenvalue of an adjacency matrix of $C_n$.
Then $\lambda_k =\lambda_{n-k}$,
for $k=1,\ldots,n-1$.
\end{proposition}
\begin{proof}
Eigenvalues of an adjacency matrix of cycle is given by $\displaystyle \lambda_k= 2\cos(\frac{2k\pi}{n})$,
where $k=0,\ldots,n-1$.
\begin{eqnarray*}
\lambda_0 &=& 2,\\
\lambda_1 &=& 2\cos (\frac{2\pi}{n}),\\
\lambda_2 &=& 2\cos (\frac{4\pi}{n}),\\
& \vdots &\\
\lambda_{n-2} &=& 2\cos (\frac{4\pi}{n}),  \\
\lambda_{n-1} &=& 2\cos (\frac{2\pi}{n}). 
\end{eqnarray*}
This shows that $\lambda_k =\lambda_{n-k}$ for $k=1,\ldots,n-1$. 
\end{proof}

\begin{proposition}
\label{prop:path}
\begin{enumerate}
\item The eigenvalues of an adjacency matrix of a path graph $P_n$ are given by $\displaystyle \lambda_k(A(P_n))= 2\cos(\frac{(k+1) \pi}{n+1}),(k=0, \ldots, n-1)$ and an eigenvector $\mathbf{u_k}$ is given by $\displaystyle (u_{ki})=\sin \frac{(i+1)(k+1) \pi}{n+1},(i=0,\ldots,n-1)$ and $(k=0,\ldots,n-1)$.
\item The eigenvalues of difference Laplacian matrix of $P_n$ are given by $\displaystyle \lambda_k(L(P_n))=2-2\cos(\frac{k \pi}{n})$,
$(k=0, \ldots, n-1)$ and its eigenvector $\mathbf{u_k}$ is given by $\displaystyle (u_{ki})=\cos\left(\frac{(2i+1)k \pi}{2n}\right)(i=0,\ldots,n-1)$.
\item The eigenvalues of normalized Laplacian matrix of a path $P_n$ are given by $\displaystyle \lambda_k(\mathcal{L}(P_n))=1-\cos(\frac{k \pi}{n-1})$
$(k=0, \ldots, n-1)$ and its eigenvector $\mathbf{u_k}$ is given by 
\[ u_{ki}=\left\{\begin{array}{cc}
 \sqrt{2} \cos\left(\frac{2 i \pi k}{2n-2}\right) & \mbox{$i=1,\ldots,n-2$,}\\
 \cos(\frac{2 i \pi k}{2n-2}) & \mbox{$i=0$ and $i=n-1$.}
 \end{array} \right. \]
\item The eigenvalues of signless Laplacian matrix of $P_n$ are given by $\displaystyle \lambda_k(SL(P_n))=2+2\cos(\frac{(k+1) \pi}{n})$,
$(k=0, \ldots, n-1)$ and its eigenvector $\mathbf{u_k}$ is given by $\displaystyle (u_{ki})= \sin(\frac{(2i+1)k \pi}{2n}),(i=0,\ldots,n-1)$.
\end{enumerate}
\end{proposition}
\begin{proof}
\begin{enumerate}
\item 
Let $\mathbf{u}=(u_i), (i=0,\ldots,n-1)$ be an eigenvector for an eigenvalue $\lambda$ of path $P_n$.
Then, we can write 
$$
P_n \mathbf{u} = 
\left(
\begin{array}{cccccc}
0   & 1    &0    & \cdots & \cdots & 0 \\
1 & 0    & 1    & 0   &        & \vdots \\
0 & 1 & 0    & 1    &        & \vdots \\
\vdots  &         & \ddots & \ddots & \ddots &\vdots \\
\vdots  &         &        & \ddots & \ddots & 1 \\
0    & 0     & \cdots & \cdots & 1& 0
\end{array}
\right)
\left(
\begin{array}{c} u_0 \\ u_1\\ u_2\\ \vdots\\\vdots\\u_{n-1} 
\end{array}
\right)
= \lambda \mathbf{u}.
$$
Then we have the following equations:
\begin{eqnarray}
\label{eqcy1}
u_1 &=& \lambda u_0, \nonumber\\
u_0+u_2&=&\lambda u_1, \nonumber\\
u_1+u_3 &=& \lambda u_2,\\
& \vdots &  \nonumber\\
u_{n-2}&=&\lambda u_{n-1}.\nonumber
\end{eqnarray}

Let $\mathbf{u}'=(u'_i), (i=0,\ldots,2n+1)$ be an eigenvector of $C_{2n+2}$,
where $\displaystyle (u'_i)=\sin \left(\frac{2(i+1)(k+1)\pi}{2n+2}\right), (i=0,\ldots,2n+1)$ and $(k=0,\ldots,2n+1)$.
The eigenvalues of an adjacency matrix of a cycle $C_{2n+2}$ are $\displaystyle \lambda_k=2\cos\left(\frac{(k+1) \pi}{n+1}\right),(k=0,\ldots,2n+1)$.
We note that $u'_n=u'_{2n+1}=0$.
Hence we can write the equation $C_{2n+2} \mathbf{u}' =\lambda_k \mathbf{u}'$ as

$$
\left(
\begin{array}{cccccccc}
 0 & 1 & 0 & \cdots & \cdots &\cdots& \cdots & 1 \\
 1 & 0 & 1 & \cdots & \cdots &\cdots&\cdots  &0 \\
0 & \ddots & \ddots& \ddots & &   & &\vdots \\
 \vdots &  &  1 & 0 & 1 &   & &\vdots\\
 \vdots &  &   &  1 & 0 & 1  & &\vdots\\
  \vdots &&      &  & \ddots & \ddots & \ddots & \vdots \\
 \vdots &  &  &    & &1 & 0 & 1  \\
 1 & 0 & \cdots & \cdots & \cdots & \cdots & 1 & 0
\end{array}
\right)
\left(
\begin{array}{c} u'_0 \\  \vdots \\ u'_{n-1} \\ 0 \\ u'_{n+1} \\ \vdots \\u'_{2n} \\ 0
\end{array}
\right)$$\\
$= \lambda_k \mathbf{u}'$.

Then we have the following equations:
\begin{eqnarray}
\label{eqcy2}
u'_1 &=& \lambda_k u'_0, \nonumber \\
u'_0+u'_2 &=& \lambda_k u'_1, \nonumber \\
 & \vdots & \\
u'_{n-2} &=& \lambda_k u'_{n-1} \nonumber.
\end{eqnarray}

Comparing Equation~\ref{eqcy1} with Equation~\ref{eqcy2},
we have $P_n \mathbf{u}=\lambda_k \mathbf{u}$,
 where $\mathbf{u}=(u'_i) (i=0,\ldots,n-1)$.
That is $\lambda_k, (k=0,\ldots,n-1)$ are eigenvalues of $P_n$ and $\mathbf{u}$ is an eigenvector of $\lambda_k$.
Since $\lambda_i \neq \lambda_j$ for ($ i \neq j$ and  $0 \leq i,j \leq n-1)$, 
we have $n$ different eigenvalues of $P_n$ and that is the complete set of eigenvalues of $P_n$.

\item 
Let $\mathbf{u}=(u_i), (i=0,\ldots,n-1)$ be an eigenvector for an eigenvalue $\lambda$ of difference Laplacian matrix $L(P_n)$.
Then we can write the equation $ L(P_n) \mathbf{u} = \lambda \mathbf{u}$ as 
\[\left(
\begin{array}{cccccc}
1  & -1    & 0    & \cdots & \cdots & 0 \\
-1 & 2     & -1   & 0   &        & 0 \\
0 & -1 & 2   & -1    &        & \vdots \\
\vdots  &         & \ddots & \ddots & \ddots &\vdots \\
\vdots  &         &        & -1 & 2 & -1 \\
0   & 0    & \cdots & \cdots & -1& 1
\end{array}
\right)
\left(
\begin{array}{c} u_0 \\ u_1 \\ \vdots\\ \vdots\\ u_{n-2}\\ u_{n-1}
\end{array}
\right)
= \lambda \mathbf{u}. \]

Then we have the following equations.
\begin{eqnarray}
\label{eqcy3}
u_0 - u_1 &=& \lambda u_0, \nonumber\\
-u_0 + 2u_1 - u_2 &=& \lambda u_1, \nonumber\\
& \vdots & \\
-u_{n-2} + u_{n-1} &=& \lambda u_{n-1}.\nonumber
\end{eqnarray}

Let $\mathbf{u}'=(u'_i), (i=0,\ldots,2n-1)$ be an eigenvector of difference Laplacian matrix of $C_{2n}$,
where $\displaystyle ({u}'_i)=\cos \left(\frac{(2i+1)k\pi}{2n}\right), (i=0,\ldots,2n-1)$ and $(k=0,\ldots,2n-1)$.
The eigenvalues of $L(C_{2n})$ are $\displaystyle \lambda_k=2-2\cos(\frac{k \pi}{n}),(k=0,\ldots,2n-1)$.
We note that $u'_0=u'_{2n-1}, u'_1=u'_{2n-2},\ldots, u'_{n-1}=u'_n$.

Then we can write the equation $L(C_{2n}) \mathbf{u}'=\lambda_k \mathbf{u}'$ as
$$
\left(
\begin{array}{cccccc}
2  & -1    & 0    & \cdots & \cdots & -1 \\
-1 & 2     & -1   & 0   &        & 0 \\
0 & -1 & 2   & -1    &        & \vdots \\
\vdots  &         & \ddots & \ddots & \ddots &\vdots \\
\vdots  &         &        & -1 & 2 & -1 \\
-1  & 0    & \cdots & \cdots & -1& 2
\end{array}
\right)
\left(
\begin{array}{c} u'_0 \\ u'_1 \\ \vdots  \\\vdots\\ u'_{2n-2} \\ u'_{2n-1}
\end{array}
\right)
= \lambda_k \mathbf{u}'.
$$

\begin{eqnarray}
\label{eqcy4}
2u'_0 -u'_1 -u'_{2n-1} &=& u'_0 - u'_1 = \lambda_k u'_0,  \nonumber\\
-u'_0 + 2u'_1 - u'_2 &=& \lambda_k u'_1,  \nonumber\\
& \vdots & \\
-u'_{n-2}+ 2u'_{n-1} - u'_{n} &=& -u'_{n-2} + u'_{n-1} = \lambda_k u'_{n-1}. \nonumber
\end{eqnarray}

Comparing Equation~\ref{eqcy3} and Equation~\ref{eqcy4},
we have $P_n \mathbf{u}=\lambda_k \mathbf{u}$,
where $\mathbf{u}=(u'_i), (i=0,\ldots,n-1)$.
That is $\lambda_k, (k=0,\ldots,n-1)$ are eigenvalues of $P_n$ and $\mathbf{u}$ is an eigenvector of $\lambda_k$.
Since $\lambda_i \neq \lambda_j$ for ($ i \neq j$ and  $0 \leq i,j \leq n-1$), 
we have $n$ different eigenvalues of $P_n$ and that is the complete set of eigenvalues of $P_n$.
\item
Let $\mathbf{u}=(u_i), (i=0,\ldots,n-1)$ be an eigenvector for an eigenvalue $\lambda$ of normalized Laplacian matrix of path $P_n$.
Then we can write the equation $ \mathcal{L}(P_n) \mathbf{u} = \lambda \mathbf{u}$ as 
$$
\left(
\begin{array}{cccccc}
1 & -\frac{1}{\sqrt{2}} & 0 &\cdots & \cdots & 0 \\
-\frac{1}{\sqrt{2}}& 1 & -\frac{1}{2} & 0 & &\vdots \\
0& \ddots &\ddots& \ddots & & \vdots\\
\vdots&  &\ddots& \ddots & \ddots  & \vdots\\
\vdots &  & & -\frac{1}{2} &  1 & -\frac{1}{\sqrt{2}} \\
\vdots & \cdots&\cdots & \cdots & -\frac{1}{\sqrt{2}} & 1
\end{array}
\right)
\left(
\begin{array}{c} u_0 \\ u_1 \\ \vdots\\ \vdots \\ u_{n-2} \\ u_{n-1}
\end{array}
\right)$$\\
$= \lambda \mathbf{u}.$

By expanding this we have the following equations.
\begin{eqnarray}
\label{eqcy5}
u_0 - \frac{1}{\sqrt{2}}u_1 &=& \lambda u_0, \nonumber\\
- \frac{1}{\sqrt{2}} u_0 + u_1 -\frac{1}{2} u_2 &=& \lambda u_1, \nonumber\\
& \vdots &\\
-\frac{1}{2} u_{n-3} + u_{n-2} - \frac{1}{\sqrt{2}} u_{n-1} &=& \lambda u_{n-2}, \nonumber\\
-\frac{1}{\sqrt{2}} u_{n-2} + u_{n-1} &=& \lambda u_{n-1}.\nonumber
\end{eqnarray}
Let $\mathbf{u}'=(u'_i), (i=0,\ldots,2n-3)$ be an eigenvector of normalized Laplacian matrix of $C_{2n-2}$,
where $(u'_i)=\cos (\frac{2 i k\pi}{2n-2}), (i=0,\ldots,2n-3)$ and $\lambda_k=1-\cos(\frac{2k \pi}{2n-2}),(k=0,\ldots,2n-3)$ be its eigenvalue.
We note that $u'_1=u'_{2n-3}, u'_2=u'_{2n-4},\ldots,u'_{n-2}=u'_n$.
Then we multiply each of these values by $\frac{1}{\sqrt{2}}$ and obtain the vector,
$u'_0, \frac{1}{\sqrt{2}} u'_1,\frac{1}{\sqrt{2}} u'_2,\ldots,\frac{1}{\sqrt{2}}u'_{n-2}, u'_{n-1},\frac{1}{\sqrt{2}}u'_{n},\ldots,\frac{1}{\sqrt{2}} u'_{2n-3}$.
We can write ${\cal L}(C_{2n-2}) \mathbf{u}' =\lambda_k \mathbf{u}'$ as 
$$
\left(
\begin{array}{cccccc}
1 & -\frac{1}{2} & \cdots & \cdots&\cdots&  -\frac{1}{2} \\
-\frac{1}{2}& 1  & -\frac{1}{2}&  &  &\vdots\\
\vdots & \ddots & \ddots & \ddots & &\vdots\\
\vdots & & \ddots & \ddots & \ddots & \vdots \\
\vdots &  &  &-\frac{1}{2} & 1 & -\frac{1}{2} \\
-\frac{1}{2} & \cdots & \cdots & \cdots & -\frac{1}{2} & 1
\end{array}
\right)
\left(
\begin{array}{c} u'_0 \\ \frac{1}{\sqrt{2}}u'_1 \\ \vdots\\
 \\  u'_{n-1} \\\vdots\\ \frac{1}{\sqrt{2}} u'_{2n-3}
\end{array}
\right)$$\\
$= \lambda_k \mathbf{u}'$.

By expanding we have,
{\small {\begin{eqnarray}
\label{eqcy6}
u'_0 - \frac{1}{2}\frac{1}{\sqrt{2}}u'_1 -
 \frac{1}{2}\frac{1}{\sqrt{2}}u'_1
&=& u'_0 - \frac{1}{\sqrt{2}}u'_1  \nonumber\\
&=& \lambda_k u'_0, \nonumber\\
-\frac{1}{2} u'_0 + \frac{1}{\sqrt{2}} u'_1
 -\frac{1}{2}\frac{1}{\sqrt{2}}u'_2
&=& \frac{1}{\sqrt{2}}(-\frac{1}{\sqrt{2}}u'_0 +  u'_1-\frac{1}{2}u'_2) \nonumber\\
&=& \lambda_k(\frac{1}{\sqrt{2}}u'_1), \\
& \vdots & \nonumber\\
-\frac{1}{2} \frac{1}{\sqrt{2}}u'_{n-3} + \frac{1}{\sqrt{2}} u'_{n-2}
 -\frac{1}{2}u'_{n-1}
& =&  \lambda_k(\frac{1}{\sqrt{2}}u'_{n-2}), \nonumber\\
-\frac{1}{2} \frac{1}{\sqrt{2}}u'_{n-2} + u'_{n-1} 
 -\frac{1}{2}\frac{1}{\sqrt{2}}u'_{n-2} 
&=& -\frac{1}{\sqrt{2}}u'_{n-2} + u'_{n-1} \nonumber\\
&=& \lambda_k u'_{n-1}. \nonumber
\end{eqnarray}}}
Comparing Equation~\ref{eqcy5} and Equation~\ref{eqcy6},
we have $P_n \mathbf{u}=\lambda_k \mathbf{u}$,
where $\mathbf{u}=(u'_0,\sqrt{2} u'_1,\ldots,\sqrt{2}u'_{n-2},u'_{n-1})$.
That is $\lambda_k, (k=0,\ldots,n-1)$ are eigenvalues of $P_n$ and $\mathbf{u}$ is an eigenvector of $\lambda_k$.
Since $\lambda_i \neq \lambda_j$ for ($ i \neq j$ and  $0 \leq i,j \leq n-1$),
we have $n$ different eigenvalues of $P_n$ and that is the complete set of eigenvalues of $P_n$.
\item
Let $\mathbf{u}=(u_i), (i=0,\ldots,n-1)$ be an eigenvector for an eigenvalue $\lambda$ of signless Laplacian matrix of path $P_n$.
Then we can write the equation $SL(P_n)\mathbf{u}=\lambda \mathbf{u}$ as

\[ 
\left(
\begin{array}{cccccc}
1  & 1    &0    & \cdots & \cdots & 0 \\
1 & 2   & 1    & 0   &        & \vdots \\
0  & 1 & 2    & 1    &        & \vdots \\
\vdots  &         & \ddots & \ddots & \ddots &\vdots \\
\vdots  &         &        & \ddots & \ddots & 1 \\
0    & 0     & \cdots & \cdots & 1& 1
\end{array}
\right)
\left(
\begin{array}{c}
u_0 \\  u_1  \\   \vdots  \\  \vdots \\ u_{n-2}\\  u_{n-1} 
\end{array}
\right)\\
= \lambda \mathbf{u}\].

\begin{eqnarray}
\label{eqcy7}
u_0 + u_1 &=& \lambda u_0, \nonumber\\
u_0 + 2u_1 + u_2 &=& \lambda u_1, \nonumber\\
& \vdots &\\
u_{n-2} + u_{n-1} &=& \lambda u_{n-1}.\nonumber
\end{eqnarray}
Let $\mathbf{u}'=(u'_i), (i=0,\ldots,2n-1)$ be an eigenvector of signless Laplacian matrix of $C_{2n}$,
where $\displaystyle (u'_i)=\sin \frac{(2i+1)k\pi}{2n}, (i=0,\ldots,2n-1)$ and $\displaystyle \lambda_k=2+2\cos(\frac{(k+1) \pi}{2n}),(k=0,\ldots,2n-1)$ be its eigenvalue.
We note that $u'_0=-u'_{2n-1}, u'_1=-u'_{2n-2},\ldots,u'_{n-1}=-u'_n$.
Then we can write the equation $ SL(C_{2n}) \mathbf{u}'=\lambda_k \mathbf{u}'$ as
$$
\left(
\begin{array}{cccccc}
2  & 1    & 0    & \cdots & \cdots & 1 \\
1 & 2     & 1   & 0   &        & 0 \\
0  & 1 & 2   & 1    &        & \vdots \\
\vdots  &         & \ddots & \ddots & \ddots &\vdots \\
\vdots  &         &        & 1 & 2 & 1 \\
1  & 0    & \cdots & \cdots & 1& 2
\end{array}
\right)
\left(
\begin{array}{c} u'_0 \\ u'_1 \\ \vdots  \\\vdots\\ u'_{2n-2} \\ u'_{2n-1}
\end{array}
\right)\\
= \lambda_k \mathbf{u}'.
$$

\begin{eqnarray}
\label{eqcy8}
2u'_0 +u'_1-u'_0 &=& u'_0+ u'_1 = \lambda_k u'_0, \nonumber\\
u'_0 + 2u'_1 + u'_2 &=& \lambda_k u'_1, \nonumber\\
& \vdots & \\
u'_{n-2}+ 2u'_{n-1} - u'_{n-1} &=& u'_{n-2} + u'_{n-1} = \lambda_k u'_{n-1}. \nonumber
\end{eqnarray}

Comparing Equation~\ref{eqcy7} and Equation~\ref{eqcy8},
we have $P_n \mathbf{u}=\lambda_k \mathbf{u}$,
where $\mathbf{u}=(u'_i), (i=0,\ldots,n-1)$.
That is $\lambda_k, (k=0,\ldots,n-1)$ are eigenvalues of $P_n$ and $\mathbf{u}$ is an eigenvector of $\lambda_k$.
Since $\lambda_i \neq \lambda_j$ for $( i \neq j$ and  $0 \leq i,j \leq n-1)$, 
we have $n$ different eigenvalues of $P_n$ and that is the complete set of signless eigenvalues of $P_n$. 
\end{enumerate}
\end{proof}

\subsection{Tridiagonal Matrices}

In this section,
 we derive eigenvalues of adjacency matrices of paths and cycles using Chebyshev polynomials.
\begin{definition}
Let $T_0(x)=1$ and $U_0(x)=0$.
For $n \in {\mathbf{N}}$, $T_n(x)$ and $U_n(x)$ are defined by
$$
\left(\begin{array}{c} T_{n+1}(x)\\U_{n+1}(x) \end{array} \right)
=
\left(\begin{array}{cc}
x & x^2-1 \\ 1 & x \end{array} \right)
\left(\begin{array}{c} T_{n}(x)\\U_{n}(x) \end{array} \right).
$$
We call $T_n(x)$ as 
the Chebyshev polynomials of the first kind,
and $U_n(x)$ as
the Chebyshev polynomials of the second kind.
\end{definition}

\begin{example}
By using the above definition we have,
\begin{eqnarray*}
\left(\begin{array}{c} T_{1}(x)\\U_{1}(x) \end{array} \right)
&=&
\left(\begin{array}{cc}
x & x^2-1 \\ 1 & x \end{array} \right)
\left(\begin{array}{c} 1 \\ 0 \end{array} \right)
=
\left(\begin{array}{c} x \\ 1 \end{array} \right), \\
\left(\begin{array}{c} T_{2}(x)\\U_{2}(x) \end{array} \right)
&=&
\left(\begin{array}{cc}
x & x^2-1 \\ 1 & x \end{array} \right)
\left(\begin{array}{c} x \\ 1 \end{array} \right)
=
\left(\begin{array}{c} 2x^2-1 \\ 2x \end{array} \right), \\
\left(\begin{array}{c} T_{3}(x)\\U_{3}(x) \end{array} \right)
&=&
\left(\begin{array}{cc}
x & x^2-1 \\ 1 & x \end{array} \right)
\left(\begin{array}{c} 2x^2-1 \\ 2x \end{array} \right)\\
&~&=
\left(\begin{array}{c} 4x^3-3x \\ 4x^2-1 \end{array} \right). \\
\end{eqnarray*}
\end{example}

\begin{proposition}
\label{propcheb}
$T_0(x)=1$, $T_1(x)=x$, $U_0(x)=0$, $U_1(x)=1$,
\begin{eqnarray*}
T_{n+1}(x)&=& 2xT_n(x)-T_{n-1}(x), \mbox{\ and\ } \\
U_{n+1}(x)&=& 2xU_n(x)-U_{n-1}(x).
\end{eqnarray*}
\end{proposition}\hfill\qed

\begin{proposition}
\begin{eqnarray*}
\cos n\theta &=& T_n(\cos \theta), \\
\sin n\theta &=& U_n(\cos \theta) \sin \theta.
\end{eqnarray*}
\end{proposition}\hfill\qed

We note that the degree of the polynomial $T_n(x)$ is $n$
and the degree of the polynomial $U_n(x)$ is $n-1$ for $n \ge 2$.

\begin{proposition}
Let $x=\cos\theta$ and $n \ge 2$.
Then 
\begin{eqnarray*}
T_{n}(x)=0 &\Leftrightarrow&
x=\cos(\frac{(2k+1)\pi}{2 n}) \ \ (k=0,\ldots,n-1). \\
U_{n}(x)=0 &\Leftrightarrow&
x=\cos(\frac{k\pi}{n}) \ \ (k=1,\ldots,n-1). \\
\end{eqnarray*}
\end{proposition}\hfill\qed

The determinant of tridiagonal matrices can be represented by using recurrence relations. 
We consider tridiagonal matrices with similar diagonal elements.
Then we derive a formula for eigenvalues of tridiagonal matrices.
\begin{definition}
A $n\times n$ tridiagonal matrix $A_n=(a_{ij})$ is a matrix which has the form
$$
A_n = \left(
\begin{array}{ccccc}
\alpha_1 & \beta_1 &   0     & \cdots & 0\\
\gamma_1 & \alpha_2 & \beta_2 & \ddots & \vdots \\
   0     & \gamma_2 & \ddots &  \ddots &  0 \\
 \vdots  & \ddots   & \ddots & \alpha_{n-1} & \beta_{n-1} \\
   0     & \cdots   &   0    & \gamma_{n-1} & \alpha_{n}
\end{array}
\right).
$$
\end{definition}

\begin{proposition}\label{prop:tridiagonal}
Let $n \ge 2$, $|A_0|=1$, and $|A_1|=\alpha_1$.
Then we have,
$$
|A_n| = \alpha_{n} |A_{n-1}| -\beta_{n-1}\gamma_{n-1}|A_{n-2}|.
$$
\end{proposition}

\begin{proposition}
Eigenvalues of adjacency matrix of a path graph are given by $\displaystyle \lambda_k(A(P_n))=2 \cos(\frac{k \pi}{n+1}) (k=1,\ldots,n)$.
\end{proposition}
\begin{proof}
The matrix $\lambda I - P_n$ is a tridiagonal matrix with
$\alpha_i=\lambda$,
 $\beta_i=\gamma_i=-1$.
Let $f_n(\lambda)=|\lambda I - P_n|$.
By Proposition~\ref{prop:tridiagonal}, 
$f_n(\lambda)$ is defined by 
$f_n(\lambda)=\lambda f_{n-1}(\lambda)-f_{n-2}(\lambda)$,
where $f_0(\lambda)=1$ and $f_1(\lambda)=\lambda$.
Let $\displaystyle g_n(\lambda)=U_{n+1}(\frac{\lambda}{2})$.
Then
\begin{eqnarray*}
g_0(\lambda) &=& U_{1}(\frac{\lambda}{2}) = 1, \\
g_1(\lambda) &=& U_{2}(\frac{\lambda}{2}) = 2(\frac{\lambda}{2}) =
 \lambda,  \\
g_n(\lambda) &=& U_{n+1}(\frac{\lambda}{2}) \\
&=&
 2\frac{\lambda}{2}U_n(\frac{\lambda}{2})-U_{n-1}(\frac{\lambda}{2}) \ (by  \ prop. ~\ref{propcheb} )\\
&=&
\lambda g_{n-1}(\lambda) -g_{n-2}(\lambda).
\end{eqnarray*}
Then we have $\displaystyle
f_n(\lambda)=g_n(\lambda)=U_{n+1}(\frac{\lambda}{2})$.
That is
\begin{eqnarray*}
f_n(\lambda)=0 & \Leftrightarrow& U_{n+1}(\frac{\lambda}{2})=0. \\
&\Leftrightarrow&  \frac{\lambda}{2}=\cos(\frac{k\pi}{n+1}) \ (k=1,\ldots,n).\\
&\Leftrightarrow&  \lambda = 2\cos(\frac{k\pi}{n+1}) \ (k=1,\ldots,n).\\
\end{eqnarray*}
Thus we obtain the result.\end{proof}

\begin{proposition}
Eigenvalues of adjacency matrix of a cycle are given by $\lambda_k(A(C_n)) =2\cos(\frac{2k\pi}{n})(k=1,\ldots,n)$.
\end{proposition}
\begin{proof}
The matrix $\lambda I - C_n$ is not a tridiagonal matrix.
But we have
$\displaystyle |\lambda I - C_n|=2(T_n(\frac{\lambda}{2})-1)$.
Since
$T_n(x)=1$
$\Leftrightarrow$
$\cos n\theta=1$
$\Leftrightarrow$
$\displaystyle \theta = \frac{2k\pi}{n}$.
We obtain
$$
|\lambda I - C_n|=\prod_{k=1}^n(\lambda -  2\cos(\frac{2k\pi}{n})).
$$ 
\end{proof}

\begin{proposition}
Let $P_n=(V_n,E_n)$ be a path graph. 
If  $G=(V_n,E_n \cup \{(v_1,v_1)\},\{(v_n,v_n)\})$ then
 the eigenvalues of Laplacian matrix of $G$ are given by $\displaystyle \lambda_k = a+2\cos(\frac{k\pi}{n+1}),(k=1,\ldots,n)$.
\end{proposition}
\begin{figure}[htb]
\begin{center}
\includegraphics[scale=0.5]{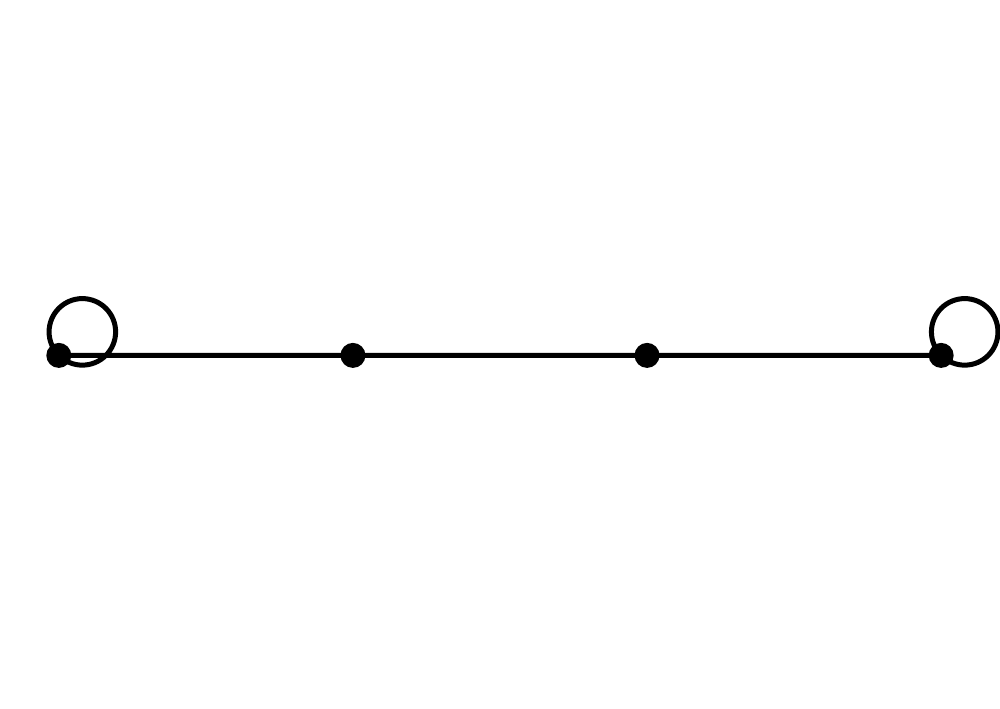}
\caption{Path graph with equal vertex degrees.}
\label{fig:fig1}
\end{center}
\end{figure}
\begin{proof} 
Let $L(P_n)$ the Laplacian matrix of a path graph with vertex weight $a$ on $n$ vertices.
The matrix $\lambda I - L(P_n)$ is a tridiagonal matrix with
$\alpha_i=\lambda-a$, $\beta_i=\gamma_i=-1$.
Let $f_n(\lambda)=|\lambda I - L(P_n)|$ and
$f_n(\lambda)$ is defined by 
$f_n(\lambda)=\lambda f_{n-1}(\lambda-a)-f_{n-2}(\lambda)$,
where $f_0(\lambda)=1$ and $f_1(\lambda)=\lambda-a$.
Let $\displaystyle g_n(\lambda)=U_{n+1}(\frac{\lambda-a}{2})$.
Since
\begin{eqnarray*}
g_0(\lambda) &=& U_{1}(\frac{\lambda-a}{2}) = 1, \\
g_1(\lambda) &=& U_{2}(\frac{\lambda-a}{2}) = 2(\frac{\lambda-a}{2}) =\lambda-a
 , and \\
g_n(\lambda) &=& U_{n+1}(\frac{\lambda-a}{2}) \\
&=&
 2(\frac{\lambda-a}{2})U_n(\frac{\lambda-a}{2})-U_{n-1}(\frac{\lambda-a}{2}) \\
&=&
(\lambda-a) g_{n-1}(\lambda) -g_{n-2}(\lambda),
\end{eqnarray*}
we have $\displaystyle
f_n(\lambda)=g_n(\lambda)=U_{n+1}(\frac{\lambda-a}{2})$.
That is
\begin{eqnarray*}
f_n(\lambda)=0 &\Leftrightarrow& U_{n+1}(\frac{\lambda-a}{2})=0 \\
&\Leftrightarrow&  \frac{\lambda-a}{2}=\cos(\frac{k\pi}{n+1}) \ (k=1,\ldots,n)\\
&\Leftrightarrow&  \lambda = a+2\cos(\frac{k\pi}{n+1}) \ (k=1,\ldots, n).\\
\end{eqnarray*}\end{proof}
\subsection{Determinant of tridiagonal matrices}
Let $n \ge 2$. 
We define a $n \times n$ matrix $A_n(a,b)$ as follows:
$$
A_n(a,b)=\left(
\begin{array}{ccccccc}
a      & b & 0 & \cdots & \cdots & \cdots& 0 \\
b      & a & b & 0      &        &          & \vdots  \\
0      & b & a & b      &   0    &          & \vdots  \\
\vdots & \ddots  & \ddots  & \ddots & \ddots  & \ddots  & \vdots  \\
\vdots &   & 0 &   b    &   a    &   b      & 0 \\
\vdots &   &   &   0    &   b    & a        & b \\
 0      & \cdots  & \cdots  & \cdots &   0    & b        & a
\end{array}
\right).
$$
\begin{lemma}
Let $n \ge 2$ and $a \not= 2b$.
$$
|A_n(a,b)|=b^n \cdot \frac{\sin(n+1)\theta}{\sin \theta},
$$
where $\displaystyle \frac{a}{2b}=\cos \theta$ ($0 < \theta < 2\pi$).
\end{lemma}
\noindent
\begin{proof}
By Proposition~17, we have
$|A_n(a,1)| = a |A_{n-1}(a,1)| - |A_{n-2}(a,1)|$.
Let $|A_0(a,1)|=1$ and $|A_1(a,1)|=a$.
Since
$U_1(x)=1$, $U_2(x)=2x$ and $U_{n+1}=2xU_n(x)-U_{n-1}(x)$,
we have
$|A_n(a,1)|$ $\displaystyle =U_{n+1}\left(\frac{a}{2}\right)$
and $\displaystyle A_n(a,1)=\frac{\sin(n+1)\theta}{\sin \theta}$,
where $\displaystyle \frac{a}{2}$ $=\cos \theta$.
Since $\displaystyle A_n(a,b)=b^n \cdot A_n\left(\frac{a}{b},1\right)$,
we have
$$
|A_n(a,b)| = b^n \cdot \left\vert
        A_n\left(\frac{a}{b},1\right)\right\vert \\
= b^n \cdot U_{n+1}\left(\frac{a}{2b}\right) \\
= b^n \frac{\sin(n+1)\alpha}{\sin \alpha}
$$
where $\displaystyle \frac{a}{2b}=\cos \alpha$.
\end{proof}

\begin{example}
\begin{description}
\item[(i)]
$\displaystyle \left\vert A_n\left(\lambda-1,\frac{1}{2}\right)
\right\vert =
\left(\frac{1}{2}\right)^n \frac{\sin (n+1) \alpha}{\sin \alpha}$,
where $\displaystyle \lambda=1+\cos \alpha$.
\item[(ii)]
$\displaystyle \left\vert
     A_n\left(\eta-\frac{2}{3},\frac{1}{3}\right)
\right\vert =
\left(\frac{1}{3}\right)^n \frac{\sin (n+1) \beta}{\sin \beta}$,
where $\displaystyle \eta= \frac{2}{3}(1 +  \cos \beta)$.
\end{description}
\end{example}

Let $n \ge 3$.
We define a $n \times n$ matrix $B_n(a_0,b_0,a,b)$
and $C_n(a,b,a_0,b_0)$ as follows:
\begin{eqnarray*}
B_n(a_0,b_0,a,b) &=&
\left(
\begin{array}{c|cccc}
a_0 & b_0 & 0  & \cdots & 0 \\
\hline
b_0 &     &      &    &   \\
0   &     & \multicolumn{2}{c}{A_{n-1}(a,b)} & \\
\vdots &  &            &&  \\
0    &    &           & &  
\end{array}
\right) \\
C_n(a,b,a_0,b_0) &=&
\left(
\begin{array}{cccc|c}
 &  &   &  & 0 \\
 & \multicolumn{2}{c}{A_{n-1}(a,b)}  &  & \vdots \\
 &  & & & 0 \\
 &  &    & & b_0  \\
\hline
0  & \cdots   & 0 & b_0 & a_0  
\end{array}
\right) \\
\end{eqnarray*}

We note that
\begin{eqnarray*}
\vert B_n(a_0,b_0,a,b) \vert &=&
a_0 \vert A_{n-1}(a,b) \vert - b_0^2 \vert A_{n-2}(a,b)\vert, \mbox{\ and } \\
\vert C_n(a,b,a_0,b_0) \vert &=&
a_0 \vert A_{n-1}(a,b) \vert - b_0^2 \vert A_{n-2}(a,b)\vert . \\
\end{eqnarray*}

We define functions
\begin{eqnarray*}
g_{n}(\beta)&=&2 \sin ((n+1)\beta) + \sin n\beta -\sin ((n-1)\beta),
 \mbox{\ and}\\
h_{n}(\gamma)&=&2 \sin ((n+1)\gamma) - \sin n\gamma -\sin ((n-1)\gamma)
\end{eqnarray*}
before introducing the next Lemma.

\begin{lemma}
Let $n \ge 3$. 
\begin{enumerate}
\item
$\displaystyle
\left\vert B_n(\lambda-1,\frac{1}{\sqrt{2}},\lambda-1,\frac{1}{2})
\right\vert = \frac{1}{2^{n-1}} \cos n \alpha$,
where $\lambda=1+\cos \alpha$.
\item
$\displaystyle
\left\vert C_n(\eta-\frac{2}{3}, \frac{1}{3},\eta-\frac{1}{2},
\frac{1}{\sqrt{6}}) \right\vert = \frac{1}{2\cdot 3^{n} \cdot \sin
\beta} g_{n}(\beta)$,
where $\displaystyle \eta = \frac{2}{3}(1+\cos \beta)$.
\item
$\displaystyle
\left\vert C_n(\mu-\frac{4}{3}, \frac{1}{3},\mu-\frac{3}{2},
\frac{1}{\sqrt{6}}) \right\vert = \frac{1}{2\cdot 3^{n} \cdot \sin
\gamma} h_{n}(\gamma)$,
where $\displaystyle \mu = \frac{2}{3}(2+\cos \gamma)$.
\end{enumerate}
\end{lemma}
\begin{proof}
\begin{enumerate}
\item
\begin{eqnarray*}
\left\vert B_n(\lambda-1,\frac{1}{\sqrt{2}},\lambda-1,\frac{1}{2})
\right\vert &=&
(\lambda-1)\left\vert A_{n-1}(\lambda-1,\frac{1}{2}) \right\vert\\
&~&- \frac{1}{2} \left\vert A_{n-2}(\lambda-1,\frac{1}{2}) \right\vert 
\end{eqnarray*}
\begin{align*}
&= (\lambda-1)\left(\frac{1}{2}\right)^{n-1}
\frac{\sin n \alpha}{\sin \alpha}
- \frac{1}{2}\left(\frac{1}{2}\right)^{n-2}\frac{\sin(n-1)\alpha}{\sin
\alpha}\\
&= \left(\frac{1}{2}\right)^{n-1} \cdot
\frac{1}{\sin \alpha}(
(\lambda-1 ) \sin n \alpha -
\sin (n-1) \alpha) \\
&= \left(\frac{1}{2}\right)^{n-1} \cdot
\frac{1}{\sin \alpha}(
\cos \alpha \sin n \alpha -\sin (n\alpha - \alpha )) \\
&= \left(\frac{1}{2}\right)^{n-1} \cdot
\frac{1}{\sin \alpha} (\cos n\alpha \sin \alpha) \\
&=\left(\frac{1}{2}\right)^{n-1} \cos n \alpha
\end{align*}
\item 
\begin{align*}
\left\vert C_n(\eta-\frac{2}{3}, \frac{1}{3},\eta-\frac{1}{2},
\frac{1}{\sqrt{6}}) \right\vert &=
\left(\eta - \frac{1}{2}\right)
\left\vert A_{n-1}\left(\eta-\frac{2}{3},\frac{1}{3}\right)
\right\vert\\
&-\frac{1}{6}
\left\vert A_{n-2}\left(\eta-\frac{2}{3},\frac{1}{3}\right)
\right\vert 
\end{align*}
\begin{align*}
&=\left(\eta-\frac{1}{2}\right)
\left(\frac{1}{3}\right)^{n-1}\frac{\sin n \beta}{\sin \beta}- \frac{1}{6} \left(\frac{1}{3}\right)^{n-2}
\frac{\sin (n-1) \beta}{\sin\beta}\\
&= \left(\frac{1}{3}\right)^{n-1}\left(\left(\eta-\frac{1}{2}\right)\frac{\sin n \beta}{\sin \beta}-\frac{1}{2}\frac{\sin (n-1) \beta}{\sin \beta}\right) \\
&= \left(\frac{1}{3}\right)^{n-1}\left(\left(\frac{1}{6} + \frac{2}{3}\cos\beta\right)
\frac{\sin n \beta}{\sin \beta}  -\frac{1}{2}\frac{\sin (n-1) \beta}{\sin \beta} \right) \\
&= \left(\frac{1}{3}\right)^{n-1}
\frac{1}{6 \sin \beta}
( \sin n \beta + 4 \cos \beta \sin n \beta - 3 \sin (n\beta -\beta )) \\
&= \left(\frac{1}{3}\right)^{n-1}
\frac{1}{6 \sin \beta}
( \sin n \beta + 4 \cos \beta \sin n \beta - \sin(n\beta-\beta)\\
&-2 \sin n\beta \cos \beta + 2 \cos n\beta \sin \beta ) 
\end{align*}
\begin{align*}
&= \left(\frac{1}{3}\right)^{n-1}
\frac{1}{6 \sin \beta}
( \sin n \beta + 2 \cos n\beta \sin \beta - \sin(n\beta-\beta)\\
& +  2 \sin n\beta \cos \beta ) \\
&=\left(\frac{1}{3}\right)^{n-1}\frac{1}{6 \sin \beta}
( \sin n \beta + 2 \sin (n \beta + \beta ) - \sin (n\beta-\beta ) ) \\
&= \left(\frac{1}{3}\right)^{n-1}
\frac{1}{6 \sin \beta}
(2 \sin (n+1) \beta + \sin n \beta - \sin (n-1) \beta ) \\
&= \left(\frac{1}{3}\right)^{n}
\frac{1}{2 \sin \beta} g_n(\beta)
\end{align*}

\item
\begin{align*}
\left\vert C_n(\mu-\frac{4}{3}, \frac{1}{3},\mu-\frac{3}{2},
\frac{1}{\sqrt{6}}) \right\vert &=
\left(\mu - \frac{3}{2}\right)
\left\vert A_{n-1}\left(\mu-\frac{4}{3},\frac{1}{3}\right)
\right\vert\\
&-\frac{1}{6}
\left\vert A_{n-2}\left(\mu-\frac{4}{3},\frac{1}{3}\right)
\right\vert 
\end{align*}
\begin{align*}
&=\left(\mu-\frac{3}{2}\right)
\left(\frac{1}{3}\right)^{n-1}\frac{\sin n \gamma}{\sin \gamma}- \frac{1}{6} \left(\frac{1}{3}\right)^{n-2}
\frac{\sin (n-1) \gamma}{\sin\gamma}\\
&=\left(\frac{1}{3}\right)^{n-1}\left(\left(\mu-\frac{3}{2}\right)\frac{\sin n \gamma}{\sin \gamma}-\frac{1}{2}\frac{\sin (n-1) \gamma}{\sin \gamma}\right) \\
&= \left(\frac{1}{3}\right)^{n-1}\left(\left(-\frac{1}{6} +
\frac{2}{3}\cos\gamma\right)
\frac{\sin n \gamma}{\sin \gamma}  -\frac{1}{2}\frac{\sin (n-1) \gamma}{\sin \gamma} \right)\\ 
&= \left(\frac{1}{3}\right)^{n-1}
\frac{1}{6 \sin \gamma}
( -\sin n \gamma + 4 \cos \gamma \sin n \gamma - 3 \sin (n\gamma -\gamma )) \\
&= \left(\frac{1}{3}\right)^{n-1}
\frac{1}{6 \sin \gamma}
(2 \sin (n+1) \gamma - \sin n \gamma - \sin (n-1) \gamma ) \\
&= \left(\frac{1}{3}\right)^{n}
\frac{1}{2 \sin \gamma} h_n(\gamma)
\end{align*}
\end{enumerate}
\end{proof}

\subsection{Eigenvalues of $\mathcal{L}(P_n)$}
\begin{example}
The adjacency matrix and the normalized Laplacian matrix of a 
path graph $P_{5}$.
\begin{eqnarray*}
A(P_5)&=&
\left(
\begin{array}{ccccc}
 0 & 1 & 0 & 0 & 0 \\
 1 & 0 & 1 & 0 & 0 \\
 0 & 1 & 0 & 1 & 0 \\
 0 & 0 & 1 & 0 & 1 \\
 0 & 0 & 0 & 1 & 0
\end{array}
\right) \\
\mathcal{L}(P_5) &=&
\left(
\begin{array}{ccccc}
 1 & -\frac{1}{\sqrt{2}} & 0 & 0 & 0 \\
 -\frac{1}{\sqrt{2}} & 1 & -\frac{1}{2} & 0 & 0 \\
 0 & -\frac{1}{2} & 1 & -\frac{1}{2} & 0 \\
 0 & 0 & -\frac{1}{2} & 1 & -\frac{1}{\sqrt{2}} \\
 0 & 0 & 0 & -\frac{1}{\sqrt{2}} & 1 \\
\end{array}
\right)
\end{eqnarray*}
\end{example}
Let $n \ge 4$.
We define $n \times n$ matrix $Q_n(a_0,b_0,a,b)$ as the following.
\begin{eqnarray*}
Q_n(a_0,b_0,a,b) &=&
\left(
\begin{array}{c|ccc|c}
a_0 & b_0 & 0  & \cdots & 0 \\
\hline
b_0 &     &      &      & \vdots \\
0   &     \multicolumn{3}{c|}{A_{n-2}(a,b)} & 0 \\
\vdots    &           & &  & b_0 \\
\hline
0  &  \cdots & 0 & b_0 & a_0
\end{array}
\right)
\end{eqnarray*}

We note that
$$
\vert Q_n(a_0,b_0,a,b) \vert = a_0 \vert C_{n-1}(a,b,a_0,b_0) \vert
- b_0^2 \vert C_{n-2}(a,b,a_0,b_0) \vert.
$$
\begin{proposition}
Let $n \ge 4$.
The characteristic polynomial of
${\mathcal L}(P_n)$ is
$$
 \vert \lambda I_n - \mathcal{L}(P_n) \vert =
 -\left(\frac{1}{2}\right)^{n-2} (\sin \alpha \sin ((n-1)\alpha)),
$$
where $\lambda=1 + \cos \alpha$. 
That is $\lambda = 1 - \cos(\frac{k \pi}{n-1})$ \,
($k=0, \ldots, n-1$).
\end{proposition}
\begin{proof}
First, we note 
$\displaystyle \mathcal{L}(P_n) = Q_n\left(1, -\frac{1}{\sqrt{2}}, 1, -\frac{1}{2}\right)$
and\\
$\displaystyle \vert \lambda I_n - \mathcal{L}(P_n) \vert =
\left\vert Q_n\left(\lambda-1,\frac{1}{\sqrt{2}},\lambda-1,\frac{1}{2}\right) \right\vert$.
{\small{
 \begin{align*}
& \left\vert
    Q_n\left(\lambda-1,\frac{1}{\sqrt{2}},\lambda-1,\frac{1}{2}\right)
   \right\vert \\
&=
(\lambda-1) \left\vert
C_{n-1}\left(\lambda-1,\frac{1}{2},\lambda-1,\frac{1}{\sqrt{2}}\right)
	    \right\vert \\
&-\frac{1}{2} \left\vert
C_{n-2}\left(\lambda-1,\frac{1}{2},\lambda-1,\frac{1}{\sqrt{2}}\right)
       \right\vert \\
&= (\lambda-1) \left\vert
B_{n-1}\left(\lambda-1,\frac{1}{\sqrt{2}},\lambda-1,\frac{1}{2}\right)
		\right\vert\\
&-\frac{1}{2} \left\vert
B_{n-2}\left(\lambda-1,\frac{1}{\sqrt{2}},\lambda-1,\frac{1}{2}\right)
       \right\vert \\
&=
(\lambda-1)\left(\frac{1}{2}\right)^{n-2} \cos((n-1)\alpha)
-\frac{1}{2}\left(\frac{1}{2}\right)^{n-3} \cos((n-2)\alpha) \\
&=
\left(\frac{1}{2}\right)^{n-2}(\cos \alpha \cdot \cos ((n-1)\alpha) 
- \cos ((n-2)\alpha))\\
&=
\left(\frac{1}{2}\right)^{n-2}(\cos \alpha \cdot \cos((n-1)\alpha) 
- (\cos((n-1)\alpha) \cdot \cos \alpha +\\
& \sin((n-1)\alpha) \cdot \sin \alpha))\\
&=
- \left(\frac{1}{2}\right)^{n-2} \sin((n-1)\alpha) \cdot \sin \alpha.
 \end{align*}}}

We have $\displaystyle \alpha=\frac{k\pi}{n-1}$ ($k=0,\ldots, n-1$).
Since $\displaystyle \cos\left(\frac{k\pi}{n-1}\right)
=\cos\left(\frac{(k+(n-1))\pi}{n-1}\right)$,
we have $\displaystyle \lambda=1+\cos\left(\frac{k\pi}{n-1}\right)$ ($k=0,\ldots, n-1$).
The set is equal to
$\displaystyle \lambda=1-\cos\left(\frac{k\pi}{n-1}\right)$ ($k=0,\ldots, n-1$).
\end{proof}

\subsection{Eigenvalues of weighted paths and $\mathcal{L}(R_{n,k})$}

\begin{example}
The adjacency matrix and the normalized Laplacian matrix of a weighted path graph $P_{4,3}$.
 \begin{eqnarray*}
A(P_{4,3})&=&
\left(
\begin{array}{ccccccc}
 0 & 1 & 0 & 0 & 0 & 0 & 0 \\
 1 & 0 & 1 & 0 & 0 & 0 & 0 \\
 0 & 1 & 0 & 1 & 0 & 0 & 0 \\
 0 & 0 & 1 & 0 & 1 & 0 & 0 \\
 0 & 0 & 0 & 1 & 1 & 1 & 0 \\
 0 & 0 & 0 & 0 & 1 & 1 & 1 \\
 0 & 0 & 0 & 0 & 0 & 1 & 1
\end{array}
\right) \\
\mathcal{L}(P_{4,3})&=&
\left(
\begin{array}{ccccccc}
 1 & -\frac{1}{\sqrt{2}} & 0 & 0 & 0 & 0 & 0 \\
 -\frac{1}{\sqrt{2}} & 1 & -\frac{1}{2} & 0 & 0 & 0 & 0 \\
 0 & -\frac{1}{2} & 1 & -\frac{1}{2} & 0 & 0 & 0 \\
 0 & 0 & -\frac{1}{2} & 1 & -\frac{1}{\sqrt{6}} & 0 & 0 \\
 0 & 0 & 0 & -\frac{1}{\sqrt{6}} & \frac{2}{3} & -\frac{1}{3} & 0 \\
 0 & 0 & 0 & 0 & -\frac{1}{3} & \frac{2}{3} & -\frac{1}{\sqrt{6}} \\
 0 & 0 & 0 & 0 & 0 & -\frac{1}{\sqrt{6}} & \frac{1}{2} \\
\end{array}
\right)
\end{eqnarray*}
\end{example}

Let $n \ge 3$ and $k \ge 3$. Then
$$
\mathcal{L}(P_{n,k})
= \left(
\begin{array}{c|c}
B_n(1,-\frac{1}{\sqrt{2}},1,-\frac{1}{2}) &
X_{n,k} \\
\hline
X_{n,k}^t &
C_k(\frac{2}{3},-\frac{1}{3},\frac{1}{2},-\frac{1}{\sqrt{6}})
\end{array}
\right),
$$
where $X_{n,k}$ is the $n \times k$ matrix defined by

\[ X_{n,k}=
\left(
\begin{array}{cccc}
0      & \cdots & \cdots & 0  \\
\vdots &        &        & \vdots \\
0      &  0     &        & \vdots \\
-\frac{1}{\sqrt{6}} & 0 & \cdots & 0
\end{array}
\right).\]

\begin{theorem}\label{prop:pnk}
Let $n\ge 3$ and $k \ge 3$.
The characteristic polynomial of
${\mathcal L}(P_{n,k})$ is
$$
\left\vert \lambda I_{n+k} - \mathcal{L}(P_{n,k}) \right\vert = p_{n,k}(\lambda),
$$
where
\begin{eqnarray*}
p_{n,k}(\lambda) &=&
\frac{1}{2^n3^k\sin \beta}
(g_k(\beta)\cos(n\alpha)) - g_{k-1}(\beta)\cos((n-1)\alpha)),
\end{eqnarray*}
$\lambda = 1 + \cos \alpha$ and
$\displaystyle\lambda = \frac{2}{3}(1 + \cos \beta)$.
\end{theorem}
\begin{proof}
Since
$\displaystyle
|B_{n}(\lambda-1,\frac{1}{\sqrt{2}},\lambda-1,\frac{1}{2})|=
\frac{\cos (n\alpha)}{2^{n-1}}$
and\\
$\displaystyle
|C_k(\lambda-\frac{2}{3},\frac{1}{3},\lambda-\frac{1}{2},\frac{1}{\sqrt{6}})|=
\frac{g_k(\beta)}{2\cdot 3^k \cdot \sin \beta}$,
we have
{\small{
$$
\left\vert \lambda I_{n+k} - \mathcal{L}(P_{n,k}) \right\vert
=\left\vert
\begin{array}{c|c}
B_n(\lambda-1,\frac{1}{\sqrt{2}},\lambda-1,\frac{1}{2}) &
X_{n,k} \\
\hline
X_{n,k}^t &
C_k(\lambda-\frac{2}{3},\frac{1}{3},\lambda-\frac{1}{2},\frac{1}{\sqrt{6}})
\end{array}
\right\vert 
$$
\begin{eqnarray*}
&=&
-\frac{1}{4}
 |B_{n-2}(\lambda-1,\frac{1}{\sqrt{2}},\lambda-1,\frac{1}{2})|
\cdot
|C_{k}(\lambda-\frac{2}{3},\frac{1}{3},\lambda-\frac{1}{2},\frac{1}{\sqrt{6}})| \\
&&
+ (\lambda -1)
|B_{n-1}(\lambda-1,\frac{1}{\sqrt{2}},\lambda-1,\frac{1}{2})|\cdot|C_{k}(\lambda-\frac{2}{3},\frac{1}{3},\lambda-\frac{1}{2},\frac{1}{\sqrt{6}})| \\
&&-\frac{1}{6} 
|B_{n-1}(\lambda-1,\frac{1}{\sqrt{2}},\lambda-1,\frac{1}{2})|
\cdot
|C_{k-1}(\lambda-\frac{2}{3},\frac{1}{3},\lambda-\frac{1}{2},\frac{1}{\sqrt{6}})| \\
&=&
-\frac{1}{4}\cdot \frac{\cos ((n-2)\alpha)}{2^{n-3}} \cdot
\frac{g_{k}(\beta)}{2 \cdot 3^k \cdot \sin \beta}
+ \cos \alpha \cdot \frac{\cos ((n-1)\alpha)}{2^{n-2}} \cdot\\
&&\frac{g_{k}(\beta)}{2 \cdot 3^k \cdot \sin \beta}
-\frac{1}{6}  \frac{\cos ((n-1)\alpha)}{2^{n-2}} \cdot
\frac{g_{k-1}(\beta)}{2 \cdot 3^{k-1} \cdot \sin \beta}
\\
&=&
\frac{1}{2^n\cdot 3^k \cdot \sin \beta}
(-\cos((n-2)\alpha)g_k(\beta)+2 \cos \alpha \cos((n-1)\alpha)
g_k(\beta) \\
&& - \cos((n-1)\alpha)g_{k-1}(\beta))\\
&=&
\frac{1}{2^n\cdot 3^k \cdot \sin \beta}
(\cos(n\alpha)g_k(\beta) - \cos((n-1)\alpha)g_{k-1}(\beta))
\\
&=& p_{n,k}(\lambda)
\end{eqnarray*}}}
We note that
\begin{eqnarray*}
&&-\cos((n-2)\alpha)+2\cos \alpha \cos((n-1)\alpha) \\
&=&
-\cos \alpha \cos ((n-1)\alpha) 
-\sin \alpha\sin((n-1)\alpha)\\
&&+2\cos \alpha \cos((n-1)\alpha)\\
&=&\cos \alpha \cos ((n-1)\alpha) - 
\sin \alpha \sin((n-1)\alpha)\\
&=&\cos (n\alpha).
\end{eqnarray*}
\end{proof}

\begin{lemma}\label{lemma:gk}
Let $k \ge 3$.
\begin{enumerate}
\item If 
$\displaystyle \frac{(4k-2)\pi}{4k-1} < \alpha < \pi$,
$0 < \beta < \pi$ and
$\displaystyle 1+\cos \alpha = \frac{2}{3}(1+\cos \beta)$
then
$\displaystyle \frac{(2k+1)\pi}{2(k+1)} < \beta$.
\item If
$\displaystyle \frac{(2k+1)\pi}{2(k+1)} < \beta < \pi$ then
$g_k(\beta) \not= 0$ and 
$\displaystyle \frac{g_{k-1}(\beta)}{g_k(\beta)}<-1$.
\item If $n=2k$ ($k \ge 0$),
$\displaystyle \frac{(2n-2)\pi}{2n-1} < \alpha < \pi$,
$0 < \beta < \pi$ and
$\displaystyle 1+\cos \alpha = \frac{2}{3}(1+\cos \beta)$
then\\
$\displaystyle 
g_{k}(\beta)\cos(n \alpha)-g_{k-1}(\beta)\cos((n-1)\alpha) \not=0.
$
\end{enumerate}
\end{lemma}
 
\begin{proof}
{1.\ } Since $k \ge 3$, we have
$\displaystyle \frac{4k-1}{k+1}$
$\displaystyle = 4-\frac{5}{k+1}$
$\displaystyle \ge \frac{11}{4}$
and 
$\displaystyle \frac{1}{4k-1} \le \frac{4}{11(k+1)}$.

Since $\displaystyle \frac{33}{8}\sqrt{\frac{2}{3}}$
$>3.36>\pi$,
we have 
$\displaystyle \sqrt{\frac{3}{2}}\cdot \frac{1}{2} \cdot
  \frac{4\pi}{11}$
$\displaystyle <\sqrt{\frac{3}{2}}\cdot \frac{1}{2} \cdot
  \frac{4}{11} \cdot \frac{33}{8} \sqrt{\frac{2}{3}}$
$\displaystyle = \frac{3}{4}$.

Since $\displaystyle \frac{3x}{2} \le \sin(\frac{\pi x}{2})$
($\displaystyle 0 \le x \le \frac{1}{3}$),
we have
$\displaystyle \frac{3}{4(k+1)} \le \sin \frac{\pi}{4(k+1)}$.
Since $\displaystyle 1+\cos \alpha=\frac{2}{3}(1+\cos \beta)$,
we have
$\displaystyle \cos^2 \frac{\alpha}{2}=\frac{2}{3}\cos^2 \frac{\beta}{2}$
and 
$\displaystyle \cos \frac{\beta}{2}=\sqrt{\frac{3}{2}}\cos \frac{\alpha}{2}$.

\begin{eqnarray*}
\sin (\frac{\pi-\beta}{2})=\cos \frac{\beta}{2}
&=& \sqrt{\frac{3}{2}} \cos \frac{\alpha}{2} \\
&=& \sqrt{\frac{3}{2}} \sin (\frac{\pi-\alpha}{2}) \\
&\le& \sqrt{\frac{3}{2}}(\frac{\pi-\alpha}{2}) \\
&<& \sqrt{\frac{3}{2}} \cdot \frac{1}{2} \cdot (\frac{\pi}{4k-1}) \\
&<& \sqrt{\frac{3}{2}} \cdot \frac{1}{2} \cdot (\frac{4\pi}{11(k+1)} \\
&<& \frac{3}{4(k+1)} \\
&=& \frac{1}{2}\cdot \frac{3}{2(k+1)} \\
&\le& \sin \frac{\pi}{4(k+1)}
\end{eqnarray*}

Then $\displaystyle \frac{\pi-\beta}{2} < \frac{\pi}{4(k+1)}$
and $\displaystyle \frac{(2k+1)\pi}{2(k+1)} < \beta$.

\noindent
{2.\ }Let $\beta'=\pi-\beta$. Then
$\displaystyle 0 < \beta' < \frac{\pi}{2(k+1)}$.
We note that if $k$ is even then $\sin(k\beta)=-\sin(k\beta')$
and if $k$ is odd then $\sin(k\beta)=\sin(k\beta')$.
Since $y=\sin x$ is convex on $\displaystyle 0<x <\frac{\pi}{2}$,
$\sin(t x_1+(1-t) x_2)> t\sin x_1+(1-t)\sin x_2$
for $\displaystyle 0<x_1<x_2<\frac{\pi}{2}$
and $0<t<1$.
Since $\displaystyle 0<(k-2)\beta'<k\beta'<(k+1)\beta'<\frac{\pi}{2}$
and $\displaystyle \frac{1}{3}(k-2)+(1-\frac{1}{3})(k+1)=k$,
we have
$\displaystyle \sin(k\beta')>\frac{1}{3}\sin((k-2)\beta')
+\frac{2}{3}\sin((k+1)\beta')$.
\begin{eqnarray*}
g_{k-1}(\beta)+g_{k}(\beta)
&=& 2\sin(k\beta)+\sin((k-1)\beta)-\sin((k-2)\beta)\\
&~&+2\sin((k+1)\beta)+\sin(k\beta)-\sin((k-1)\beta)\\
&=& -\sin((k-2)\beta) +3\sin(k\beta) + 2 \sin((k+1)\beta)
\end{eqnarray*}
If $k$ is even then
$g_k(\beta)$
$= 2 \sin((k+1)\beta)+\sin(k\beta)-\sin((k-1)\beta)$ 
$= 2 \sin((k+1)\beta')-\sin(k\beta')-\sin((k-1)\beta)>0$.

\begin{eqnarray*}
g_{k-1}(\beta)+g_{k}(\beta)
&=& -\sin((k-2)\beta) +3\sin(k\beta) + 2 \sin((k+1)\beta)\\
&=& \sin((k-2)\beta) -3\sin(k\beta) + 2 \sin((k+1)\beta)\\
&=&
 3(\frac{1}{3}\sin((k-2)\beta')+\frac{2}{3}\sin((k+1)\beta')\\
&~&-\sin(k\beta'))
<0.
\end{eqnarray*}
Since $g_k(\beta)>0$,
$\displaystyle \frac{g_{k-1}(\beta)}{g_k(\beta)}+1<0$.

If $k$ is odd then
$g_k(\beta)$
$= 2 \sin((k+1)\beta)+\sin(k\beta)-\sin((k-1)\beta)$ 
$= -2 \sin((k+1)\beta')+\sin(k\beta')+\sin((k-1)\beta)>0$.

\begin{eqnarray*}
g_{k-1}(\beta)+g_{k}(\beta)
&=& -\sin((k-2)\beta) +3\sin(k\beta) + 2 \sin((k+1)\beta)\\
&=& -\sin((k-2)\beta) +3\sin(k\beta) - 2 \sin((k+1)\beta)\\
&=&
 3(\sin(k\beta'))-\frac{1}{3}\sin((k-2)\beta')\\
&~&-\frac{2}{3}\sin((k+1)\beta')
>0.
\end{eqnarray*}
Since $g_k(\beta)<0$,
$\displaystyle \frac{g_{k-1}(\beta)}{g_k(\beta)}+1<0$.

\noindent
{3.\ } Let $\alpha'=\pi-\alpha$.
Then $\displaystyle 0 < \alpha' < \frac{\pi}{4k-1}$.
We note that if $n$ is even then $\cos(n\alpha)=\cos(n\alpha')$
and if $n$ is odd then $\cos(n\alpha)=-\cos(n\alpha')$.

{\small{
\begin{eqnarray*}
g_k(\beta)\cos(n\alpha)-g_{k-1}(\beta)\cos((n-1)\alpha)
&=&g_k(\beta)\cos((n-1)\alpha)\\
&~&(\frac{\cos(n\alpha)}{\cos((n-1)\alpha)}- \frac{g_{k-1}(\beta)}{g_k(\beta)}).
\end{eqnarray*}}}

Since $n=2k$,
we have $\cos(n\alpha)=\cos(n\alpha')$
and $\cos((n-1)\alpha)=-\cos((n-1)\alpha')<0$.
Since $0<(n-1)\alpha'<n\alpha'<\pi$
and $\displaystyle (n-1)\alpha' < \frac{(n-1)\pi}{2n-1}< \frac{\pi}{2}$,
we have
$\cos((n-1)\alpha')>\cos(n\alpha')$,
$-\cos((n-1)\alpha)>\cos(n\alpha)$
and
$\displaystyle \frac{\cos(n\alpha)}{\cos((n-1)\alpha)} >-1$.
Since $\displaystyle \frac{g_{k-1}(\beta)}{g_k(\beta)}<-1$,
we have
$\displaystyle \frac{\cos(n\alpha)}{\cos((n-1)\alpha)} -
\frac{g_{k-1}(\beta)}{g_k(\beta)} > 0$.

If $k$ is even then $g_k(\beta)>0$ and $\cos((n-1)\alpha)<0$,
then 
$$
g_k(\beta)\cos(n\alpha)-g_{k-1}(\beta)\cos((n-1)\beta)<0.
$$
If $k$ is odd then $g_k(\beta)<0$ and $\cos((n-1)\alpha)<0$,
then 
$$
g_k(\beta)\cos(n\alpha)-g_{k-1}(\beta)\cos((n-1)\beta)>0.
$$
\end{proof}

\begin{proposition}
\label{propO3}
If $k \ge 3$ and $\lambda_2({\mathcal L}(P_{2k,k}))$
the second eigenvalue of ${\mathcal L}(P_{2k,k})$ then
$$
1 - \cos \frac{\pi}{4k-1} \le \lambda_2(\mathcal{L}(P_{2k,k})).
$$
\end{proposition}
\begin{proof}
Let $0<\alpha<\pi$.
If $\displaystyle 1+\cos\alpha < 1-\cos\frac{\pi}{4k-1}$
then $\displaystyle \cos\alpha<-\cos\frac{\pi}{4k-1}$
and $\displaystyle \frac{\pi}{4k-1}<\alpha$.
By Theorem~\ref{prop:pnk} and Lemma~\ref{lemma:gk},
we have 
$p_{n,k}(\lambda) \not=0$
if $\displaystyle \frac{\pi}{4k-1}<\alpha<\pi$
and $\lambda=1+\cos\alpha$.
This shows that
$$
1 - \cos \frac{\pi}{4k-1} \le \lambda_2(\mathcal{L}(P_{2k,k})).
$$
\end{proof}

\begin{example}
If $k=3$ then $n=6$ and $\displaystyle
 1-\cos\frac{\pi}{4k-1}=0.0405\cdots$.
If $k=4$ then $n=8$ and $\displaystyle
 1-\cos\frac{\pi}{4k-1}=0.02185\cdots$.
The blue curve in the Figure~\ref{fig:fig22} is $y=p_{6,3}(x)$ and the red curve is $y=p_{8,4}(x)$.
\begin{figure}
\begin{center}
\includegraphics[scale=0.8]{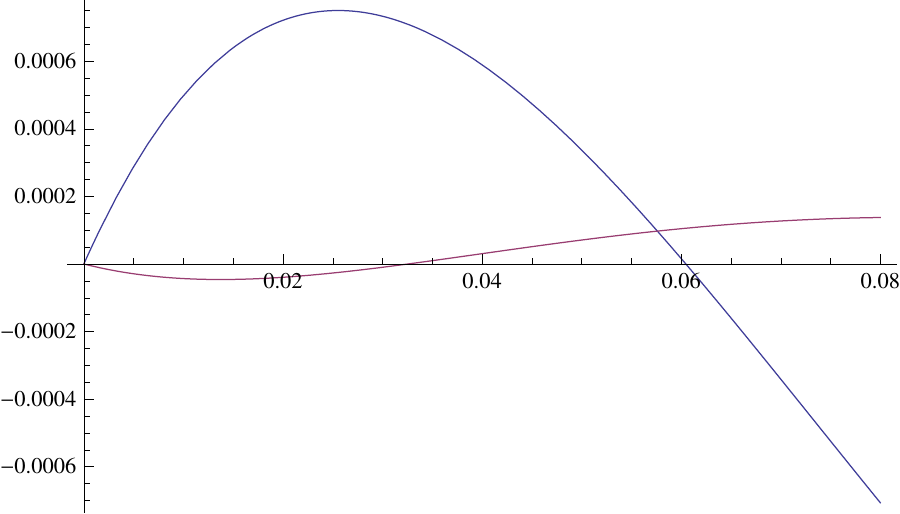}
\end{center}
\caption{Eigenvalues of $P_{6,3}$ and $P_{8,4}$}
\label{fig:fig22}
\end{figure}
\end{example}


\subsection{Eigenvalues of $\mathcal{L}(R_{n,k})$}
\begin{example}
The adjacency matrix and the normalized Laplacian matrix of a graph $R_{5,5}$.\\
 $
A(R_{5,5})=
\left(
\begin{smallmatrix}
 0 & 1 & 0 & 0 & 0 & 0 & 0 & 0 & 0 & 0 & 0 & 0 & 0 & 0 & 0 & 0 & 0 & 0 & 0 & 0 \\
 1 & 0 & 1 & 0 & 0 & 0 & 0 & 0 & 0 & 0 & 0 & 0 & 0 & 0 & 0 & 0 & 0 & 0 & 0 & 0 \\
 0 & 1 & 0 & 1 & 0 & 0 & 0 & 0 & 0 & 0 & 0 & 0 & 0 & 0 & 0 & 0 & 0 & 0 & 0 & 0 \\
 0 & 0 & 1 & 0 & 1 & 0 & 0 & 0 & 0 & 0 & 0 & 0 & 0 & 0 & 0 & 0 & 0 & 0 & 0 & 0 \\
 0 & 0 & 0 & 1 & 0 & 1 & 0 & 0 & 0 & 0 & 0 & 0 & 0 & 0 & 0 & 0 & 0 & 0 & 0 & 0 \\
 0 & 0 & 0 & 0 & 1 & 0 & 1 & 0 & 0 & 0 & 0 & 0 & 0 & 0 & 0 & 1 & 0 & 0 & 0 & 0 \\
 0 & 0 & 0 & 0 & 0 & 1 & 0 & 1 & 0 & 0 & 0 & 0 & 0 & 0 & 0 & 0 & 1 & 0 & 0 & 0 \\
 0 & 0 & 0 & 0 & 0 & 0 & 1 & 0 & 1 & 0 & 0 & 0 & 0 & 0 & 0 & 0 & 0 & 1 & 0 & 0 \\
 0 & 0 & 0 & 0 & 0 & 0 & 0 & 1 & 0 & 1 & 0 & 0 & 0 & 0 & 0 & 0 & 0 & 0 & 1 & 0 \\
 0 & 0 & 0 & 0 & 0 & 0 & 0 & 0 & 1 & 0 & 0 & 0 & 0 & 0 & 0 & 0 & 0 & 0 & 0 & 1 \\
 0 & 0 & 0 & 0 & 0 & 0 & 0 & 0 & 0 & 0 & 0 & 1 & 0 & 0 & 0 & 0 & 0 & 0 & 0 & 0 \\
 0 & 0 & 0 & 0 & 0 & 0 & 0 & 0 & 0 & 0 & 1 & 0 & 1 & 0 & 0 & 0 & 0 & 0 & 0 & 0 \\
 0 & 0 & 0 & 0 & 0 & 0 & 0 & 0 & 0 & 0 & 0 & 1 & 0 & 1 & 0 & 0 & 0 & 0 & 0 & 0 \\
 0 & 0 & 0 & 0 & 0 & 0 & 0 & 0 & 0 & 0 & 0 & 0 & 1 & 0 & 1 & 0 & 0 & 0 & 0 & 0 \\
 0 & 0 & 0 & 0 & 0 & 0 & 0 & 0 & 0 & 0 & 0 & 0 & 0 & 1 & 0 & 1 & 0 & 0 & 0 & 0 \\
 0 & 0 & 0 & 0 & 0 & 1 & 0 & 0 & 0 & 0 & 0 & 0 & 0 & 0 & 1 & 0 & 1 & 0 & 0 & 0 \\
 0 & 0 & 0 & 0 & 0 & 0 & 1 & 0 & 0 & 0 & 0 & 0 & 0 & 0 & 0 & 1 & 0 & 1 & 0 & 0 \\
 0 & 0 & 0 & 0 & 0 & 0 & 0 & 1 & 0 & 0 & 0 & 0 & 0 & 0 & 0 & 0 & 1 & 0 & 1 & 0 \\
 0 & 0 & 0 & 0 & 0 & 0 & 0 & 0 & 1 & 0 & 0 & 0 & 0 & 0 & 0 & 0 & 0 & 1 & 0 & 1 \\
 0 & 0 & 0 & 0 & 0 & 0 & 0 & 0 & 0 & 1 & 0 & 0 & 0 & 0 & 0 & 0 & 0 & 0 & 1 & 0
\end{smallmatrix}
\right)$\\

$\mathcal{L}(R_{5,5})$ can be written as \\
{\tiny
$\left(
\begin{smallmatrix}
 1 & -\frac{1}{\sqrt{2}} & 0 & 0 & 0 & 0 & 0 & 0 & 0 & 0 & 0 & 0 & 0 & 0 & 0 & 0 & 0 & 0 & 0 & 0 \\
 -\frac{1}{\sqrt{2}} & 1 & -\frac{1}{2} & 0 & 0 & 0 & 0 & 0 & 0 & 0 & 0 & 0 & 0 & 0 & 0 & 0 & 0 & 0 & 0 & 0 \\
 0 & -\frac{1}{2} & 1 & -\frac{1}{2} & 0 & 0 & 0 & 0 & 0 & 0 & 0 & 0 & 0 & 0 & 0 & 0 & 0 & 0 & 0 & 0 \\
 0 & 0 & -\frac{1}{2} & 1 & -\frac{1}{2} & 0 & 0 & 0 & 0 & 0 & 0 & 0 & 0 & 0 & 0 & 0 & 0 & 0 & 0 & 0 \\
 0 & 0 & 0 & -\frac{1}{2} & 1 & -\frac{1}{\sqrt{6}} & 0 & 0 & 0 & 0 & 0 & 0 & 0 & 0 & 0 & 0 & 0 & 0 & 0 & 0 \\
 0 & 0 & 0 & 0 & -\frac{1}{\sqrt{6}} & 1 & -\frac{1}{3} & 0 & 0 & 0 & 0 & 0 & 0 & 0 & 0 & -\frac{1}{3} & 0 & 0 & 0 & 0 \\
 0 & 0 & 0 & 0 & 0 & -\frac{1}{3} & 1 & -\frac{1}{3} & 0 & 0 & 0 & 0 & 0 & 0 & 0 & 0 & -\frac{1}{3} & 0 & 0 & 0 \\
 0 & 0 & 0 & 0 & 0 & 0 & -\frac{1}{3} & 1 & -\frac{1}{3} & 0 & 0 & 0 & 0 & 0 & 0 & 0 & 0 & -\frac{1}{3} & 0 & 0 \\
 0 & 0 & 0 & 0 & 0 & 0 & 0 & -\frac{1}{3} & 1 & -\frac{1}{\sqrt{6}} & 0 & 0 & 0 & 0 & 0 & 0 & 0 & 0 & -\frac{1}{3} & 0 \\
 0 & 0 & 0 & 0 & 0 & 0 & 0 & 0 & -\frac{1}{\sqrt{6}} & 1 & 0 & 0 & 0 & 0 & 0 & 0 & 0 & 0 & 0 & -\frac{1}{2} \\
 0 & 0 & 0 & 0 & 0 & 0 & 0 & 0 & 0 & 0 & 1 & -\frac{1}{\sqrt{2}} & 0 & 0 & 0 & 0 & 0 & 0 & 0 & 0 \\
 0 & 0 & 0 & 0 & 0 & 0 & 0 & 0 & 0 & 0 & -\frac{1}{\sqrt{2}} & 1 & -\frac{1}{2} & 0 & 0 & 0 & 0 & 0 & 0 & 0 \\
 0 & 0 & 0 & 0 & 0 & 0 & 0 & 0 & 0 & 0 & 0 & -\frac{1}{2} & 1 & -\frac{1}{2} & 0 & 0 & 0 & 0 & 0 & 0 \\
 0 & 0 & 0 & 0 & 0 & 0 & 0 & 0 & 0 & 0 & 0 & 0 & -\frac{1}{2} & 1 & -\frac{1}{2} & 0 & 0 & 0 & 0 & 0 \\
 0 & 0 & 0 & 0 & 0 & 0 & 0 & 0 & 0 & 0 & 0 & 0 & 0 & -\frac{1}{2} & 1 & -\frac{1}{\sqrt{6}} & 0 & 0 & 0 & 0 \\
 0 & 0 & 0 & 0 & 0 & -\frac{1}{3} & 0 & 0 & 0 & 0 & 0 & 0 & 0 & 0 & -\frac{1}{\sqrt{6}} & 1 & -\frac{1}{3} & 0 & 0 & 0 \\
 0 & 0 & 0 & 0 & 0 & 0 & -\frac{1}{3} & 0 & 0 & 0 & 0 & 0 & 0 & 0 & 0 & -\frac{1}{3} & 1 & -\frac{1}{3} & 0 & 0 \\
 0 & 0 & 0 & 0 & 0 & 0 & 0 & -\frac{1}{3} & 0 & 0 & 0 & 0 & 0 & 0 & 0 & 0 & -\frac{1}{3} & 1 & -\frac{1}{3} & 0 \\
 0 & 0 & 0 & 0 & 0 & 0 & 0 & 0 & -\frac{1}{3} & 0 & 0 & 0 & 0 & 0 & 0 & 0 & 0 & -\frac{1}{3} & 1 & -\frac{1}{\sqrt{6}} \\
 0 & 0 & 0 & 0 & 0 & 0 & 0 & 0 & 0 & -\frac{1}{2} & 0 & 0 & 0 & 0 & 0 & 0 & 0 & 0 & -\frac{1}{\sqrt{6}} & 1 \\
\end{smallmatrix}
\right)$
}
\end{example}

\begin{theorem}
Let $n \ge 3$, $k \ge 3$.
The characteristic polynomial of
${\mathcal L}(R_{n,k})$ is
$$
\vert \lambda I_{2(n+k)} - \mathcal{L}(R_{n,k}) \vert 
= p_{n,k}(\lambda) \cdot q_{n,k}(\lambda),
$$
where
\begin{eqnarray*}
p_{n,k}(\lambda) &=&
\frac{1}{2^n3^k\sin \beta}
(g_k(\beta)\cos(n\alpha)) - g_{k-1}(\beta)\cos((n-1)\alpha)), \\
q_{n,k}(\lambda) &=&
\frac{1}{2^n3^k\sin\gamma}
(h_{k}(\gamma)\cos(n\alpha) - h_{k-1}(\gamma)\cos((n-1)\alpha)),
\end{eqnarray*}
and 
$\displaystyle \lambda = 1 + \cos \alpha= \frac{2}{3}(1 + \cos \beta)=\frac{2}{3}(2 + \cos \gamma)$.
\end{theorem}
\begin{proof}
Since
$\displaystyle
|B_{n}(\lambda-1,\frac{1}{\sqrt{2}},\lambda-1,\frac{1}{2})|=
\frac{\cos (n\alpha)}{2^{n-1}}$
and \\
$\displaystyle
|C_k(\lambda-\frac{4}{3},\frac{1}{3},\lambda-\frac{3}{2},\frac{1}{\sqrt{6}})|=
\frac{h_k(\gamma)}{2\cdot 3^k \cdot \sin \gamma}$,
we have
\noindent
{\small{
\begin{eqnarray*}
&&\left\vert
\begin{array}{c|c}
B_n(\lambda-1,\frac{1}{\sqrt{2}},\lambda-1,\frac{1}{2}) &
X_{n,k} \\
\hline
X_{n,k}^t &
C_k(\lambda-\frac{4}{3},\frac{1}{3},\lambda-\frac{3}{2},\frac{1}{\sqrt{6}})
\end{array}
\right\vert 
\\&=&
-\frac{1}{4}
 |B_{n-2}(\lambda-1,\frac{1}{\sqrt{2}},\lambda-1,\frac{1}{2})|
\cdot
|C_{k}(\lambda-\frac{4}{3},\frac{1}{3},\lambda-\frac{3}{2},\frac{1}{\sqrt{6}})| \\
&&
+ (\lambda -1)
|B_{n-1}(\lambda-1,\frac{1}{\sqrt{2}},\lambda-1,\frac{1}{2})|
\cdot 
|C_{k}(\lambda-\frac{4}{3},\frac{1}{3},\lambda-\frac{3}{2},\frac{1}{\sqrt{6}})| \\
&&
-\frac{1}{6} 
|B_{n-1}(\lambda-1,\frac{1}{\sqrt{2}},\lambda-1,\frac{1}{2})|
\cdot
|C_{k-1}(\lambda-\frac{4}{3},\frac{1}{3},\lambda-\frac{3}{2},\frac{1}{\sqrt{6}})| \\
&=&
-\frac{1}{4}\cdot \frac{\cos ((n-2)\alpha)}{2^{n-3}} \cdot
\frac{h_{k}(\gamma)}{2 \cdot 3^k \cdot \sin \gamma}\\
&~&
+ \cos \alpha \cdot \frac{\cos ((n-1)\alpha)}{2^{n-2}} \cdot
\frac{h_{k}(\gamma)}{2 \cdot 3^k \cdot \sin \gamma}\\
&~&
-\frac{1}{6}  \frac{\cos ((n-1)\alpha)}{2^{n-2}} \cdot
\frac{h_{k-1}(\gamma)}{2 \cdot 3^{k-1} \cdot \sin \gamma}\\
&=&
\frac{1}{2^n\cdot 3^k \cdot \sin \gamma}
(-\cos((n-2)\alpha)h_k(\gamma)+2 \cos \alpha \cos((n-1)\alpha)\\
&~&h_k(\gamma) - \cos((n-1)\alpha)h_{k-1}(\gamma))\\
&=&
\frac{1}{2^n\cdot 3^k \cdot \sin \gamma}
(\cos(n\alpha)h_k(\gamma) - \cos((n-1)\alpha)h_{k-1}(\gamma))
\\
&=&
q_{n,k}(\lambda).
\end{eqnarray*}}}
We note that
\begin{eqnarray*}
&& -\cos((n-2)\alpha)+2\cos \alpha \cos((n-1)\alpha)\\
&=&
-\cos \alpha \cos ((n-1)\alpha) 
- \sin \alpha \sin((n-1)\alpha)\\
&&+2\cos \alpha \cos((n-1)\alpha)\\
&=&
\cos \alpha \cos ((n-1)\alpha) 
- \sin \alpha \sin((n-1)\alpha)
\\
&=&
\cos (n\alpha).
\end{eqnarray*}
So we have,
{\small
\begin{eqnarray*}
&&\left\vert \lambda I_{2(n+k)} - \mathcal{L}(R_{n,k}) \right\vert\\
&=& 
\left\vert
\begin{array}{c|c}
B_n(\lambda-1,\frac{1}{\sqrt{2}},\lambda-1,\frac{1}{2}) &
X_{n,k} \\
\hline
X_{n,k}^t &
C_k(\lambda-\frac{2}{3},\frac{1}{3},\lambda-\frac{1}{2},\frac{1}{\sqrt{6}})
\end{array}
\right\vert  \\
&& \times 
\left\vert
\begin{array}{c|c}
B_n(\lambda-1,\frac{1}{\sqrt{2}},\lambda-1,\frac{1}{2}) &
X_{n,k} \\
\hline
X_{n,k}^t &
C_k(\lambda-\frac{4}{3},\frac{1}{3},\lambda-\frac{3}{2},\frac{1}{\sqrt{6}})
\end{array}
\right\vert  \\
&=& p_{n,k}(\lambda) 
\times q_{n,k}(\lambda),
\end{eqnarray*}
}
where 
$\displaystyle \lambda = 1 + \cos \alpha
= \frac{2}{3}(1 + \cos \beta) = \frac{2}{3}(2 + \cos \gamma)$.
\end{proof}


\begin{definition}
Let $n \ge 1$.
we define two matrices
$T((a_i)_{1\le i \le n}, (b_i)_{1\le i \le n-1}, (c_i)_{2\le i \le n})$
and $F$ as follows:
\begin{eqnarray*}
&&
T((a_i)_{1\le i \le n}, (b_i)_{1\le i \le n-1}, (c_i)_{2\le i \le n}) \\
&=&
\left(
\begin{array}{ccccccc}
 a_1 & b_1 & 0 & 0 & \text{...} & 0 & 0 \\
 c_2 & a_2 & b_2 & 0 & \text{...} & 0 & 0 \\
 0 & c_3 & a_3 & b_3 & \text{...} & 0 & 0 \\
 \vdots & \vdots & \vdots & \vdots & \vdots & \vdots & \vdots \\
 0 & \text{...} & 0 & c_{n-2} & a_{n-2} & b_{n-2} & 0 \\
 0 & \text{...} & 0 & 0 & c_{n-1} & a_{n-1} & b_{n-1} \\
 0 & \text{...} & 0 & 0 & 0 & c_n & a_n
\end{array}
\right), \mbox{\ and}
\end{eqnarray*}
$$
F=(f_{ij})_{1\le i, j \le n}, \mbox{\ \ where \ \ }
	  f_{ij} = \begin{cases}
(-1)^i & (i=j), \\
0 & (\mbox{otherwise}).
		   \end{cases}
$$
\end{definition}

\begin{lemma}\label{lemma:ftf}
$$
F^{-1}\cdot T((a_i),(b_i),(c_i))\cdot F
= T((a_i),(-b_i),(-c_i)).
$$
\end{lemma}
\begin{proof}
First, we note that $F^{-1}=F$.
Each element of $b_i$ or $c_i$ is in odd row and even column
or even row and odd column.
The right multiplication of $F$ changes the sign of an odd row
and the left multiplication of $F$ changes the sing of an odd column.
The sign of $a_i$ is changed twice and the sign of $b_i$ or $c_i$
is changed once. So we have
$F^{-1}\cdot T((a_i),(b_i),(c_i))\cdot F$
$= T((a_i),(-b_i),(-c_i))$.
\end{proof}

\begin{proposition}\label{prop:evenpnk}
Let $n \ge 1$, $k \ge 2$,
$$
P = 
\left(
\begin{array}{c|c}
B_n(1,-\frac{1}{\sqrt{2}},1,-\frac{1}{2}) &
X_{n,k} \\
\hline
X_{n,k}^t &
C_k(\frac{2}{3},-\frac{1}{3},\frac{1}{2},-\frac{1}{\sqrt{6}})
\end{array}
\right) \mbox{\ \ and \ \ }
$$
$$
Q = 
\left(
\begin{array}{c|c}
B_n(1,-\frac{1}{\sqrt{2}},1,-\frac{1}{2}) &
X_{n,k} \\
\hline
X_{n,k}^t &
C_k(\frac{4}{3},-\frac{1}{3},\frac{3}{2},-\frac{1}{\sqrt{6}})
\end{array}
\right).
$$
\begin{enumerate}
\item Let $\lambda\in \Re$ and $u\in R^{n+k}$.
Then $Pu=\lambda u$ if and only if $Q(Fu)=(2-\lambda)(Fu)$.
\item
An eigenvalue $\lambda\not=0$ of $P$ is simple.
\item
An eigenvalue $\lambda\not=0$ of $Q$ is simple.
\item
Let $\lambda \in \Re$,
$u=(u_i)_{1\le i\le 2(n+k)} \in R^{2(n+k)}$
and $u_i=u_{n+k+i}$ ($1 \le i \le n+k$).
Then ${\mathcal L}(R_{n,k})u=\lambda u$
if and only if $Pu=\lambda u$,
where
$u=(u_i)_{1\le i \le n+k}$.
\item
Let $\lambda \in \Re$,
$u=(u_i)_{1\le i\le 2(n+k)} \in R^{2(n+k)}$
and $u_i=-u_{n+k+i}$ ($1 \le i \le n+k$).
Then ${\mathcal L}(R_{n,k})u=\lambda u$
if and only if $Qu=\lambda u$ where
$u=(u_i)_{1\le i \le n+k}$.
\end{enumerate}
\end{proposition}
\begin{proof}
\begin{enumerate}
\item First, 
we note that $Q=F^{-1}(2I-P)F$ by Lemma~\ref{lemma:ftf}.
So $Q$ and $2I-P$ have same eigenvalues and $Fu$ is an eigenvector
of $Q$ if and only if $u$ is an eigenvector of $P$.

\item If $\lambda$ is not simple,
we can have an eigenvector $u=(u_i)_{1\le i \le n+k}$,
where $u_1=0$.
By $Pu = \lambda u$ and $\lambda\not=0$,
we have $u=0$ and it contradict that $u$ is an eigenvector of $P$.
So we have $\lambda(\not=0)$ is simple.

\item It is similar to 2.

\item Assume ${\mathcal L}(R_{n,k})u=\lambda u$,
then we have $Pu=\lambda u$ by direct computations.
The converse is also hold.

\item It is similar to 4.
\end{enumerate}
\end{proof}

\begin{example}
Let $n=k=2$.
Then
{\small
$$
{\mathcal L}(R_{2,2})=
\left(
\begin{array}{cccccccc}
 1 & -\frac{1}{\sqrt{2}} & 0 & 0 & 0 & 0 & 0 & 0 \\
 -\frac{1}{\sqrt{2}} & 1 & -\frac{1}{\sqrt{6}} & 0 & 0 & 0 & 0 & 0 \\
 0 & -\frac{1}{\sqrt{6}} & 1 & -\frac{1}{\sqrt{6}} & 0 & 0 & -\frac{1}{3} & 0 \\
  0 & 0 & -\frac{1}{\sqrt{6}} & 1 & 0 & 0 & 0 & -\frac{1}{2} \\
 0 & 0 & 0 & 0 & 1 & -\frac{1}{\sqrt{2}} & 0 & 0 \\
 0 & 0 & 0 & 0 & -\frac{1}{\sqrt{2}} & 1 & -\frac{1}{\sqrt{6}} & 0 \\
 0 & 0 & -\frac{1}{3} & 0 & 0 & -\frac{1}{\sqrt{6}} & 1 & -\frac{1}{\sqrt{6}} \\
 0 & 0 & 0 & -\frac{1}{2} & 0 & 0 & -\frac{1}{\sqrt{6}} & 1
\end{array}
\right),
$$
}
$$
P=
\left(
\begin{array}{cccc}
 1 & -\frac{1}{\sqrt{2}} & 0 & 0 \\
 -\frac{1}{\sqrt{2}} & 1 & -\frac{1}{\sqrt{6}} & 0 \\
 0 & -\frac{1}{\sqrt{6}} & \frac{2}{3} & -\frac{1}{\sqrt{6}} \\
 0 & 0 & -\frac{1}{\sqrt{6}} & \frac{1}{2}
\end{array}
\right), \mbox{\ \ and \ \ }
$$
$$
Q=
\left(
\begin{array}{cccc}
 1 & -\frac{1}{\sqrt{2}} & 0 & 0 \\
 -\frac{1}{\sqrt{2}} & 1 & -\frac{1}{\sqrt{6}} & 0 \\
 0 & -\frac{1}{\sqrt{6}} & \frac{4}{3} & -\frac{1}{\sqrt{6}} \\
 0 & 0 & -\frac{1}{\sqrt{6}} & \frac{3}{2}
\end{array}
\right).
$$

Eigenvalues of $R_{2,2}$ are
$2.$, $1.79533$, $1.62867$, $1$, $1$,
$0.371333$, $0.204666$ and $0$.
Corresponding eigenvectors are
{\tiny
$$
\left(
\begin{array}{cccccccc}
 0.707107 & -1. & 1.22474 & -1. & -0.707107 & 1. & -1.22474 & 1. \\
 -6.90985 & 7.772 & -3.17291 & 1. & -6.90985 & 7.772 & -3.17291 & 1. \\
 -0.868326 & 0.772002 & 0.315168 & -1. & 0.868326 & -0.772002 & -0.315168 & 1. \\
 0.707107 & 0 & -1.22474 & 1 & 0.707104 & 0 & -1.22474 & 1\\
 -0.707107 & 0  & 1.22474 & 1 & 0.707104 & 0 & -1.22474 & -1\\
  -0.868326 & -0.772002 & 0.315168 & 1. & -0.868326 & -0.772002 & 0.315168 & 1. \\
 -6.90985 & -7.772 & -3.17291 & -1. & 6.90985 & 7.772 & 3.17291 & 1. \\
 0.707107 & 1. & 1.22474 & 1. & 0.707107 & 1. & 1.22474 & 1.
\end{array}
\right).
$$
}
Eigenvalues of $P$ are
$1.79533$, $1$, $0.371333$, and $0$.
Corresponding eigenvectors are
$$
\left(
\begin{array}{cccc}
 -6.90985 & 7.772 & -3.17291 & 1. \\
 0.707107 & 0. & -1.22474 & 1. \\
 -0.868326 & -0.772002 & 0.315168 & 1. \\
 0.707107 & 1. & 1.22474 & 1.
\end{array}
\right).
$$
Eigenvalues of $Q$ are
$2$, $1.62867$, $1$, and $0.204666$.
Corresponding eigenvectors are
$$
\left(
\begin{array}{cccc}
 -0.707107 & 1. & -1.22474 & 1. \\
 0.868326 & -0.772002 & -0.315168 & 1. \\
 -0.707107 & 0. & 1.22474 & 1. \\
 6.90985 & 7.772 & 3.17291 & 1.
\end{array}
\right).
$$
Each eigenvector of $P$ is corresponding to an even eigenvector of
${\mathcal L}(R_{2,2})$ and $Q$ an odd eigenvector of ${\mathcal
 L}(R_{2,2})$.
Even though eigenvalues of $P$ and $Q$ are simple,
an eigenvalue $1$ of ${\mathcal L}(R_{2,2})$ is not simple.

\end{example}


\section{counter examples for $Mcut(G) \neq Lcut(G)$}
This section present counter example graphs, on which spectral methods and minimum normalized cut produce different clusters.


\subsection{$Mcut(G)$ and $Lcut(G)$}

\begin{definition}[$Lcut(G)$]
Let $G=(V,E)$ be a connected graph,
$\lambda_2$ the second smallest eigenvalue of $\mathcal{L}(G)$,
$U_2=((U_2)_i)$ $(1 \le i \le |V|)$ a second eigenvector of
$\mathcal{L}(G)$ with $\lambda_2$.
We assume that $\lambda_2$ is simple.
Then $Lcut(G)$ is defined as
 $\displaystyle Lcut(G)= Ncut(V^+(U_2)\cup V^0(U_2), V^-(U_2))$.
\end{definition}

\begin{example}
Figure~\ref{fig:lmcut} shows some examples, 
where $Mcut(G)=Lcut(G)$ and $Mcut(G) \neq Lcut(G)$.
\begin{figure}[htb]
\begin{center}
\subfigure[$Lcut(G)=Mcut(G)$]{\label{fig:lmcut-a}\includegraphics[scale=0.4]{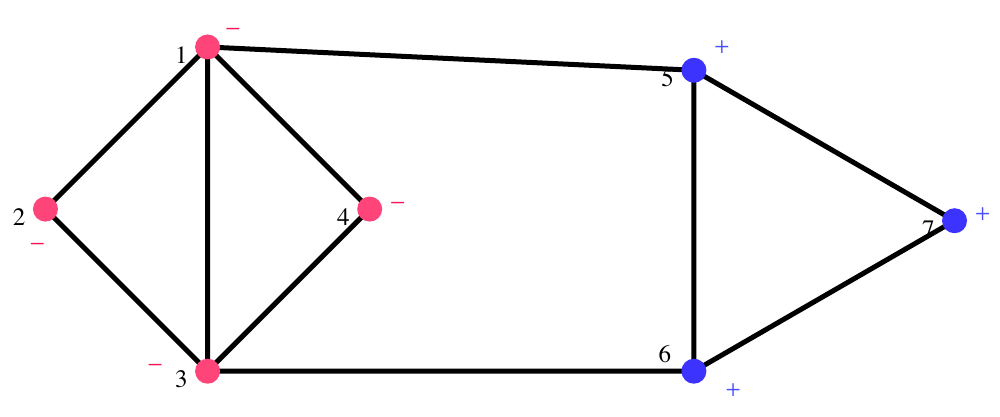}}\subfigure[$Lcut(R_{4,7})=Mcut(R_{4,7})$]{\label{fig:lmcut-b}\includegraphics[scale=0.4]{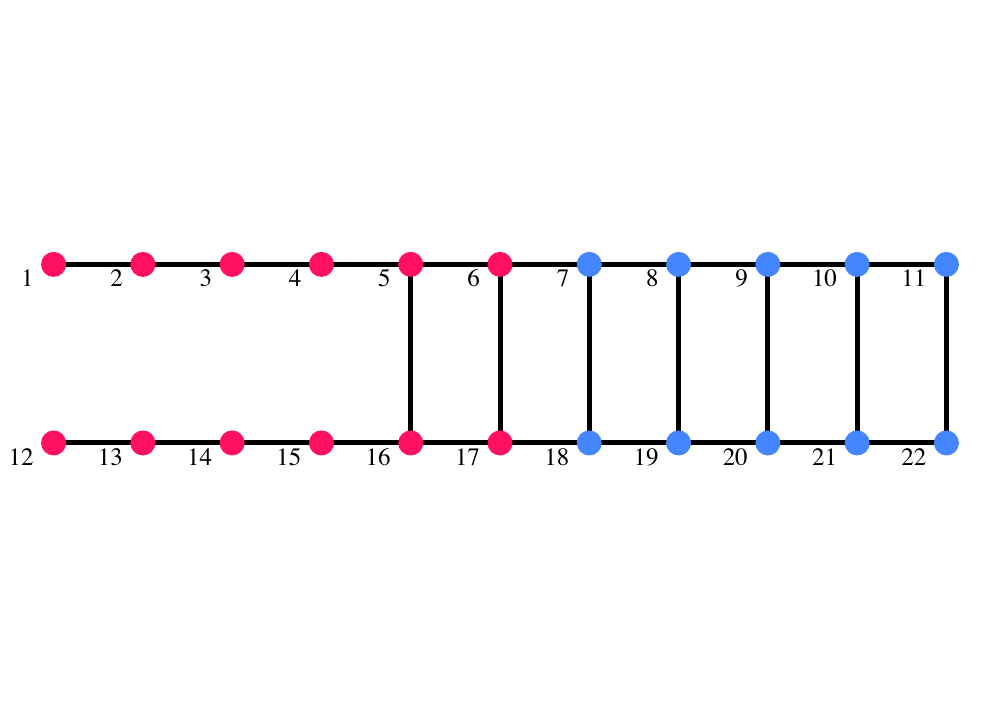}}
\subfigure[$Mcut(R_{6,4})$]{\label{fig:lmcut-c}\includegraphics[scale=0.4]{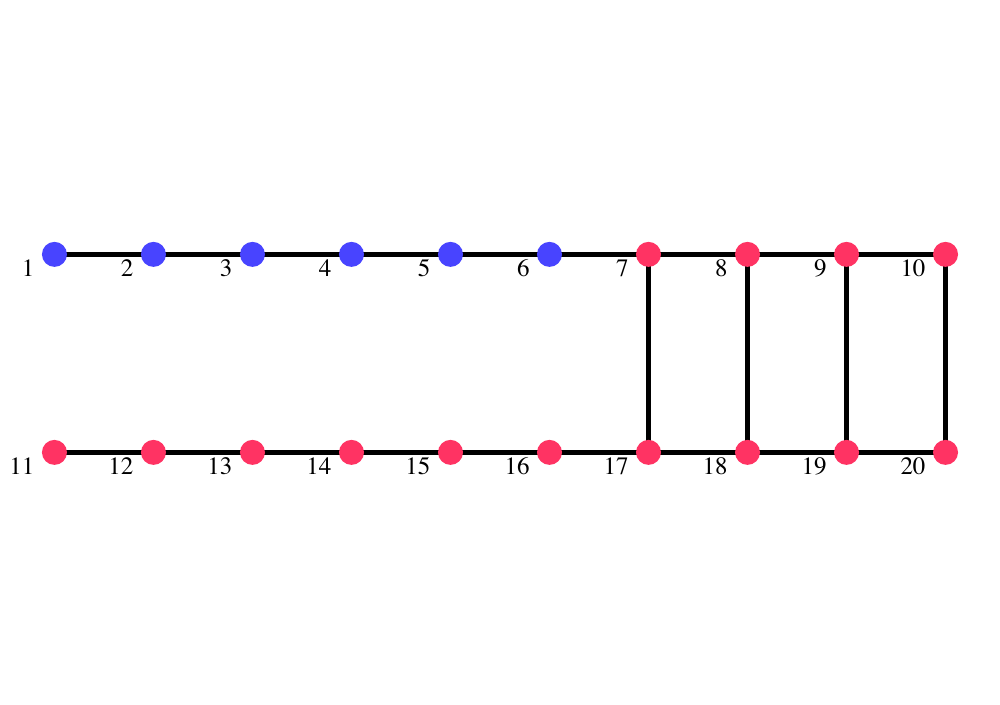}}\subfigure[$Lcut(R_{6,4})$]{\label{fig:lmcut-d}\includegraphics[scale=0.4]{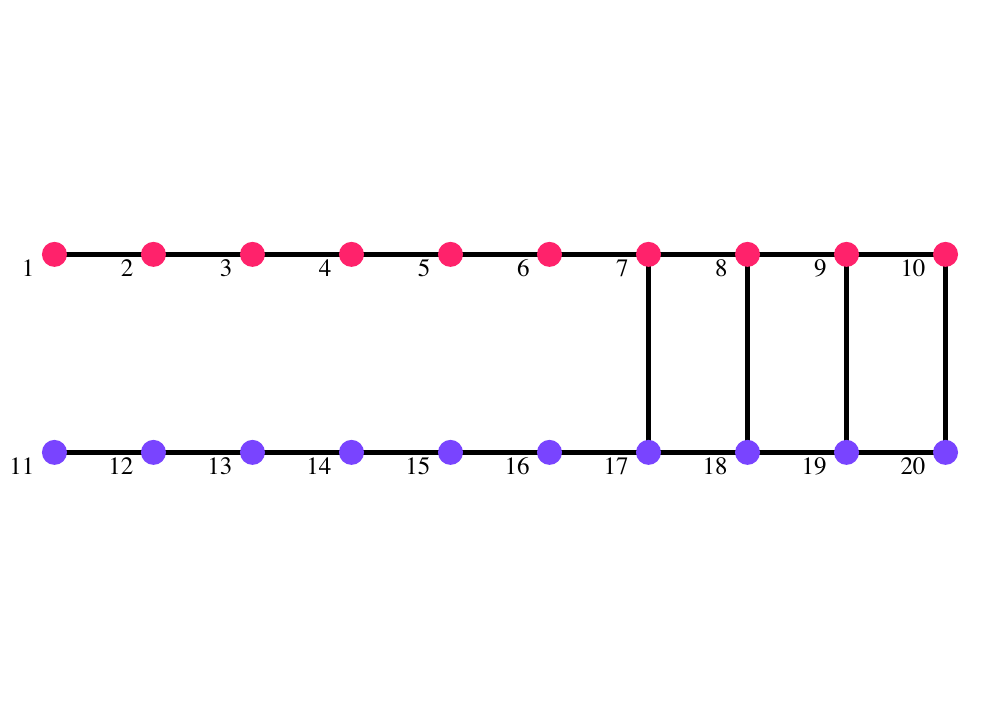}}
\caption{$Mcut(G)$ and $Lcut(G)$.}
\label{fig:lmcut}
\end{center}
\end{figure}
\end{example}

\begin{proposition}[\cite{luxburg:2007}]
\label{prop25}
Let $G=(V,E,w)$ be a weighted graph,
$W$ the weighted adjacency matrix of $G$, 
$L$ the weighted difference Laplacian $L(G)$ of $G$,
and $A$ a subset of $V$.
If vector $\displaystyle y=(y_1,\ldots,y_n)^T \in \Re^n$ is defined as 
\[ y=
\left \{ \begin{array}{ll} \sqrt{\frac{vol(V \setminus A)}{vol(A)}} & \mbox{if $v_i \in A$,}\\
 -\sqrt{\frac{vol(A)}{vol(V \setminus A)}} & \mbox{if $v_i \in V \setminus A$,}
\end{array}
\right. \] then
\begin{enumerate}
\item $y^TLy= vol(V)\cdot Ncut(A, V \setminus A)$,
\item $y^TDy=vol(V)$ and 
\item $(Dy)^T\vec{1}=0$.
\end{enumerate}
\end{proposition}

\begin{proof}
\begin{enumerate} 
 \item {\small {\begin{eqnarray*}
y^TLy&=&y^TDy-y^TWy\\
&=&\sum_{i=1}^n d_iy_i^2 -\sum_{i,j} y_iw_{ij}y_j\\
&=&\frac{1}{2} \left( \sum_{i=1}^n d_iy_i^2 -2 \sum_{i,j} y_iy_jw_{ij} +\sum_{j=1}^n d_jy_j^2 \right )\\
&=& \frac{1}{2} \sum_{i,j=1}^n w_{ij}(y_i-y_j)^2
\end{eqnarray*}}}
This can be further reduced to,\\
{\small{ 
\begin{eqnarray*}
& =&  \frac{1}{2}\sum_{i\in A, j\in (V \setminus A)} w_{ij}\left( \sqrt{\frac{vol(V \setminus A)}{vol(A)}}+
\sqrt{\frac{vol(A)}{vol(V \setminus A)}}\right)^2+\\
&&\frac{1}{2}\sum_{i\in (V \setminus A), j\in A} w_{ij}\left( -\sqrt{\frac{vol(A)}{vol(V \setminus A)}}-
\sqrt{\frac{vol(V \setminus A)}{vol(A)}} \right)^2\\
&=& cut(A,V \setminus A) \left(\frac{vol(A)}{vol(V \setminus A)}+\frac{vol(V \setminus A)}{vol(A)}+2 \right)\\
&=&cut(A,V \setminus A)\left( \frac{vol(A)+vol(V \setminus A)}{vol(V \setminus A)}+
\frac{vol(A)+vol(V \setminus A)}{vol(A)}\right)\\
&=& vol(V).Ncut(A, V \setminus A).
\end{eqnarray*}}}

\item {\small{\begin{eqnarray*}
y^TDy &=& \sum_{i=1}^n d_{i}y_i^2 = \sum_{i\in A} d_iy_i^2+ \sum_{i\in V \setminus A} d_iy_i^2\\
&=& \sum_{i\in A}d_i \left( \frac{vol(V \setminus A)}{vol(A)}\right)+
 \sum_{i \in (V \setminus A)}d_i \left(\frac{vol(A)}{vol(V \setminus A)} \right)\\
&=& vol(A)\frac{vol(V \setminus A)}{vol(A)}+vol(V \setminus A)\frac{vol(A)}{vol(V \setminus A)}\\
&=& vol(V).
\end{eqnarray*}}}\hfill\qed
\item {\small{ \begin{eqnarray*}
(Dy)^T\vec{1}&=& \sum_{i=1}^n d_i y_i \\
&=& \sum_{i \in A} d_i \sqrt{\frac{vol(V \setminus A)}{vol(A)}}-
\sum_{i \in V \setminus A} d_i \sqrt{\frac{vol(A)}{vol(V \setminus A)}}\\
&=& vol(A) \sqrt{\frac{vol(V \setminus A)}{vol(A)}}-
vol(V \setminus A)\sqrt{\frac{vol(A)}{vol(V \setminus A)}}\\
&=& 0.
\end{eqnarray*} }}
\end{enumerate}
\end{proof}

\noindent
By Proposition~\ref{prop25}, we have
{\small
\begin{eqnarray*}
Ncut(A,V\setminus A)&=& \frac{y^T L y}{y^T D y}\\
&=&
\frac{y^T (D-W) y}{y^T D y} \\
&=&
\frac{(D^{1/2}y)^T(I-D^{-1/2}WD^{-1/2})(D^{1/2}y)}
{(D^{1/2}y)^T (D^{1/2}y)}\\
&=& \frac{z^T \mathcal{L}(G) z}{z^T z},
\end{eqnarray*}
}
\noindent
where $z=D^{1/2}y$ and $\mathcal{L}(G)=I-D^{-1/2}WD^{-1/2}$.
The least eigenvalue of $\mathcal{L}(G)$
is $0$ and an eigenvector is $D^{1/2}\vec{1}$.
Let $\lambda_2$ be the second eigenvalue of $\mathcal{L}(G)$.
It is well known
$$
\lambda_2 = \min\{ \frac{z^T \mathcal{L}(G) z}{z^T z}\ |\ 
z \in \Re^n, \ z \perp D^{1/2}\vec{1} \}.
$$
If $z$ is a second eigenvector, then
$\displaystyle \lambda_2=\frac{z^T \mathcal{L}(G) z}{z^T z}$
and $z \perp D^{1/2}\vec{1}$.
These results guide to consider relations
between a set $A$ attaining $Mcut(G)=Ncut(A,V\setminus A)$
and a set $V^+(U)$,
where $U$ is a second eigenvector
of $\mathcal{L}(G)$.
The set $V^+(U)$ is a good approximation of $A$.



\subsection{The graph $R_{n,k}$}
In this section,
we review the formulae of
$Mcut(R_{n,k})$ and conditions in Theorem~\ref{propmcutg},
consider some properties of subsets $A$ of $V(R_{n,k})$,
which attains
$Lcut(R_{n,k})=Ncut(A,V\setminus A)$,
and assign a condition of $n$ and $k$ to cause
$Mcut(R_{n,k})\not=Lcut(R_{n,k})$.

Let $R_{n,k}=(V,E)$, $V=\{v_i\ |\ 1\le i \le 2(n+k)\}$,
where
$$
v_i= \begin{cases} x_i & (1 \le i \le n+k), \\
y_{i-(n+k)} & (n+k+1 \le i \le 2(n+k)). \end{cases}
$$
We review subsets $A_1$, $A_2$ and $A_4(\alpha)$ defined in the proof
of Theorem~\ref{propmcutg}. That is
\begin{eqnarray*}
A_1 &=& \{v_i \ |\ 1 \le i \le n+k\}, \\
A_2 &=& \{v_i \ |\ 1 \le i \le n \}, \mbox{\ and} \\
A_4(\alpha) &=&\{v_i, v_{i+n+k} \ |\ 1 \le i \le n+ \alpha \} \\
&& (1 \le \alpha <k). \\
\end{eqnarray*}
For a vector $U=(u_1,u_2,\ldots,u_{2(n+k)}) \in \Re^{2(n+k)}$,
we write $\bar{U}=(u_1,u_2,\ldots,u_{n+k}) \in \Re^{n+k}$.
For a vector $\bar{U}=(u_1,u_2,\ldots,u_{(n+k)}) \in \Re^{n+k}$,
we write $(\bar{U},\bar{U}) \in \Re^{2(n+k)}$ as a
vector $U=(u_1,u_2,\ldots,u_{2(n+k)}) \in \Re^{2(n+k)}$
such that $u_{i+(n+k)}=u_i \ (1\le i \le n+k)$.
In this section,
we consider an automorphism $\phi$,
where $\phi(v_i)=v_{i+n+k}$ to consider even and odd vectors.

\begin{proposition}
\label{prop34}
If $U=(u_1,u_2,\ldots,u_{2(n+k)})$ is
an eigenvector of $\mathcal{L}(R_{n,k})$ with an eigenvalue $\lambda$,
then
$\bar{U}$ is an eigenvector of $\mathcal{L}(P_{n,k})$
with an eigenvalue $\lambda$.
Conversely, if $\bar{U}=(u_1,u_2,\ldots,u_{n+k})$ is an eigenvector
of $\mathcal{L}(P_{n,k})$ with an eigenvalue $\lambda$,
then $U=(\bar{U},\bar{U})$ is an eigenvector of $\mathcal{L}(R_{n,k})$.
\end{proposition}
\begin{proof}
If $U$ is an even vector then we can write $U=(\bar{U},\bar{U})$.
The matrix $\mathcal{L}(R_{n,k})$ can be written as 
\[ \mathcal{L}(R_{n,k})=\left( \begin{array}{cc}
\mathcal{L}_1 & C \\
C^T & \mathcal{L}_1 
\end{array} \right), \]
 where $\mathcal{L}_1$ is the $(n+k) \times (n+k)$ principal sub matrix of $\displaystyle \mathcal{L}(R_{n,k})$ and $C=(c_{ij})$ is the $(n+k) \times (n+k)$ matrix such that \[ c_{ij}= \left\{\begin{array}{cc}
-\frac{1}{d_i} &\mbox{ if $n+1 \le i \le n+k $ and $i=j$,}\\
0 &\mbox{otherwise.}
\end{array} \right. \]
We notice that $\displaystyle \mathcal{L}_1+C=\mathcal{L}(P_{n,k})$.
If $\lambda$ is an eigenvalue of $\mathcal{L}(R_{n,k})$ then 
$\mathcal{L}(R_{n,k}) U=\lambda U $ can be written as,
\[\left( \begin{array}{cc} \mathcal{L}_1 & C \\ C^T & \mathcal{L}_1 \end{array} \right) \left( \begin{array}{c} 
\bar{U} \\
\bar{U} \end{array} \right) = \lambda \left( \begin{array}{c} \bar{U} \\
\bar{U} \end{array}\right) \]
This gives
\begin{eqnarray*}
\mathcal{L}_1\bar{U} + C \bar{U} &=& \lambda \bar{U}.
\end{eqnarray*}
This can be written as, 
$\displaystyle (\mathcal{L}_1+C )\bar{U} =\mathcal{L}(P_{n,k})\bar{U}=
\lambda \bar{U}$.
Therefore $\lambda$ is an eigenvalue of $\mathcal{L}(P_{n,k})$ and $\bar{U}$ is an eigenvector.
Thus if $U$ is an even vector of $\displaystyle \mathcal{L}(R_{n,k})$ with eigenvalue $\lambda$,
then $\bar{U}$ is an eigenvector of $\displaystyle \mathcal{L}(P_{n,k})$ with the same eigenvalue.
The converse also holds.\end{proof}

\begin{proposition}
\label{lemma16}
Let $U=(u_1,u_2,\ldots,u_{(n+k)})$ be an eigenvector of $\mathcal{L}(P_{n,k})$ with a second smallest eigenvalue $\lambda_2$. 
Then there exists some $\alpha \in \mathbf{Z^+}$ such that $\displaystyle u_i \geq 0 \ (1 \leq i \leq \alpha)$ and $\displaystyle u_i <0 \ (\alpha +1 \leq i \leq n+k)$ or $\displaystyle u_i <0 \ (1 \leq i \leq \alpha)$ and $\displaystyle u_i \ge 0 \ (\alpha +1 \leq i \leq n+k)$.
\end{proposition}
\begin{proof}
If $U=(u_1,u_2,\ldots,u_{n+k})$ is the second eigenvector of $\mathcal{L}(P_{n,k})$,
then $\displaystyle U \perp  D^{1/2}\vec{1}$.
Then by Lemma~\ref{lemma4},
$V^+(U) \neq \emptyset $ and $V^-(U) \neq \emptyset $.
Since $\lambda_2$ is simple,
induced subgraphs by $V^+(U)$, $V^-(U)$, $V^+(U)\cup V^0(U)$
and $V^-(U)\cup V^0(U)$ are connected
by the nodal domain theorem \cite{Davis:2001}.
Thus there exists some $\alpha \in \mathbf{Z^+}$ as given in the
proposition.
\end{proof}

\begin{corollary}
If $\bar{U}=(u_1,u_2,\ldots,u_{n+k})$ is
a first eigenvector of $\mathcal{L}(P_{n,k})$,
then
$U=(\bar{U},\bar{U})$ is a first eigenvector of $\mathcal{L}(R_{n,k})$.
\end{corollary}
\hfill\qed
 
\begin{proposition}
\label{propeven}
Let $\lambda_2$ be
the second smallest eigenvalue of $\mathcal{L}(R_{n,k})$,
$\lambda'_2$ the second smallest eigenvalue of $\mathcal{L}(P_{n,k})$,
and $U=(u_1,u_2,\ldots,u_{2(n+k)})$ 
an eigenvector of $\mathcal{L}(R_{n,k})$ with $\lambda_2$.
If $U$ is an even vector then $\lambda_2=\lambda'_2$.
That is $\bar{U}=(u_1,u_2,\ldots,u_{n+k})$ is
an second eigenvector of $\mathcal{L}(P_{n,k})$ with $\lambda'_2$.
\end{proposition}
\begin{proof}
Since $U$ is an even vector,
$\bar{U}$ is an eigenvector of $\mathcal{L}(P_{n,k})$ with $\lambda_2$.
So we have $\lambda'_2 \le \lambda_2$.
We note $\displaystyle U \perp D^{1/2}(R_{n,k})\vec{1}$
and $\bar{U} \perp D^{1/2}(P_{n,k})\vec{1}$.

Let $U'=(u'_1,u'_2,\ldots,u'_{n+k})$ be a second eigenvector
of $\mathcal{L}(P_{n,k})$ with $\lambda'_2$.
Since $(U',U')$ is an eigenvector of $\mathcal{L}(R_{n,k})$
with $\lambda'_2$,
we have $\lambda_2 \le \lambda'_2$ and $\lambda_2=\lambda'_2$.
\end{proof}


Let $\lambda_2$ be the second eigenvalue of $R_{n,k}$,
$U$ an eigenvector of $R_{n,k}$ with $\lambda_2$.
Since $\lambda_2$ is simple,
induced subgraphs by $V^-(U)$ and $V^+(U)\cup V^0(U)$
are connected
by the nodal domain theorem \cite{Davis:2001}.
Since $U$ is an odd vector or an even vector, 
it is easy to show 
Lemma~\ref{lemma15} and Lemma~\ref{lemma17}.

\begin{lemma}
\label{lemma15}
Let $U=(u_1,\ldots,u_{2(n+k)})$ be a second eigenvector
of $\mathcal{L}(R_{n,k})$.
If $U$ is an odd vector then 
$$
Lcut(R_{n,k})=Ncut(A_1,V\setminus A_1).
$$
\end{lemma}
\hfill\qed
 
\begin{lemma}
\label{lemma17}
Let $U=(u_1,\ldots,u_{2(n+k)})$ be a second eigenvector of $\mathcal{L}(R_{n,k})$.
If $U$ is an even vector then
there exists $\alpha$ $(1 \le \alpha < k)$ such that
$$
Lcut(R_{n,k})=Ncut(A_4(\alpha),V\setminus A_4(\alpha)).
$$
\end{lemma}
\hfill\qed

\begin{proposition}
Let $G=R_{n,k}(n \geq 1, k \ge 2)$.
If $n$ and $k$ belong to the following region $R$ then 
$Mcut(G) < Lcut(G)$.

\begin{eqnarray*}
 R &=& \{(n,k) \ | \ ((k \geq 4)\wedge (2 \mid k) \wedge (3 \mid n) \wedge \\
 &~&(1-\frac{1}{\sqrt{2}}-\frac{3 k}{2}+\frac{3 k}{\sqrt{2}} \leq n) )\vee\\ 
&~& (k=2 \wedge (n\geq 2))\vee  (k=3 \wedge (n \geq 3)) \}.
\end{eqnarray*}
 \end{proposition}
 
\begin{proof}
Let $G=(V,E)$, $K_1$, $K_2$, $K_3$ and $K_4$ are formulae defined
in the Theorem~\ref{propmcutg}.
If $k \ge 2$ then $K_2<K_3<K_4<K_1$.
So if $(n,k) \in R$ then $Mcut(G)=Ncut(A_2,V\setminus A_2)$
and denoted by $c_2$ in the Theorem~\ref{propmcutg}.

Let $U=(u_1,u_2,\ldots,u_{2(n+k)})$ be an eigenvector corresponding
to the second smallest eigenvalue of $\mathcal{L}(R_{n,k})$.
If $U$ is an odd vector,
then $Lcut(G)=Ncut(A_1,V\setminus A_1)$ by Lemma~\ref{lemma15}.
So we have $Mcut(G) < Lcut(G)$ by Theorem~\ref{propmcutg}.

If $U$ is an even vector,
then $Lcut(G)=Ncut(A_4(\alpha),V\setminus A_4(\alpha))$ 
for some $\alpha$ by Lemma~\ref{lemma17}.
So we have $Mcut(G) < Lcut(G)$ by Theorem~\ref{propmcutg}.
\end{proof}

\begin{theorem}
Let $k \ge 3$,
$\lambda_2({\mathcal L}(P_{2k,k}))$,
$\lambda_2({\mathcal L}(P_{4k}))$, and
$\lambda_2({\mathcal L}(R_{2k,k}))$
the second eigenvectors of 
${\mathcal L}(P_{2k,k})$,
${\mathcal L}(P_{4k})$, and
${\mathcal L}(R_{2k,k})$, respectively.

\begin{enumerate}
\item $\lambda_2({\mathcal L}(P_{4k}))
	   < \lambda_2({\mathcal L}(P_{2k,k}))$.
\item $\lambda_2({\mathcal L}(R_{2k,k}))
	   < \lambda_2({\mathcal L}(P_{4k}))$.
\item A second eigenvector $U$ of ${\mathcal L}(R_{2k,k})$
is an odd vector.
\item The second eigenvalue of ${\mathcal L}(R_{2k,k})$ is simple.
\item $Mcut(R_{2k,k})<Lcut(R_{2k,k})$.
\end{enumerate}
\end{theorem}
\begin{proof}
\begin{enumerate}
\item Since $\lambda_2({\mathcal L}(P_{4k}))$
$\displaystyle =1-\cos\left(\frac{\pi}{4k-1}\right)$
by Proposition~\ref{prop:path},
we have $\lambda_2({\mathcal L}(P_{4k}))$
$<\lambda_2({\mathcal L}(P_{2k,k}))$
by Proposition~\ref{propO3}.

\item Let
 $A=(a_{ij})_{1\le i,j\le 4k}$ be the adjacency matrix of
 $P_{4k}$,
 $B=(b_{ij})_{1\le i,j\le 6k}$ be the adjacency matrix of
 $R_{2k,k}$,
$d=(d_i)_{1\le i \le 4k},$
where $\displaystyle d_i=\sum_{j=1}^{4k}a_{ij}$,
$e=(e_i)$ where $\displaystyle e_i=\sum_{j=1}^{6k}b_{ij}$,
and $x=(x_i)_{1\le i \le 4k}$
an eigenvector of ${\mathcal L}(P_{4k})$ corresponding to 
$\lambda_2({\mathcal L}(P_{4k}))$ with $x^T x=1$.
We note that
$\displaystyle d^{\frac{1}{2}}\vec{1} \perp x$
and $\lambda_2({\mathcal L}(P_{4k})) = x^T{\mathcal L}(P_{4k})x$
$\displaystyle =\frac{1}{2}\sum_{i=1}^{4k} \sum_{j=1}^{4k}
\left(\frac{x_i}{\sqrt{d_i}}-\frac{x_j}{\sqrt{d_j}}\right)^2a_{ij}$.
Let
$$
y_i = \begin{cases}
       x_i & (1 \le i \le 2k) \\
       0   & (2k+1 \le i \le 3k, \ 5k+1 \le i \le 6k) \\
       x_{7k-i+1} & (3k+1 \le i \le 5k)
      \end{cases}
$$
and consider a vector $y=(y_i)_{1\le i \le 6k}$.
Since $x$ is a second eigenvector of ${\mathcal L}(P_{4k})$,
we have $\displaystyle \sum_{i=1}^{6k}y_i^2=\sum_{i=1}^{4k}x_i^2=1$,
$\displaystyle
 \sum_{i=1}^{6k}\sqrt{e_i}y_i=\sum_{i=1}^{4k}\sqrt{d_i}x_i=0$,
and $x_{2k}=-x_{2k+1}\not=0$.
So we have $y^T y=1$, $\displaystyle e^{\frac{1}{2}}\vec{1}\perp y$,
and
{\small
\begin{eqnarray*}
\displaystyle \lambda_2(R_{2k,k}) &
= & \inf_{e^{\frac{1}{2}}\vec{1}\perp u} \frac{u^T{\mathcal
 L}(R_{2k,k})u}{u^Tu} \\
& \le & y^T {\mathcal L}(R_{2k,k})y \\
& =  & \frac{1}{2}\sum_{i=1}^{6k} \sum_{j=1}^{6k}
\left(\frac{y_i}{\sqrt{e_i}}-\frac{y_j}{\sqrt{e_j}}\right)^2b_{ij} \\
& = &
\frac{1}{2}\sum_{i=1}^{4k} \sum_{j=1}^{4k}
\left(\frac{x_i}{\sqrt{d_i}}-\frac{x_j}{\sqrt{d_j}}\right)^2a_{ij}
-\left(\frac{x_{2k}}{\sqrt{d_{2k}}}-\frac{x_{2k+1}}{\sqrt{d_{2k+1}}}\right)^2 \\
&&
+\left(\frac{y_{2k}}{\sqrt{e_{2k}}}-\frac{y_{2k+1}}{\sqrt{e_{2k+1}}}\right)^2
+\left(\frac{y_{5k}}{\sqrt{e_{5k}}}-\frac{y_{5k+1}}{\sqrt{e_{5k+1}}}\right)^2 \\
&= & \lambda_2\left({\mathcal L}\left(P_{4k}\right)\right)
-\left(\frac{2x_{2k}}{\sqrt{d_{2k}}}\right)^2
+\left(\frac{y_{2k}}{\sqrt{e_{2k}}}\right)^2
+\left(\frac{y_{5k}}{\sqrt{e_{5k}}}\right)^2 \\
& = &\lambda_2\left({\mathcal L}\left(P_{4k}\right)\right)
-2\left(\frac{x_{2k}}{\sqrt{d_{2k}}}\right)^2 \\
& < & \lambda _2\left({\mathcal L}\left(P_{4k}\right)\right)
\end{eqnarray*}
}
\item If a second eigenvector $U$ of ${\mathcal L}(R_{2k,k})$
corresponding to $\lambda_2({\mathcal L}(R_{2k,k}))$ is an even vector,
then $\lambda_2({\mathcal L}(R_{2k,k}))=\lambda_2({\mathcal
 L}(P_{2k,k}))$
by Proposition~\ref{prop:evenpnk}.
But it contradicts that
$\lambda_2({\mathcal L}(R_{2k,k}))<\lambda_2({\mathcal L}(P_{2k,k}))$
induced by 1. and 2.
So we have a second eigenvector $U$ of ${\mathcal L}(R_{2k,k})$, 
which is an odd vector.

\item
By 3. and Proposition~\ref{prop:evenpnk}, 3. and 4.,
$\lambda_2({\mathcal L}(R_{2k,k}))$ is simple.

\item
Since the second eigenvector of $R_{2k,k}$
is an odd vector,
$Lcut(R_{2k,k})=Ncut(A_1,V\setminus A_1)$ by Lemma~\ref{lemma15}.
Thus we have $Mcut(R_{2k,k}) < Lcut(R_{2k,k})$
by Theorem~\ref{propmcutg}.
\end{enumerate}

\end{proof}

\section{Conclusion}
We presented a survey of the known results associated with difference,
normalized, and signless Laplacian matrices.
We also stated upper and lower bounds for the difference and normalized
Laplacian matrices
using isoperimetric numbers and the Cheeger constant.
We gave a uniform proof for the eigenvalues and eigenvectors of paths and
cycles
on the basis of all three Laplacian matrices using circulant matrices, and presented
an
alternate proof for finding the eigenvalues of the adjacency matrix of cycles and
paths
using Chebyshev polynomials.
We also introduced concrete formulae for $Mcut(G)$ for some classes of graphs. 
Then,
we established characteristic polynomials for the normalized Laplacian matrices ${\mathcal L}(P_{n,k})$ and ${\mathcal L}(R_{n,k})$.
Finally,
we presented counter example graphs based on $R_{n,k}$,
where $Mcut(G)$ and $Lcut(G)$ produce different clusters.
In particular, 
we established criteria for $Mcut(G)$ and $Lcut(G)$ to have different values.
 
\section*{Acknowledgments}
We would like to specially thank Professor Hiroyuki Ochiai for his ideas pertaining to computations and comparisons
of the second eigenvalues of ${\mathcal L}(P_{n,k})$,
which gave us useful hints
to finish this study.
We are also grateful to Dr. Tetsuji Taniguchi
for his helpful comments and encouragement
during the course of this study.
This research was partially supported 
by the Global COE Program "Educational-and-Research Hub for
Mathematics-for-Industry" at Kyushu University.


\bibliographystyle{plain}

\end{document}